%% file: paper.tex
\documentclass[a4paper]{article}          
\usepackage{amssymb}
\usepackage{amsthm}
\usepackage{mathtools}
\usepackage{commath}
\usepackage{nicefrac}
\usepackage{enumitem}
\usepackage{dsfont}

\usepackage{subcaption}
\usepackage{xcolor}
\usepackage{authblk}
\usepackage{graphicx}

\DeclareCaptionSubType[alph]{figure}
\makeatletter
\renewcommand\p@subfigure{\thefigure~}
\makeatother
\newtheorem{lemma}{Lemma}
\newtheorem{definition}{Definition}
\newtheorem{theorem}{Theorem}
\newtheorem{remark}{Remark}
\newtheorem{example}{Example}

\title{The locally adapted parametric finite element method for interface problems on triangular
meshes}

\author[1]{Johan Hoffman\thanks{\texttt{jhoffman@kth.se}}}
\affil[1]{\small{Department of Computational Science and Technology, School of Computer Science and
Communication, KTH Royal Institute of Technology, Sweden}}
\author[1]{B\"arbel Holm\thanks{\texttt{barbel@kth.se} corresponding author}}
\author[2]{Thomas Richter\thanks{\texttt{richter@math.fau.de}}}
\affil[2]{\small{Universit\"at Magdeburg,
Fakult\"at f\"ur Mathematik,
Institut f\"ur Analysis und Numerik,
Universit\"atsplatz 2,
39106 Magdeburg, Germany}}

\begin{document}
\maketitle
\begin{abstract}
We present a locally adapted parametric finite element method for interface problems. For this
adapted finite element method we show optimal
convergence for elliptic interface problems with a discontinuous diffusion parameter. The
method is based on the adaption of macro elements where a local basis represents the interface. The
macro elements are independent of the interface and can be cut by the interface. A
macro element which is a triangle in the triangulation is divided into four subtriangles. On these
subtriangles, the basis functions of the macro element are interpreted as linear functions. The
position of the vertices of these subtriangles is determined by the location of the interface in the
case a macro element is cut by the interface. Quadrature is performed on the subtriangles via
transformations to a reference element. Due to the locality of the method, its use is well suited on
distributed architectures.
\end{abstract}






\section{Introduction}
\label{intro}
In this paper we assume the domain
$\Omega \subset \mathds{R}^2$ to be partitioned into two non-overlapping parts $\Omega = \Omega_1
\cup \Gamma \cup
\Omega_2$, and with the intersection $\Gamma \coloneqq \partial\Omega_1\cap\partial\Omega_2$ we denote the interface
between $\Omega_1$ and $\Omega_2$.
We consider the problem 
\begin{equation}
  -\nabla \cdot (\kappa_i\nabla u) = f \text{ in } \Omega_i \subset\mathds{R}^2,\ i= 1,2, \quad [u] = 0,\quad
[\kappa\partial_nu] = 0 \text{ on }\Gamma,
\label{laplace}
\end{equation}
with positive diffusion parameters $\kappa_i>0$ defined on $\Omega_i, i = 1, 2,$
where
\[ [u](x) \coloneqq \lim_{s\downarrow 0} u(x+sn) - \lim_{s\uparrow 0} u(x+sn), \quad x \in
\Gamma \]
denotes the jump at the interface, with $n$ a normal vector of $\Gamma$.
For a domain $\Omega \in \mathds{R}^2$ we denote by $H^k(\Omega)$ the Sobolev
space of integer order $k\ge 0$ with norm $\norm{\cdot}_{H^k(\Omega)}$ and seminorm
$\abs{\cdot}_{H^k(\Omega)}$ involving only the highest derivatives. Throughout the paper, we shall use the notation of
the inner product on $L^2(\Omega) = H^0(\Omega)$ given by
\[
(u,\varphi)_{L^2(\Omega)} = (u,\varphi)_{\Omega} =  (u,\varphi) = \int_{\Omega} u\varphi \dif x,
\]
and the norm
\[
\|u\|_{L^2(\Omega)}=
\|u\|_{\Omega}=
\|u\|=
\left(\int_\Omega|u|^2\dif x\right)^{\nicefrac{1}{2}},
\]
induced by this inner product.
We assume that both subdomains $\Omega_1$ and $\Omega_2$ have a boundary with sufficient regularity
such that for smooth right hand sides, the solution has the regularity
\[
u \in H^1_0(\Omega) \cap H^{r+1}(\Omega_1\cup\Omega_2),
\]
for a given $r \in \mathds{N}$, see~\cite{Babuška1970}.
The variational formulation of this interface problem reads
\begin{equation}
  \text{Find }u\in H^1_0(\Omega): \quad a(u,\varphi) \coloneqq \sum_{i = 1}^2 (\kappa_i\nabla u, \nabla
\varphi)_{\Omega_i} =
(f,\varphi)_{\Omega}\quad \text{for all }\varphi \in H^1_0(\Omega).\label{contproblem}
\end{equation}
By standard arguments, the existence of solutions follows.
The error between an analytical solution $u$ and an approximation $u_h$ by a standard finite element
method with linear or higher order basis functions which do not conform to the interface will be bounded
by 
\[
\norm{\nabla (u - u_h)}_\Omega = O(h^{\nicefrac12}),
\]
see~\cite{Babuška1970}, \cite{Mackinnon}, and Figure~\ref{conv}, which shows the error using
standard finite elements for the
numerical example presented in Section~\ref{circular_interface}.

\begin{figure}[h]
  \centering
  \resizebox{\textwidth}{!}{
  \input{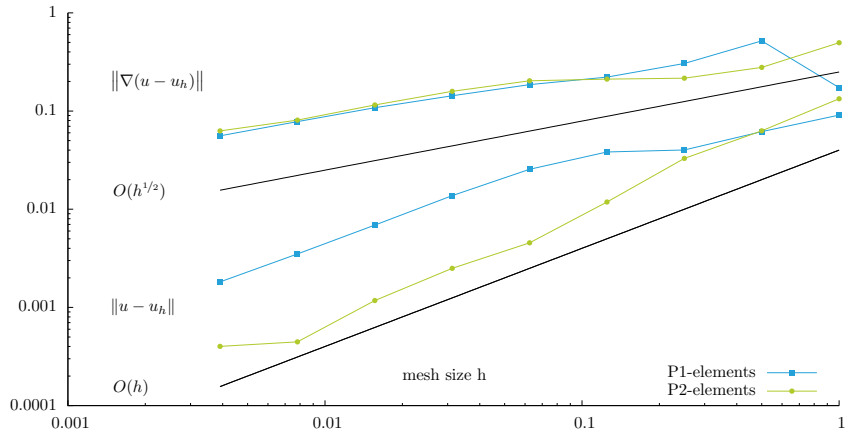}
  }
\caption{Convergence for standard finite element methods for interface problem~\eqref{laplace}. Solid lines indicate
a slope of $h^{\nicefrac12}$ and $h$, respectively.}\label{conv}
\end{figure}
\noindent
To recover the optimal order of convergence for interface problems various techniques have been
proposed, including so called unfitted finite element 
methods which locally modify or enrich the finite element basis. Examples are the extended finite element
method (XFEM)~\cite{Moes99}, the generalized finite element method~\cite{Babuska04}, and the
unfitted Nitsche method~\cite{Hansbo02,Hansbo04}. These methods locally
modify the finite element basis. Thus, the connectivity of the system matrix is changed and
degrees of freedom are added or removed. In view of distributed parallel algorithms, this would lead to
costly load balancing.
\noindent
Recent work also includes cut finite element methods (CutFEM)~\cite{Burman15,Burman2015}. This
approach uses Nitsche's method and stabilization of the finite element method on facets close to the
interface. 
\noindent
In~\cite{FreiRichter14}, a locally adapted (patch) finite element method is proposed for
quadrilateral elements, where the mesh is locally adapted to align with the interface.
A similar method to the one we present here is described in~\cite{Gangl2016}. In contrast to their
approach, we use rectangular triangles as reference triangles and thus, get a slightly different
bound for the
maximum angle condition. 
In this paper, we describe the locally adapted finite element method for interface problems and prove a priori error estimates which are verified in numerical experiments.

The outline of this paper is as follows. In Section~\ref{fem}, we introduce the method and describe
the finite element spaces on cut cells. In order to prove optimal a priori error estimates for the locally
adapted finite element patch method, we show in Section~\ref{maxangle} that for a certain choice of free parameters in a cell cut by the interface, a maximum angle condition is satisfied.
In Section~\ref{error}, we prove optimal a priori error bounds for the locally adapted patch finite element
method for interface problems, followed by numerical examples in Section~\ref{numerics}. We conclude
this paper with Section~\ref{conclusion}.  

\section{The locally adapted patch finite element method}
\label{fem}

Let $\mathcal{T}_h$ be a shape-regular triangulation of a domain $\Omega \subset \mathds{R}^2$ into triangles. Since we
do not require that the triangulation is aligned with the interface, triangles $T \in
\mathcal{T}_h$ can be cut by the interface. On these triangles we get contributions from both
subproblems. To integrate those
contributions we propose a locally adapted finite element method on patches of subtriangles.
Note that the triangulation does not necessarily coincide with the subdivision of the domain $\Omega =
\Omega_1\cup\Gamma\cup\Omega_2$.
We assume that the triangulation has a patch structure such that each triangle $T\in\mathcal{T}_h$ is
divided into four smaller subtriangles $T_0,T_1,T_2,T_3$, see Figures~\ref{refA} and~\ref{refB} for the two
different reference configurations of a triangle $T \in \mathcal{T}_h$. By a linear
transformation, the vertices close to the cut by the interface are
mapped to the exact location of the cut. 
We will now construct a finite element method for a mesh $\mathcal{T}_h$ of locally transformed
triangles $T$ cut by the interface.

\subsection{Definition of the finite element space on cut triangles}
Before defining the finite element space on cut triangles we define when we consider a patch
triangle to be cut. We allow two possible configurations which
are:
\begin{enumerate}
  \item Each (open) patch $T \in \mathcal{T}_h$ is not cut, such that $T \cap \Gamma = \emptyset$
    holds or
  \item a patch $T \in \mathcal{T}_h$ is cut in exactly two points on its boundary such that  
    $T \cap \Gamma \ne \emptyset$ and $\partial T \cap \Gamma = \{x_1^T, x_2^T\}$.
\end{enumerate}
If the interface cuts through two vertices of a patch, we do not consider the patch cut.
We restrict our method such that if a patch is cut, the two points
$\{x_1^T,x_2^T\}$ may not be inner points of the same edge.
That means we do not allow a patch to be cut multiple times and the interface may not enter and
leave the patch at the same edge. Using refinement of the underlying mesh, these restrictions can be avoided.
The finite element space $V_h\subset H_0^1(\Omega)$ is defined as an isoparametric space on the
triangulation $\mathcal{T}_h$ given as
\[
V_h \coloneqq \Big\{\varphi \in C(\bar{\Omega}),\ \varphi \circ F^{-1}_{T_i} \Bigm\vert_{T_i} \in \hat{P}
\text{ for } i = 0,\ldots, 3, \text{ and all patches } T \in
\mathcal{T}_h\Big\},
\]
where $F_{T_i}$ is the mapping between the reference patch $\hat{T}$ and every patch $T\in
\mathcal{T}_h$ such that 
\[
F_{T_i}(\hat{x}_k) = x_k,\quad k = 1,\ldots,6, \quad i = 0, \ldots, 3,
\]
for the six nodes $x_1^T,\ldots,x_6^T$ in every patch.
We choose the reference space $\hat{P}$ of polynomials as the standard
space of piecewise linear functions which is given as
\[
\hat{P} = 
\Big\{\varphi \in C(\bar{T}),\ \varphi \Bigm\vert_{T_i} \in \operatorname{span}\{1,x,y\},\ 
T_i \in T, i = 0, \ldots, 3\Big\}.
\]
With $\{\hat{\varphi}_1, \ldots,\hat{\varphi}_6\}$ we denote the standard Lagrange basis of
$\hat{P}$ for which $\hat{\varphi}_i(\hat{x}_j^T) = \delta_{ij}$ holds. Accordingly, the transformation
$F_{T_i}$ is given by
\[
F_{T_i}(\hat{x}) = \sum_{k=1}^6 x_k^T\hat{\varphi}_k(\hat{x}), \quad i = 0, \ldots, 3.
\]
In the following, we will describe the transformation in more detail.
\begin{figure}[h]
  \centering
  \begin{subfigure}[b]{.48\textwidth}
  \centering
  \resizebox{.7\textwidth}{!}{\input{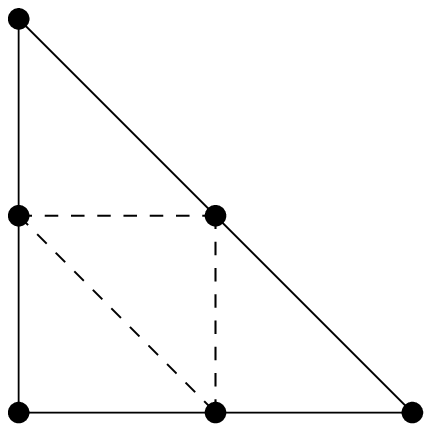}}
  \caption{Reference configuration A.}
  \label{refA}
\end{subfigure}%
  \begin{subfigure}[b]{.48\textwidth}
  \centering
   \resizebox{.7\textwidth}{!}{\input{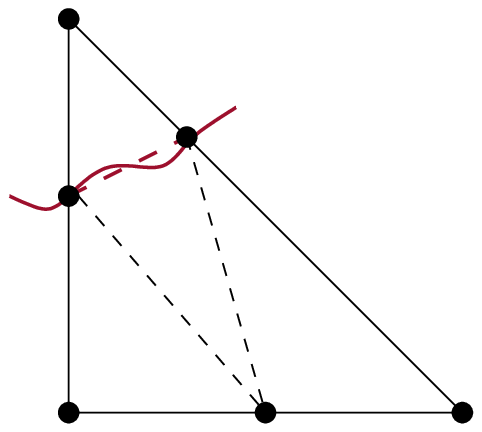}}
   \caption{Cut through two edges.}
   \label{cutedge}
\end{subfigure}%
\caption{Cut through two edges, reference configuration~\subref{refA} and actual, adapted
configuration~\subref{cutedge}, interface in solid red line, linear approximation of the interface
in dashed red line.}\label{cuttwoedges}
\end{figure}

\begin{figure}[h]
  \centering
  \begin{subfigure}[b]{.48\textwidth}
  \centering
  \resizebox{.7\textwidth}{!}{\input{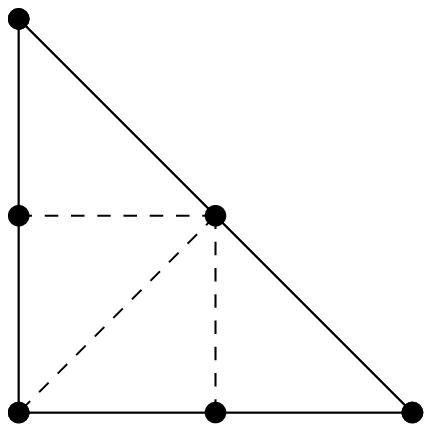}}
  \caption{Reference configuration B.}
  \label{refB}
\end{subfigure}%
  \begin{subfigure}[b]{.48\textwidth}
  \centering
   \resizebox{.7\textwidth}{!}{\input{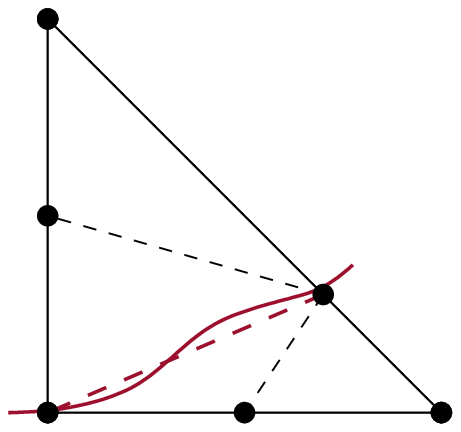}}
   \caption{Cut through a vertex and an edge.}
   \label{cutvertexedge}
\end{subfigure}%
\caption{Cut through lower left vertex and the opposite edge, reference
configuration~\subref{refB} and
actual, adapted configuration~\subref{cutvertexedge}, interface in solid red line, linear approximation of the interface
in dashed red line.}\label{cutlowerleftvertex}
\end{figure}

\begin{figure}[h]
  \centering
  \begin{subfigure}[b]{.48\textwidth}
  \centering
  \resizebox{.7\textwidth}{!}{\input{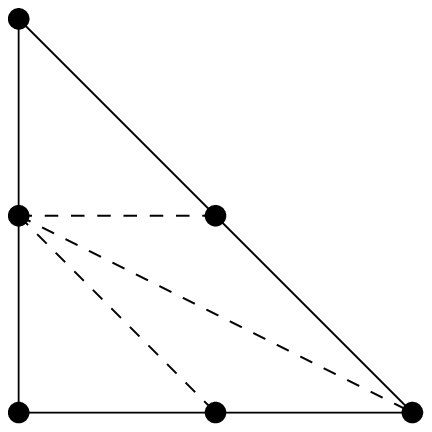}}
  \caption{Reference configuration C.}
  \label{refC}
\end{subfigure}%
  \begin{subfigure}[b]{.48\textwidth}
  \centering
   \resizebox{.7\textwidth}{!}{\input{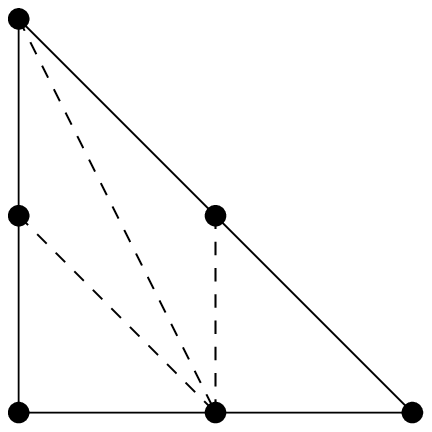}}
   \caption{Reference configuration D.}
   \label{refD}
\end{subfigure}%
\caption{Reference configuration for the cut of the right vertex and the opposite
edge~\subref{refC} and
for the cut of the upper left vertex and the opposite edge~\subref{refD}.}
\end{figure}

\subsection{Quadrature on cut triangles}
To define the quadrature rules on the cells which are cut by the interface, we introduce the four
reference configurations that we use.
\begin{description}
  \item[Configuration A] The cell is cut through exactly two edges, see Figure~\ref{cuttwoedges} for
   both, the reference and an actual configuration.
  \item[Configuration B] The cell is cut through the lower left vertex and the opposite edge, see
    Figure~\ref{cutlowerleftvertex} for
   both, the reference and an actual configuration.
  \item[Configuration C] The cell is cut through the right vertex and the opposite edge, see
    Figure~\ref{refC} for the reference configuration.
  \item[Configuration D] The cell is cut through the upper left vertex and the opposite edge, see
    Figure~\ref{refD} for the reference configuration.
  \end{description}
Since the finite element method is defined using standard polynomial basis functions, we also use
standard quadrature rules on the triangles $T_0, \ldots, T_3 \in T \in \mathcal{T}_h$.
For the quadrature on a patch $T\in \mathcal{T}_h$, the quadrature rule is composed by a combination
of standard rules on
all triangles $T_0, \ldots, T_3$, as it is sketched in Figure~\ref{quadrature}.
\begin{figure}[h]
  \centering
   \resizebox{1.\textwidth}{!}{\input{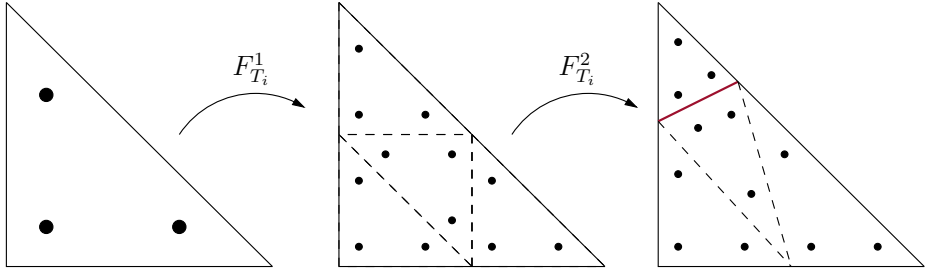}}
   \caption{Quadrature rule on a patch $T$ as a composition of quadrature rules on each triangle $T_0, \ldots, T_3$.}\label{quadrature}
\end{figure}
We start from a quadrature rule for triangles on the reference triangle (Figure~\ref{quadrature} on
the left). This can be any quadrature rule suitable for the integration of linear polynomials on
the reference triangle. Then, for each triangle $\hat{T}_i$, the quadrature points are mapped via the
transformation $F_{T_i}^1$ to each of the subtriangles of one of the reference configurations A --
D. With $F_{T_i}^2$, these quadrature points then are mapped to their location in real coordinates. Thus, the
transformation $F_{T_i}$ can be decomposed into 
\[
F_{T_i} = F_{T_i}^2\circ F_{T_i}^1.
\]
The quadrature weights are scaled appropriately.
\begin{example}
  Let us assume that the triangle with vertices $v_1 = (0,0),\ v_2 = (1,0)$, and $v_3 = (0,1)$ is cut
  with the parameters $q = \nicefrac{9}{16},\ s = \nicefrac12$, and $r = \nicefrac{11}{16}$ and
  further let the basis quadrature rule for integrating linear polynomials have the quadrature
  points
  \[
  \hat{q}_1 = (\nicefrac23, \nicefrac16), \quad \hat{q}_2 = (\nicefrac16, \nicefrac16),\quad  \hat{q}_3 =
  (\nicefrac16, \nicefrac23),
  \]
  and the weights $\hat{\omega}_i = \nicefrac16$ for $i = 1,2,3$.
  Then, the corresponding quadrature points for a triangle based on reference configuration A are given as
  \begin{align*}
    F_{T_0}\hat{q}_1 &= (\nicefrac13,\nicefrac{3}{32}), \quad &F_{T_0}\hat{q}_2
    &= (\nicefrac{1}{12}, \nicefrac{3}{32}), \quad &
    F_{T_0}\hat{q}_3 &= (\nicefrac{1}{12}, \nicefrac{3}{8}),\\
    F_{T_1}\hat{q}_1 &= (\nicefrac{77}{96},\nicefrac{11}{96}), \quad &F_{T_1}\hat{q}_2 &= (\nicefrac{53}{96},
    \nicefrac{11}{96}),
    \quad &F_{T_1}\hat{q}_3 &= (\nicefrac{11}{24}, \nicefrac{11}{24}),\\
    F_{T_2}\hat{q}_1 &= (\nicefrac{5}{24},\nicefrac{23}{32}), \quad &F_{T_2}\hat{q}_2 &= (\nicefrac{5}{96},
    \nicefrac{21}{32}),
    \quad &F_{T_2}\hat{q}_3 &= (\nicefrac{5}{96}, \nicefrac{7}{8}),\\
    F_{T_3}\hat{q}_1 &= (\nicefrac{13}{96},\nicefrac{47}{96}), \quad &F_{T_3}\hat{q}_2 &=
    (\nicefrac{7}{24}, \nicefrac{53}{96}),
    \quad &F_{T_3}\hat{q}_3 &= (\nicefrac{37}{96}, \nicefrac{5}{24}),\\
\end{align*}
and the weights are scaled such that 
\[
\omega_i = \nicefrac{1}{24}, \quad i = 1,2,3, \text{ for all } T_0,\ldots,T_3.
\]
\end{example}

\subsection{Discrete variational formulation}
With the definition of the discrete finite element space, we are ready to formulate the discrete counterpart of
problem~\eqref{contproblem} as
\begin{equation}
  \text{Find }u_h\in V_h: \quad a_h(u_h,\varphi_h) \coloneqq \sum_{i = 1}^2 (\kappa_i\nabla u_h, \nabla
\varphi_h)_{\mathcal{T}_{i,h}} =
(f,\varphi_h)_{\Omega}\quad \text{for all }\varphi_h \in V_h.\label{discproblem}
\end{equation}
Note that we do not have the standard Galerkin orthogonality property due to the fact that the value of $\kappa$ differs
in a small layer between the continuous interface $\Gamma$ and the linear approximation $\Gamma_h$,
as is shown in Figure~\ref{subdommesh}.
\begin{figure}[h]
\centering
\begin{subfigure}[b]{.45\textwidth}
  \centering
   \resizebox{.8\textwidth}{!}{\input{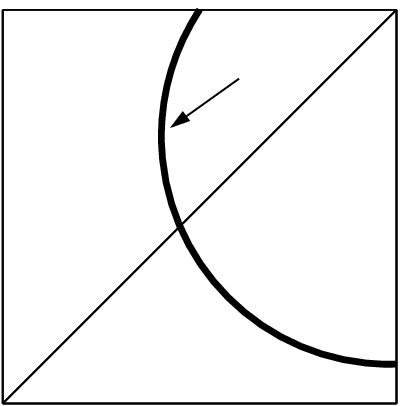}}
 \end{subfigure}%
\begin{subfigure}[b]{.45\textwidth}
  \centering
   \resizebox{.8\textwidth}{!}{\input{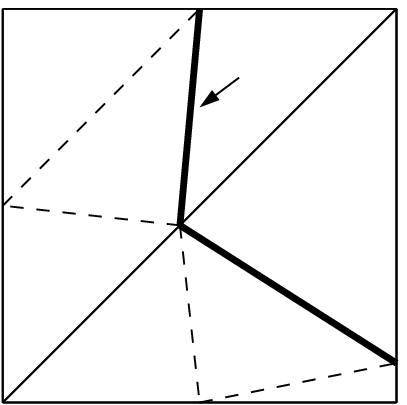}}
 \end{subfigure}%
 \caption{Splitting into subdomains and subtriangulations}\label{subdommesh}
\end{figure}

\section{Maximum angle condition}\label{maxangle}
In order to prove the optimal order of convergence for the locally adapted patch finite element method,
we need the Lagrangian interpolation operator
\begin{equation*}
  L_h \colon H^2(T)\cap C(\bar{T}) \to V_h
\end{equation*}
to satisfy
\begin{equation}
\norm{\nabla^k(v-L_hv)}_T \le ch_{T,\max}^{2-k}\norm{\nabla^2v}_T,\label{interpolation}
\end{equation}
where $c>0$ is a constant and $h_{T,\max}$ is the maximum diameter of a triangle $T$.
In~\cite{Apel1999} it is shown that a necessary condition for the above estimate to hold is the
maximum angle condition.\\
\begin{definition}[Maximum angle condition]
There is a constant $\gamma_* < \pi$, independent of $h$ and $T_i \in \mathcal{T}_h$ such that the
maximal interior angle $\gamma$ of any element $T_i$ is bounded by $\gamma_*$:
\[
\gamma \le \gamma_* <\pi.
\]
\end{definition}
\begin{figure}[h]
  \centering
  \begin{subfigure}[b]{.45\textwidth}
  \centering
    \resizebox{.7\textwidth}{!}{\input{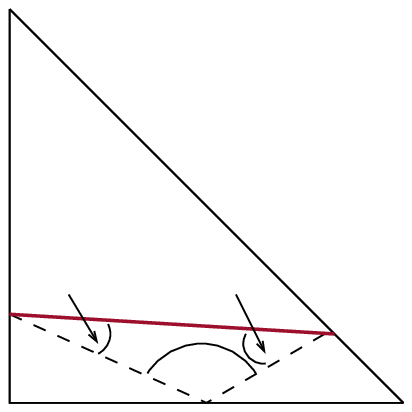}}
    \caption{Two edges are cut and if $q\to 0$ and $r\to 0$ then $\gamma \to \pi$.}
\label{twofacets}
  \end{subfigure}%
\hspace{5mm}
  \begin{subfigure}[b]{.45\textwidth}
  \centering
    \resizebox{.7\textwidth}{!}{\input{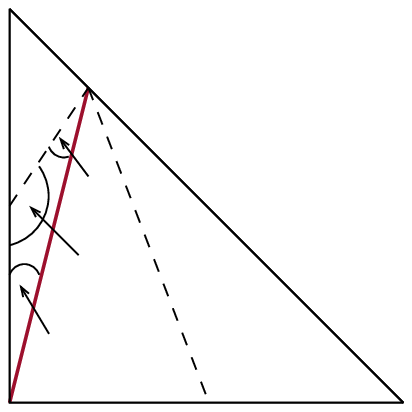}}
    \caption{A vertex and an edge are cut and if $r\to 1$ then $\alpha_1 \to \pi$.}
    \label{vertexfacet}
  \end{subfigure}%
  \caption{Illustration of two cases in which one of the inner triangles becomes anisotropic and the
  maximum angle condition does not hold}\label{bigangles}
\end{figure}
In contrast to the corresponding finite element method on quadrilateral cells where the maximal angle
condition is fulfilled by construction, see ~\cite{FreiRichter14}, the version of this method on triangles 
lacks this property.

There are only two possible configurations of how the interface can cut a triangle; either the
interface cuts two edges (Figure~\ref{cutedge}) or the interface cuts one vertex and one edge, see
Figure~\ref{cutvertexedge}.
For each of these configurations, there are cases in which the maximum angle condition is not fulfilled as shown in
Figure~\ref{bigangles}.
Since the cut of the interface through a cell determines one of the three parameters $s, r$ and
$q$ in the case that the interface goes through a vertex and two parameters if the cut
goes through two edges of a cell, it leaves at least one parameter per cell free to choose. In the following we discuss
different possibilities to choose the free parameters and we show that by those choices the maximum angle condition
will be satisfied. First, we treat the case in which only edges are cut. Later, the case of a cut
through a vertex and an edge is discussed. 
For the estimation of the angles, we use the notation introduced in Figure~\ref{angles}.
  \begin{figure}[h]
    \centering
    \begin{subfigure}[b]{.3\textwidth}
    \centering
    \resizebox{1.\textwidth}{!}{
    \input{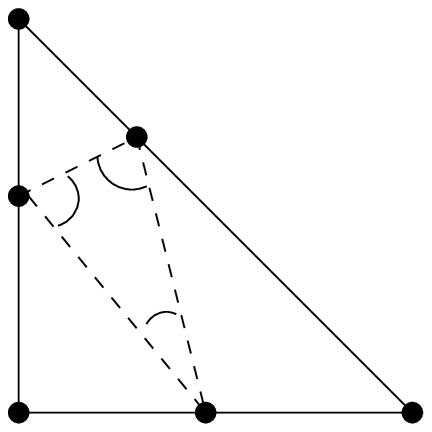}
    }
    \caption{Cut through two edges\phantom{Cut through lower right vertex and an edge}}
  \end{subfigure}%
\hspace{5mm}
    \begin{subfigure}[b]{.3\textwidth}
    \centering
    \resizebox{1.\textwidth}{!}{
    \input{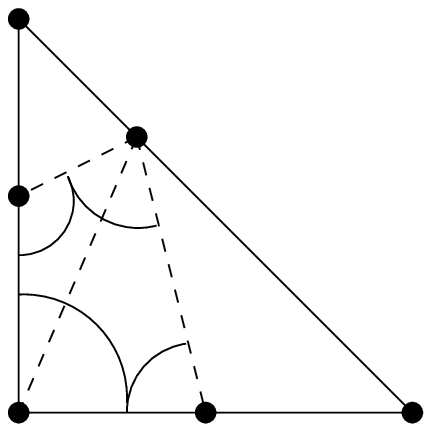}
    }
    \caption{Cut through lower left vertex and an edge}\label{cutleftvertex}
  \end{subfigure}%
\hspace{5mm}
    \begin{subfigure}[b]{.3\textwidth}
    \centering
    \resizebox{1.\textwidth}{!}{
    \input{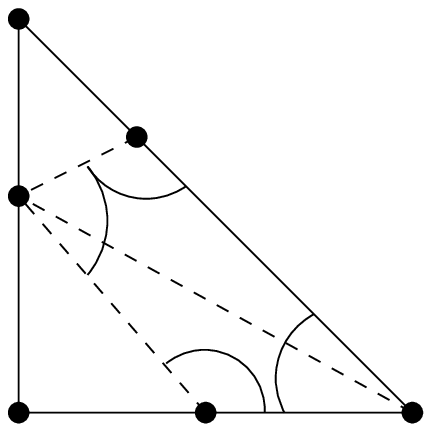}
    }
    \caption{Cut through lower right vertex and an edge}\label{cutrightvertex}
  \end{subfigure}%
    \caption{Angles to be estimated}\label{angles}
  \end{figure}

\subsection{Two edges are cut}
Extreme shapes of the inner triangles $T_0,\ldots, T_3$ can arise. Their definition can be seen in Figure~\ref{cutedge}. The discussion of how to choose
the free parameters is divided into two different parts. We start with
treating the case in which the interface
comes arbitrarily close to an edge so that an angle will approach the value $\pi$.

\subsubsection{The interface cuts arbitrarily close to an edge}
This situation is shown in Figure~\ref{twofacets}.
There are three situations in which this could happen, namely if
\begin{enumerate}
  \item $q \to 1$ and $s \to 1$, or if 
  \item $s \to 0$ and $r \to 1$, or if
  \item $q \to 0$ and $r \to 0$. 
\end{enumerate}
We use that for the angle $\vartheta$ between two vectors $a$ and $b$ it holds
\[
\cos(\vartheta) = \frac{(a,b)}{\norm{a}\norm{b}},
\]
where $(\cdot,\cdot)$ and $\norm{\cdot}$ are the Euclidean scalar product and norm, respectively. 
Thus, for the angles $\alpha, \beta$ and $\gamma$ in Figure~\ref{cutleftvertex} it holds
\begin{align*}
  \cos(\alpha) &= \frac{(1-r)(1-r-s)+r(r-q)}{\sqrt{(q-r)^2+(r-1)^2}\sqrt{(1-r-s)^2+r^2}}, \\
  \cos(\beta) &= \frac{s(1-r)+q(q-r)}{\sqrt{(q-r)^2+(r-1)^2}\sqrt{s^2+q^2}},\\
  \cos(\gamma) &= \frac{s(s-1+r)+rq}{\sqrt{(1-r-s)^2+r^2}\sqrt{s^2+q^2}}.
\end{align*}
If $q \to 1$ and $s \to 1$ it follows that $\cos(\alpha) \to \pi$. Therefore, for $q>\nicefrac12$
and $s>\nicefrac12$, we recommend one of the following choices.
\subsubsection{Choices for $r$}
\begin{equation}
  r \coloneqq 1-s, \qquad r \coloneqq q, \qquad r \coloneqq (1-s)(1-q).\label{chooser}
\end{equation}
Independent of the choice above, we can estimate such that it holds
\[
1\ge \cos(\alpha)\ge \frac{-1}{\sqrt2}, \qquad 
1\ge \cos(\beta)\ge \frac{1}{\sqrt5}, \qquad 
1\ge \cos(\gamma)\ge \frac{-1}{\sqrt2},
\]
and therefore for all angles in these cases it follows that $\alpha, \beta, \gamma \in (0^\circ, 135^\circ)$ for $q,s \in
(\nicefrac12, 1)$. The other two cases are treated similarly.
For $s < \nicefrac12$ and $r > \nicefrac12$ we suggest to set the values for $q$
as follows. 
\subsubsection{Choices for $q$}
\begin{equation}
  q \coloneqq s, \qquad q \coloneqq r, \qquad q \coloneqq (1-r)s,\label{chooseq}
\end{equation}
This means that independent of the choice above, we have that 
\[
1\ge \cos(\alpha)\ge \frac{-1}{\sqrt5}, \qquad 
1\ge \cos(\beta)\ge \frac{-1}{\sqrt2}, \qquad 
1\ge \cos(\gamma)\ge \frac{-1}{\sqrt{10}},
\]
and therefore for all angles in these cases it holds $\alpha, \beta, \gamma \in (0^\circ,
135^\circ)$ for $r \in (\nicefrac12, 1)$ and $s \in (0, \nicefrac12)$.
The last case in which the interface can come arbitrarily close to an edge is when $q\to0$ and
$r\to0$. In that case, for $q<\nicefrac12$ and $r<\nicefrac12$, we suggest one of the following
choices.
\subsubsection{Choices for $s$}
\begin{equation}
  s \coloneqq 1-r, \qquad s \coloneqq q, \qquad s \coloneqq qr,\label{chooses}
\end{equation}
Due to symmetry reasons, we find that independent of the choice above, it follows that
\[
1\ge \cos(\alpha)\ge \frac{-1}{\sqrt5}, \qquad 
1\ge \cos(\beta)\ge \frac{-1}{\sqrt{10}}, \qquad 
1\ge \cos(\gamma)\ge \frac{-1}{\sqrt{2}}.
\]
Thus, for all the angles in these cases it holds $\alpha, \beta, \gamma \in (0^\circ,
135^\circ)$ for $q$ and $r \in (0, \nicefrac12)$.

\subsubsection{The remaining cases in which two edges are cut}
In all the remaining cases in which two edges are cut, we choose the parameter which is not
determined by the cut of the interface as $\nicefrac12$. 
For these cases it follows that 
\[
1\ge \cos(\alpha)\ge \frac{-3}{\sqrt{10}}, \qquad 
1\ge \cos(\beta)\ge \frac{-2}{\sqrt{5}}, \qquad 
1\ge \cos(\gamma)\ge \frac{-2}{\sqrt{5}},
\]
and therefore it holds for the angles in all the cases
\[
\alpha, \beta, \gamma \in (0^\circ, 162^\circ), \quad \text{ for } q, r, s \in (0,1).
\]

\subsection{A vertex and an edge are cut}
In contrast to the case in which two edges are cut, only one value is determined in the case in
which a vertex and an edge are cut. In the
configuration shown in Figure~\ref{vertexfacet} we are free to choose the values for $s$ and
$q$. The value of $r$ is given by the cut of the interface.
Because of symmetry reasons, it suffices to consider the cases depicted in Figure~\ref{cutleftvertex}
and~\ref{cutrightvertex} to estimate the inner angles.
We propose to choose the parameters as follows:
\begin{equation}
  s =
\begin{cases}
  1-r & \text {if }r < \nicefrac12,\\
  q & \text{if } q < \nicefrac12,\\
  \nicefrac12 & \text{otherwise,}
\end{cases}\
  r =
\begin{cases}
  q & \text {if }q > \nicefrac12,\\
  1-s & \text {if }s > \nicefrac12,\\
  \nicefrac12 & \text{otherwise,}
\end{cases}\
  q =
\begin{cases}
  r & \text {if }r > \nicefrac12,\\
  s & \text {if }s < \nicefrac12,\\
  \nicefrac12 & \text{otherwise.}
\end{cases}\label{choicevertexfacet}
\end{equation}
For the inner angles in $T_0$ of the configuration in Figure~\ref{cutleftvertex} it holds
\begin{align*}
  \cos(\alpha_1) &= \frac{q-r}{\sqrt{(1-r)^2+(r-q)^2}},\\
  \cos(\beta_1) &= \frac{r}{\sqrt{(1-r)^2+r^2}},\\
  \cos(\gamma_1) &=
  \frac{(1-r)^2+r(r-q)}{\sqrt{(1-r)^2+r^2}\sqrt{(1-r)^2+(q-r)^2}},
\end{align*}
whereas for the inner angles in $T_1$, we find
\begin{align*}
  \cos(\alpha_2) &= \frac{1-r}{\sqrt{(1-r)^2+r^2}},\\
  \cos(\beta_2) &= \frac{s-1+r}{\sqrt{r^2+(s-1+r)^2}},\\
  \cos(\gamma_2) &=
  \frac{(1-r)(1-r-s)+r^2}{\sqrt{r^2+(1-r)^2}\sqrt{r^2+(1-r-s)^2}}.
\end{align*}
Therefore, we find that 
\begin{alignat*}{6}
1 &\ge \cos(\alpha_1)&&\ge \frac{-1}{\sqrt{2}} \quad &&\Rightarrow \alpha_1 \in (0^\circ, 135^\circ)
\quad &\text{ for } r,q,s \in (0,1),\\
1 &\ge \cos(\alpha_2)&&\ge 0 \quad &&\Rightarrow \alpha_2 \in (0^\circ, 90^\circ)
\quad &\text{ for } r,q,s \in (0,1),\\
1 &\ge \cos(\beta_1)&&\ge 0 \quad &&\Rightarrow \beta_1 \in (0^\circ, 90^\circ)
\quad &\text{ for } r,q,s \in (0,1),\\
1 &\ge \cos(\beta_2)&&\ge  \frac{-1}{\sqrt{2}} \quad &&\Rightarrow \beta_2 \in (0^\circ, 135^\circ)
\quad &\text{ for } r,q,s \in (0,1),\\
1 &\ge \cos(\gamma_1)&&\ge 0 \quad &&\Rightarrow \gamma_1 \in (0^\circ, 90^\circ)
\quad &\text{ for } r,q,s \in (0,1),\\
1 &\ge \cos(\gamma_2)&&\ge 0 \quad &&\Rightarrow \gamma_2 \in (0^\circ, 90^\circ)
\quad &\text{ for } r,q,s \in (0,1).
\end{alignat*}
If the cut goes through the lower right vertex and an edge as shown in Figure~\ref{cutrightvertex},
we arrive at
\begin{align*}
  \cos(\alpha_3) &= \frac{-s}{\sqrt{s^2+q^2}},&
  \cos(\alpha_4) &= \frac{1+q}{\sqrt{2}\sqrt{1+q^2}},\\
  \cos(\beta_3) &= \frac{1}{\sqrt{1+q^2}},&
  \cos(\beta_4) &=
  \frac{(1-r)+q(r-q)}{\sqrt{1+q^2}\sqrt{(1-r)^2+(r-q)^2}},\\
  \cos(\gamma_3) &=
  \frac{s+q^2}{\sqrt{1+q^2}\sqrt{s^2+q^2}},&
  \cos(\gamma_4) &= \frac{2r-1-q}{\sqrt{2}\sqrt{(1-r)^2+(r-q)^2}}
\end{align*}
for the remaining angles to be estimated. 
For the angles, we find that it holds
\begin{alignat*}{6}
1 &\ge \cos(\alpha_3)&&\ge \frac{-1}{\sqrt{2}} \quad &&\Rightarrow \alpha_1 \in (0^\circ, 135^\circ)
\quad &\text{ for } r,q,s \in (0,1),\\
1 &\ge \cos(\alpha_4)&&\ge \frac{1}{\sqrt{2}} \quad &&\Rightarrow \alpha_2 \in (0^\circ, 45^\circ)
\quad &\text{ for } r,q,s \in (0,1),\\
1 &\ge \cos(\beta_3)&&\ge \frac{1}{\sqrt{5}} \quad &&\Rightarrow \beta_3 \in (0^\circ, 64^\circ)
\quad &\text{ for } r,q,s \in (0,1),\\
1 &\ge \cos(\beta_4)&&\ge  0 \quad &&\Rightarrow \beta_4 \in (0^\circ, 90^\circ)
\quad &\text{ for } r,q,s \in (0,1),\\
1 &\ge \cos(\gamma_3)&&\ge \frac{1}{\sqrt{2}} \quad &&\Rightarrow \gamma_3 \in (0^\circ, 45^\circ)
\quad &\text{ for } r,q,s \in (0,1),\\
1 &\ge \cos(\gamma_4)&&\ge \frac{-3}{\sqrt{10}} \quad &&\Rightarrow \gamma_4 \in (0^\circ, 162^\circ)
\quad &\text{ for } r,q,s \in (0,1).
\end{alignat*}
For the cut through the upper vertex and the opposite edge, the estimates for the inner angles
follow due to symmetry reasons. 
The findings are collected in the following Lemma.
\begin{lemma}
  With the choice of parameters in~\eqref{chooser},~\eqref{chooseq},~\eqref{chooses},
  and~\eqref{choicevertexfacet}, all the interior angles in the triangles that can occur through a
  cut of an interface are bounded by $162^\circ$ independent of $r,q,s \in (0,1)$.
\end{lemma}
With this result we are in the position to analyze the a priori error of this locally adapted finite
element patch method.
\begin{remark}
Due to the adjustments of the parameters which are not determined by the location of the interface,
we have to ensure continuity across edges. That means that the parameter for
the neighboring element across an edge has to be set to the same value.
\end{remark}
\begin{remark}
  Note that the analysis of the angles also includes the case that two of the parameters $r,q$ and
  $s$ are equal $\nicefrac12$ and the remaining parameter takes any value between $0$ and $1$. This
  means that if a parameter is set in order to ensure continuity, we do not have to adjust another
  parameter in those cells. In particular, we only adjust those cells which are direct neighbors of
  a cut cell. 
\end{remark}
The case that two neighboring cells are cut by the interface can be circumvented by refinement of
the mesh.

\section{A priori error analysis}\label{error}
With the maximum angle condition satisfied, we can define a robust Lagrangian interpolation operator
$L_h \colon H^2(T)\cap C(\bar{T})\to V_h$ such that 
\[
\norm{\nabla^k(v-L_hv)}_T \le ch_{T,\max}^{2-k}\norm{\nabla^2v}_T.
\]
holds
with $c$ a positive constant,
see~\cite{Apel1999}. In order to derive a priori error estimates we have to take into
account that the partitioning of the mesh into submeshes $\mathcal{T}_h = \mathcal{T}_{1,h} \cup
\mathcal{T}_{2,h}$ connected to $\kappa_1$ and $\kappa_2$ does not coincide with the partitioning of the domain $\Omega = \Omega_1 \cup
\Gamma \cup \Omega_2$ which means that $\mathcal{T}_{i,h}$ not necessarily covers $\Omega_i$, see
Figure~\ref{intersection}.
This would only be possible if the interface $\Gamma$ is a polygon.
\begin{figure}[h]
   \centering
   \resizebox{.28\textwidth}{!}{\input{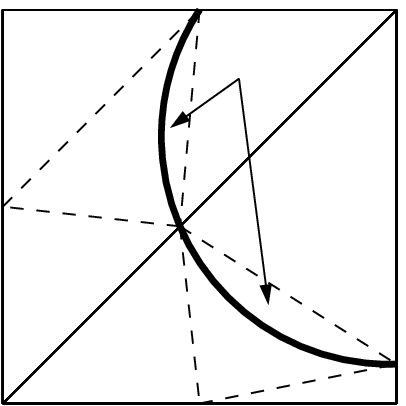}}
     \caption{Region between $\Omega_i$ and $\mathcal{T}_{i,h}$}\label{intersection}
\end{figure}
Later, we will use the following auxiliary result which is an estimate for functions in the region between
$\Omega_i$ and $\mathcal{T}_{i,h}$.
\begin{lemma}\label{region}
  Let $v\in H^1(\Omega)$. For the convex region between $\Omega_i$ and $\mathcal{T}_{i,h}$ it holds
  \[
  \norm{v}_{\Omega_i\setminus\mathcal{T}_{i,h}} \le ch\norm{v}_{H^1(\Omega)}, \quad i = 1,2.
  \]
  For $v_h\in V_h$ it holds
  \[
  \norm{v_h}_{H^s(\Omega_i\setminus\mathcal{T}_{i,h})} \le
  ch^{\nicefrac12}\norm{v_h}_{H^s(\Omega)}, \quad i = 1,2, \ s = 0,1.
  \]
\end{lemma}
\begin{proof}
  For a cell $T \in \mathcal{T}_h$ and $v_h$ a linear polynomial on $T$, the inverse
  inequality 
  \[
  \norm{v_h}_{\partial T} \le c h^{-\nicefrac12}\norm{v_h}_T
  \]
  holds, see~\cite{Ciarlet1978}. Therefore, we can estimate the norm over the discrete interface
  $\Gamma_h$ as
  \[
  \norm{v_h}_{\Gamma_h} \le c h^{-\nicefrac12} \norm{v_h}_{\Omega}.
  \]
  For the convex region between $\Omega_i$ and $\mathcal{T}_{i,h}$, the relations 
  \begin{align*}
  \norm{v}_{\Omega_i\setminus\mathcal{T}_{i,h}}^2 &\le c(h^2\norm{v}_{\Gamma}^2 +
  h^4\norm{\nabla v}_{\Omega}^2),\\
  \norm{v_h}_{\Omega_i\setminus\mathcal{T}_{i,h}}^2 &\le c(h^2\norm{v}_{\Gamma_h}^2 +
  h^4\norm{\nabla v_h}_{\Omega}^2)
\end{align*}
  are proven in~\cite{Bramble1994}. With the inverse inequality and the discrete trace inequality,
  the estimate follows for $v_h \in V_h$. The global trace inequality for $v \in H^1(\Omega)$
  \[
  \norm{v}_\Gamma \le c \norm{v}_{H^1(\Omega)} 
  \]
  yields the remaining inequality to be proven.
\end{proof}

We cannot directly apply the interpolation estimate~\eqref{interpolation} to all cells in the triangulation $\mathcal{T}_h$
as the solution $u$ is not smooth enough on cells which are cut by the interface, i.e.~for cells
$T\in\mathcal{T}_h$ with $T \cap \Gamma \neq\emptyset$.
Therefore, we define the set of
all elements $T \in \mathcal{T}_{h}$ which are cut by the interface
\begin{align}
S_h &= \{T\in \mathcal{T}_{h} \ \vert \ T\cap\Gamma \neq\emptyset\}.
\end{align}
An interpolation estimate on these cells of the triangulation is proven first.
\begin{lemma}\label{lemmainterpolation}
  For the cells $T\in S_h$, the interpolation estimate
  \begin{equation}
    \norm{\nabla(L_hu -u)}_{S_h} \le ch\norm{u}_{H^2(\Omega_1\cup\Omega_2)}
  \end{equation}
  holds with a positive constant $c$.
\end{lemma}
\begin{proof}
   We divide the estimation of the interpolation
   error into the cells which are effected by the interface and those which are not effected. On
   those cells not cut by the interface, we use the standard interpolation estimate and extend the
   domain to the complete domain again:
   \begin{align}
   \norm{\nabla (u -L_h u)}_{\Omega}^2 &= \norm{\nabla (u -L_h u)}_{\Omega\setminus S_h}^2 +
   \norm{\nabla (u -L_h u)}_{S_h}^2\nonumber\\
&\le ch^2 \norm{\nabla^2 u}_{\Omega_1 \cup \Omega_2}^2 +
\norm{\nabla (u -L_h u)}_{S_h}^2.
 \end{align}
 For the last term we introduce a continuous extension which allows us to use the interpolation
 estimate. Let $\tilde{u}_i \in H^2(\Omega)$ be a continuous extension of $u\in H^2(\Omega_i)$ to
 the complete domain $\Omega$. For this extension it holds
 \begin{equation}
   \norm{\tilde{u}_i-u}_{H^2(\Omega_i)} = 0, \quad \norm{\tilde{u}_i}_{H^2(\Omega)}\le
   c\norm{u}_{H^2(\Omega_i)}, \quad i = 1,2, 
   \label{extension}
 \end{equation}
 if the interface $\Gamma$ is smooth enough, see~\cite{Wloka1987}.
 For the remaining term we add and subtract the continuous extension $\tilde{u}$ and derive
 \begin{align}
 \norm{\nabla(u-L_hu)}_{S_h} &\le \norm{\nabla(u-\tilde{u})}_{S_h} 
 + \norm{\nabla(\tilde{u}-L_hu)}_{S_h}, \nonumber\\ 
 &= \norm{\nabla(u-\tilde{u})}_{S_h} 
 + \norm{\nabla(\tilde{u}-L_h\tilde{u})}_{S_h}\label{gradextension}
 \end{align}
 as for the nodal interpolant on $S_h$ it holds $L_h u = L_h\tilde{u}$. The continuous extension
 $\tilde{u}$ has enough regularity to apply~\eqref{interpolation} which gives us
 \begin{equation}
 \norm{\nabla(\tilde{u}-L_h\tilde{u})}_{S_h} \le ch\norm{\nabla^2\tilde{u}}_{S_h} \le ch
 \norm{\nabla^2\tilde{u}_i}_{\Omega} \le ch \norm{u}_{H^2(\Omega_1\cup\Omega_2)}.\label{Si}
 \end{equation}
 Here, we enlarged the domain from $S_i$ to $\Omega$ and used the continuity of the
 extension~\eqref{extension}.
 For the first term in~\eqref{gradextension}, we have
 \[
 \norm{\nabla (u -\tilde{u})}_{S_h} \le \norm{\tilde{u}_i-u}_{H^2(\Omega)} = 0,
 \]
 which completes the proof.
\end{proof}
\begin{theorem}\label{thmestimate}
  Let $\Omega \in \mathds{R}^2$ be a domain with convex polygonal boundary. We assume that the
  interface $\Gamma$ admits a $C^2$-parameterization and that it splits the domain into
  $\Omega = \Omega_1\cup\Gamma\cup\Omega_2$ such that the solution $u\in H^1_0(\Omega)$
  satisfies a stability estimate
  \[
  u \in H^1_0(\Omega) \cap H^2(\Omega_1\cup\Omega_2), \quad \norm{u}_{H^2(\Omega_1\cup\Omega_2)} \le
  c_s\norm{f}.
  \]
  Then the estimate for the adapted finite element solution $u_h \in V_h$ 
  \[
  \norm{\nabla(u-u_h)}_{\Omega} \le Ch\norm{f}, \quad \norm{u-u_h}_{\Omega} \le Ch^2\norm{f}
  \]
  holds.
\end{theorem}

\begin{proof}
  \begin{enumerate}
    \item We prove the first inequality $\norm{\nabla (u-u_h)}\le Ch\norm{f}$:\\
  For the error $e_h = u-u_h$ and for all $\varphi_h \in V_h$ it holds 
  \begin{align*}
    (\kappa\nabla e_h, \nabla\varphi_h)_\Omega
    &= \sum_{i=1}^2 (\kappa_i\nabla e_h,
    \nabla\varphi_h)_{\Omega_i} \\
    &= \sum_{i=1}^2 \{(\kappa_i\nabla u,
    \nabla\varphi_h)_{\Omega_i} - (\kappa_i\nabla u_h,
    \nabla\varphi_h)_{\Omega_i}\},
  \end{align*}
  Using $\Omega_i = (\mathcal{T}_{i,h} \setminus (\mathcal{T}_{i,h} \setminus \Omega_i)) \cup (\Omega_i
  \setminus \mathcal{T}_{i,h})$ and the relations
  \[
  \Omega_1\setminus\mathcal{T}_{1,h} = \mathcal{T}_{2,h}\setminus\Omega_2, \qquad
  \Omega_2\setminus\mathcal{T}_{2,h} = \mathcal{T}_{1,h}\setminus\Omega_1,
  \]
  results in
  \[
   \sum_{i=1}^2 (\kappa_i\nabla u_h, \nabla\varphi_h)_{\Omega_i}
   = \sum_{i=1}^2 (\kappa_i\nabla u_h, \nabla\varphi_h)_{\mathcal{T}_{i,h}}
   + \sum_{i=1}^2 (\delta\kappa_i\nabla u_h,
   \nabla\varphi_h)_{\Omega_i\setminus\mathcal{T}_{i,h}},
  \]
   with 
   \[
   \delta\kappa_i = \begin{cases} \kappa_1 -\kappa_2, & i = 1, \\
   \kappa_2 - \kappa_1, & i = 2. \end{cases}
   \]
   Taking~\eqref{contproblem} and~\eqref{discproblem} into account, it follows that
   \[
   \sum_{i=1}^2 (\kappa_i\nabla u, \nabla\varphi_h)_{\Omega_i} = \sum_{i=1}^2 (\kappa_i\nabla u_h,
   \nabla\varphi_h)_{\mathcal{T}_{i,h}}
   \]
   and thus,
   a perturbed Galerkin orthogonality
   \begin{equation}
   (\kappa\nabla e_h, \nabla\varphi_h)_\Omega = \sum_{i=1}^2 (\delta\kappa_i\nabla u_h,
   \nabla\varphi_h)_{\Omega_i\setminus\mathcal{T}_{i,h}},\label{Galerkin}
 \end{equation}
   holds.
   Estimating
   \begin{align*}
     \norm{\nabla e_h}^2\le 
     (\kappa \nabla
     e_h, \nabla e_h)
     &= \big(\kappa \nabla e_h, \nabla (u-\varphi_h)\big) + \big(\kappa \nabla e_h, \nabla
     (\varphi_h - u_h)\big),
   \end{align*}
 and picking the Lagrangian interpolant $\varphi_h = L_h u \in V_h$ yields
   \begin{align*}
\norm{\nabla e_h}^2 
     \le c&\norm{\nabla e_h}\norm{\nabla(u-L_h u)} \\&+ \sum_{i=1}^2 \norm{\delta\kappa_i\nabla
     u_h}_{\Omega_i\setminus\mathcal{T}_{i,h}}
     \norm{\nabla (L_h u -u_h)}_{\Omega_i\setminus\mathcal{T}_{i,h}}.
   \end{align*}
   By Lemma~\ref{region}, we can bound the terms on the convex remainders as
   \begin{align*}
     \norm{\delta\kappa_i\nabla
     u_h}_{\Omega_i\setminus\mathcal{T}_{i,h}} &\le 
     c_\kappa h^{\nicefrac12}\norm{\nabla u_h}_{\Omega},\\
     \norm{\nabla (L_h u -u_h)}_{\Omega_i\setminus\mathcal{T}_{i,h}} &\le 
     ch^{\nicefrac12}\norm{\nabla (L_h u -u_h)}_{\Omega},
   \end{align*} 
   and arrive at
   \begin{align}
\norm{\nabla e_h}^2
\le c\norm{\nabla e_h}\norm{\nabla(u-L_h u)} + c h\norm{\nabla u_h}\norm{\nabla (u_h-L_h u)}.
   \end{align}
   Applying Young's inequality $ab \le \frac{a^2}{2\varepsilon} + \frac{\varepsilon b^2}{2}$
   and adding and subtracting $u$ results in
   \[
   \norm{\nabla e_h}
   \le c\norm{\nabla(u-L_h u)} + c h\norm{\nabla u}.
   \]
   That leaves us to estimate the interpolation error on the whole domain $\Omega.$
   We divide the estimation into the cells which are affected by the interface and those which are not affected. 
   By Lemma~\ref{lemmainterpolation}, we have an estimate for cells $T\in S_h$:
   \[
    \norm{\nabla(L_hu -u)}_{S_h} \le ch\norm{u}_{H^2(\Omega_1\cup\Omega_2)}.
   \]
   On
   those cells not cut by the interface, we use the standard interpolation estimate and extend the
   domain to the complete domain again:
   \begin{align}
   \norm{\nabla (u -L_h u)}_{\Omega}^2 &= \norm{\nabla (u -L_h u)}_{\Omega\setminus S_h}^2 +
   \norm{\nabla (u -L_h u)}_{S_h}^2\nonumber\\
   &\le ch^2 \norm{u}_{H^2(\Omega_1 \cup \Omega_2)}^2.\label{Lh}
 \end{align}
 With the stability estimate, the first estimate
 \begin{equation}
 \norm{\nabla(u-u_h)}_\Omega \le Ch\norm{f}_\Omega\label{energyerror}
 \end{equation}
 is proven.
    \item We prove the second inequality $\norm{u-u_h}\le Ch^2\norm{f}$:\\
 To show the estimate for the $L^2$-error, we apply a standard duality argument. Therefore, let 
 $z\in H^1_0(\Omega)$ be the solution of the adjoint problem
 \[
 \sum_{i=1}^2(\kappa_i\nabla\varphi,\nabla z) = (e_h,\varphi)\norm{e_h}^{-1}
 \]
 for all $\varphi\in H^1_0(\Omega)$. For the dual solution it holds $z\in H^1_0(\Omega) \cup
 H^2(\Omega_1\cup\Omega_2)$ and $\norm{z}_{H^2(\Omega_1\cup\Omega_2)}\le c_s$. Using the perturbed Galerkin orthogonality~\eqref{Galerkin}, we obtain
 \begin{align*}
 \norm{e_h} &= (e_h,e_h)\norm{e_h}^{-1} = (\kappa \nabla e_h, \nabla z)\\
 &= (\kappa \nabla e_h,
 \nabla(z-L_hz)) + \sum_{i=1}^2(\delta\kappa_i\nabla u_h, \nabla
 L_hz)_{\Omega_i\setminus\mathcal{T}_{i,h}},
 \end{align*}
 and we estimate
 \begin{equation}
 \norm{e_h} \le c\norm{\nabla e_h}\norm{\nabla(z-L_hz)} + c\sum_{i=1}^2\norm{\delta\kappa_i\nabla
 u_h}_{\Omega_i\setminus\mathcal{T}_{i,h}}\norm{\nabla L_h
 z}_{\Omega_i\setminus\mathcal{T}_{i,h}}.\label{L2error}
 \end{equation}
 By adding and subtracting $u$ and $L_hu$, we find that
\begin{align*}
  \norm{\nabla u_h}_{\Omega_i\setminus\mathcal{T}_{i,h}} &\le \norm{\nabla
  u}_{\Omega_i\setminus\mathcal{T}_{i,h}} + \norm{\nabla
  (u-L_hu)}_{\Omega_i\setminus\mathcal{T}_{i,h}} + \norm{\nabla (L_h u -
  u_h)}_{\Omega_i\setminus\mathcal{T}_{i,h}}.
\end{align*}
 With Lemma~\ref{region}, it holds
 \[
 \norm{\nabla u}_{\Omega_i\setminus\mathcal{T}_{i,h}}\le ch \norm{u}_{H^2(\Omega_1\cup\Omega_2)},
 \]
 and thus, using Lemma~\ref{region} and~\eqref{Lh}, it follows that
 \[
  \norm{\nabla u_h}_{\Omega_i\setminus\mathcal{T}_{i,h}} 
\le ch\norm{u}_{H^2(\Omega_1\cup\Omega_2)} + ch^{\nicefrac12}\norm{\nabla (L_h u - u_h)}_{\Omega}.
 \]
 After adding and subtracting $u$ again, and employing~\eqref{Lh} and~\eqref{energyerror}, we derive
\begin{align*}
\norm{\nabla u_h}_{\Omega_i\setminus\mathcal{T}_{i,h}} &\le ch\norm{u}_{H^2(\Omega_1\cup\Omega_2)} +
ch^{\nicefrac12}\norm{\nabla ( u - L_hu)}_{\Omega} + ch^{\nicefrac12}\norm{\nabla (e_h)}_{\Omega},\\
&\le ch \norm{f}_\Omega. 
\end{align*}
Similarly, by adding and subtracting the dual solution $z$, we find that for the interpolation of the dual
solution it holds
\[
\norm{\nabla L_h z}_{\Omega_i\setminus\mathcal{T}_{i,h}} \le \norm{\nabla
z}_{\Omega_i\setminus\mathcal{T}_{i,h}} + \norm{\nabla(z-L_hz)}_{\Omega_i\setminus\mathcal{T}_{i,h}}\le
ch\norm{z}_{H^2(\Omega_1\cup\Omega_2)}.
\]
Using~\eqref{Lh},~\eqref{energyerror}, the stability estimate
$\norm{z}_{H^2(\Omega_1\cup\Omega_2)}\le c_s$, and~\eqref{L2error} we
deduce the estimate
\begin{align*}
  \norm{u-u_h}_\Omega 
&\le Ch^2\norm{f}_\Omega.
\end{align*}
\end{enumerate}
\end{proof}

\section{Numerical examples}
\label{numerics}
In this section, we present three different numerical test cases which are chosen to numerically
also show the analytically proven convergence. All test are taken from~\cite{FreiRichter14} and
results can directly be compared. 

We also test the behavior of
the presented method dependent on how we choose the free parameters $r,q,s$ in the cells which are
cut by the interface. Note that only for the case that two edges are cut by the interface, we
proposed different strategies to choose the free parameter. We will show results for three different
choices of parameters, namely:

\begin{description}
  \item[Strategy 1:] Choose the free parameters to be $\nicefrac12$, also if a vertex is
    cut.
  \item[Strategy 2:] Choose the free parameters to be $r = 1-s$, $q = s$, and $s =
    1-r$.
  \item[Strategy 3:] Choose the free parameters to be $r = (1-s)(1-q)$, $q = (1-r)s$, and $s
    = qr$.
\end{description}

All computations in this paper are done with the initial mesh shown in Figure~\ref{initial_mesh}.
\begin{figure}[h]
  \centering
  \begin{subfigure}[b]{.3\textwidth}
  \centering
  \includegraphics[width=\textwidth]{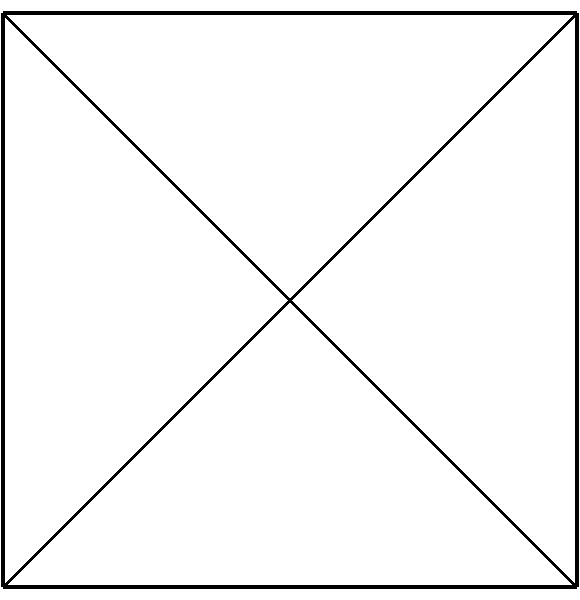}
  \caption{Initial mesh used for all computations.}\label{initial_mesh}
\end{subfigure}%
\hspace{2mm}
  \begin{subfigure}[b]{.66\textwidth}
  \centering
  \includegraphics[width=.95\textwidth]{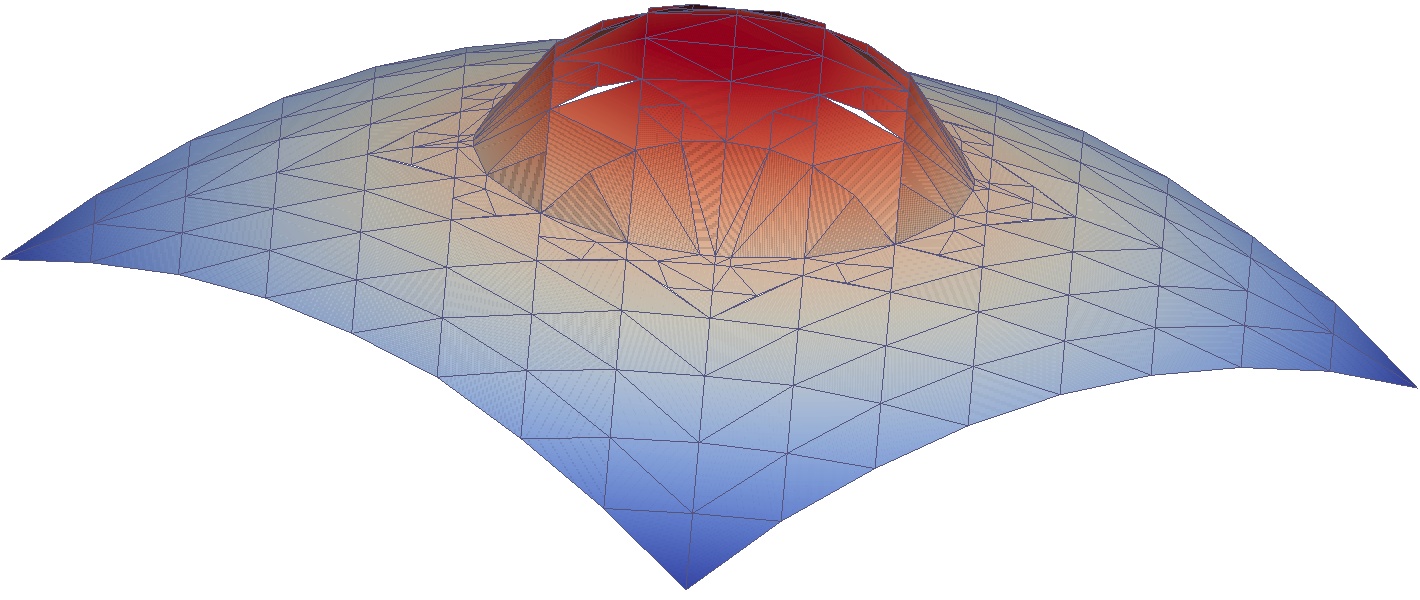}
  \caption{Finite element approximation of~\eqref{circular_u} with the presented method after several global refinements
  of the initial mesh.}\label{solution}
\end{subfigure}%
\caption{Initial mesh~\subref{initial_mesh} and approximated solution for a circular
interface~\subref{solution}.}
\end{figure}

\subsection{Circular interface}\label{circular_interface}
This example was used to compute the results for Figure~\ref{conv} in Section~\ref{intro}.
We choose an analytical solution to the interface problem~\eqref{laplace} as
\begin{equation}
u(x) = \begin{cases} -2 \kappa_2
  \norm{x}^4,& x \in \Omega_1,\\-\kappa_1\norm{x}^2+\frac14\kappa_1-\frac18\kappa_2,& x\in
  \Omega_2\end{cases}\label{circular_u}
\end{equation}
and compute the right hand side and boundary conditions accordingly. The domains are given as
\begin{align*}
  \Omega_1 &= \{x \in \mathds{R}^2: \norm{x}<\nicefrac14\},\\
  \Omega_2 &= (-1,1)^2\setminus \Omega_1,
\end{align*}
and the diffusion coefficient is defined as 
\[
\kappa = \begin{cases} 0.1, &x\in \Omega_1\\ 1, &x\in\Omega_2\end{cases}.
\]
A sketch of a finite element approximation with the locally adapted patch method is given in
Figure~\ref{solution}. We note that in contrast to~\cite{FreiRichter14}, we only split the cut cells into subtriangles for visualization. Thus, there are hanging nodes in the visualization. However, the solution is
continuous at these edges.
For the presented adapted finite element method we plot the error in the $L^2$- and in the
$H^1$-norm for several levels of global refinement of the mesh in Figure~\ref{figoptcon}. We recover
the optimal quadratic convergence for the error in the $L^2$-norm and linear convergence in the
$H^1$-norm as proven in Theorem~\ref{thmestimate}.
\begin{figure}[h]
  \centering
  \resizebox{\textwidth}{!}{
  \input{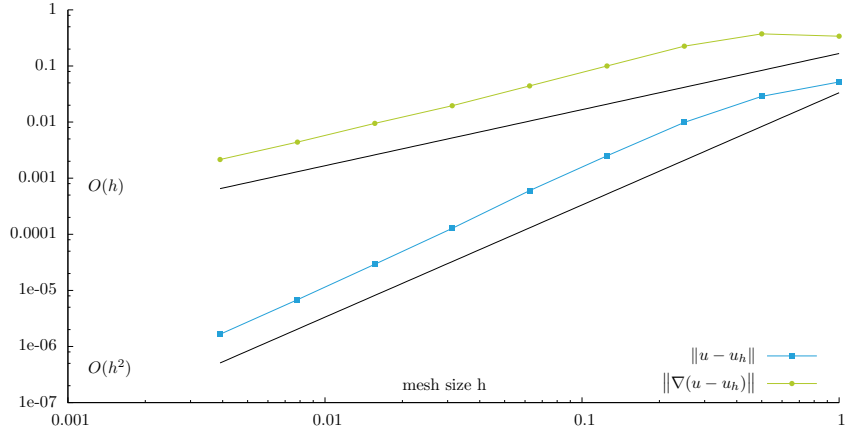}
  }
\caption{Optimal convergence recovered with the adapted finite element method. Solid lines indicate
a slope of $h$ and $h^2$, respectively.}\label{figoptcon}
\end{figure}
In Figure~\ref{different_cuts}, we present three meshes which arise in the approximation
of~\eqref{circular_u}. The subtriangles of the cells which are cut by the interface are shown. In the
mesh in Figure~\ref{no_adj}, the free parameters which are not determined by the cut of the interface
are chosen to be $\nicefrac12$. Therefore, the maximum angle condition could be violated. For the other
two meshes in Figures~\ref{with_adj} and~\ref{with_adj_prod}, we apply two different strategies for
choosing the free parameters. In those meshes, the maximum angle condition will be satisfied.
\begin{figure}[h]
  \centering
  \begin{subfigure}[b]{.3\textwidth}
  \centering
  \includegraphics[width=\textwidth]{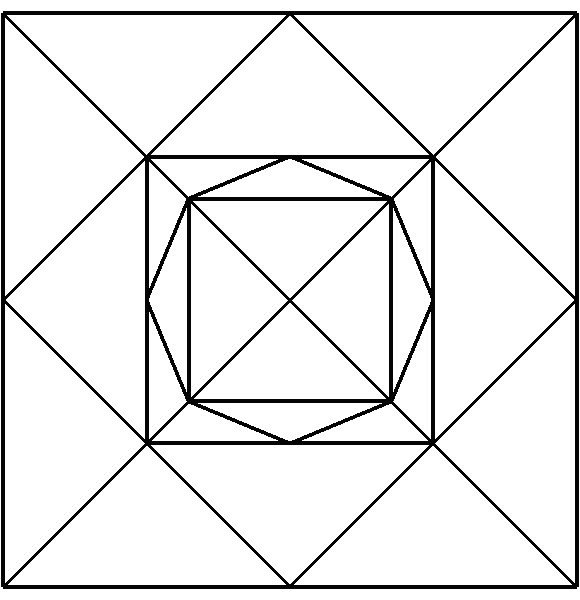}
  \caption{Strategy 1.}\label{no_adj}
\end{subfigure}%
\hspace{2mm}
  \begin{subfigure}[b]{.3\textwidth}
  \centering
  \includegraphics[width=\textwidth]{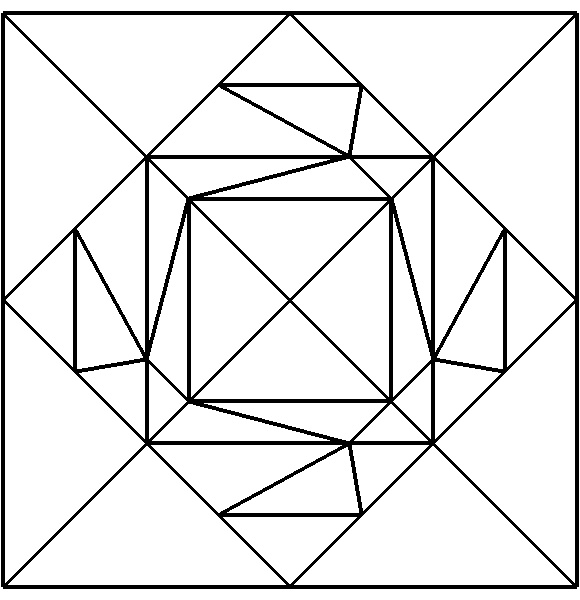}
  \caption{Strategy 2.}\label{with_adj}
\end{subfigure}%
\hspace{2mm}
  \begin{subfigure}[b]{.3\textwidth}
  \centering
  \includegraphics[width=\textwidth]{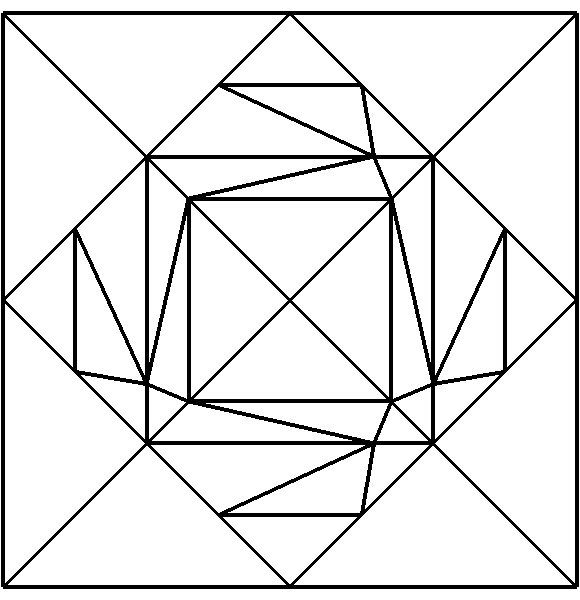}
  \caption{Strategy 3.}\label{with_adj_prod}
\end{subfigure}%
\caption{Cut cells for the approximation of~\eqref{circular_u} after one level of global refinement of the initial mesh in
Figure~\ref{initial_mesh}, free parameters chosen after different strategies.}\label{different_cuts}
\end{figure}
\subsection{Horizontal cut}
In this example we study the behavior of the locally adapted patch finite element method for patch triangles which get
very anisotropic. Let us assume that the domain $\Omega = (-1,1)^2$ is cut horizontally into
\[
\Omega_1(\varepsilon) = \{x\in \Omega \vert x_2 < \varepsilon h\}\quad  \text{ and }\quad
\Omega_2(\varepsilon) = \{x\in \Omega \vert x_2 > \varepsilon h\}, 
\]
where $h$ is the maximum edge length of a triangle $T$.
The analytical solution to the interface problem~\eqref{laplace} is chosen as
\[
u(x) = \begin{cases} \frac{\kappa_2}{\kappa_1} (x_2 -\varepsilon h) - (x_2 -\varepsilon h)^2,& x \in \Omega_1,\\
  (x_2 -\varepsilon h) - (x_2 -\varepsilon h)^2,& x\in
  \Omega_2.\end{cases}
\]
\begin{figure}[h]
  \centering
  \resizebox{\textwidth}{!}{
  \input{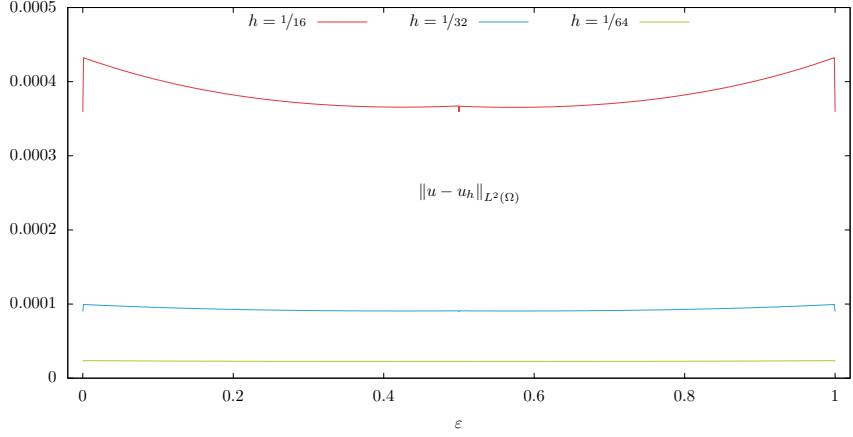}
  }
\caption{Error in the $L^2$-norm error depending on $\varepsilon$, strategy 1.}\label{horizontalL2}
\end{figure}

The right hand side and the boundary conditions are computed accordingly. The vertical edge length
in a patch triangulation will vary due to the value of $\varepsilon$. Vertical edge lengths between
$\varepsilon h$ and $(1-\varepsilon)h$ will occur for $\varepsilon \in [0,1]$. Here, $h$ denotes the
maximum edge length of a triangle $T$.

\begin{figure}[h]
  \centering
  \resizebox{\textwidth}{!}{
  \input{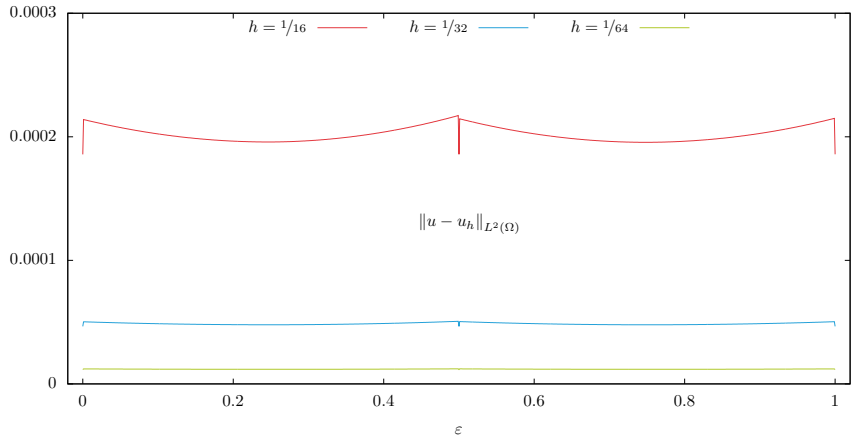}
  }
\caption{Error in the $L^2$-norm depending on $\varepsilon$, strategy 2.}\label{horizontalL2_cross}
\end{figure}

We plot the error in the $L^2$- and $H^1$-norm for different values of $\varepsilon\in[0,1]$ in
Figures~\ref{horizontalL2} to~\ref{horizontalH1}. As expected and already reported
in~\cite{FreiRichter14}, we observe the smallest errors for $\varepsilon = 0$, $\varepsilon =
\nicefrac12$, and $\varepsilon = 1$, as then the cut is resolved by the mesh. 

\begin{figure}[h]
  \centering
  \resizebox{\textwidth}{!}{
  \input{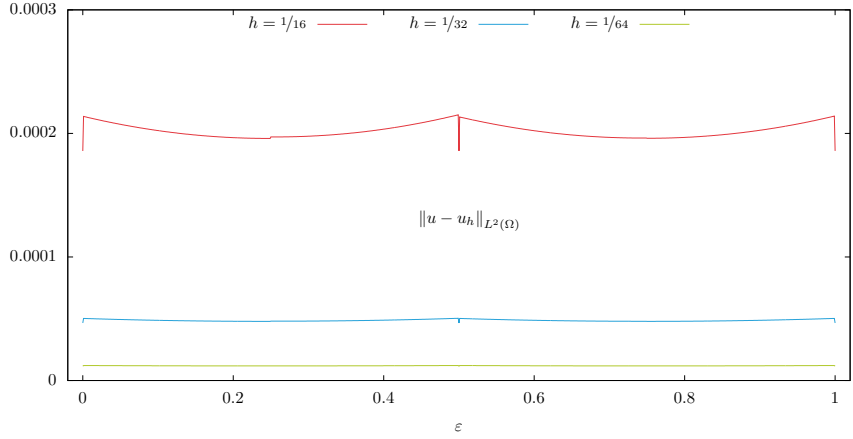}
  }
\caption{Error in the $L^2$-norm depending on $\varepsilon$, strategy 3.}\label{horizontalL2_cross_prod}
\end{figure}

The largest errors
arise for the cases in which $\varepsilon \to 0$ and $\varepsilon \to 1$ since then, the
anisotropies for the patch cells become maximal. However, the errors stay bounded. This behavior is similar on finer meshes but the variations get smaller.

\begin{figure}[h]
  \centering
  \resizebox{\textwidth}{!}{
  \input{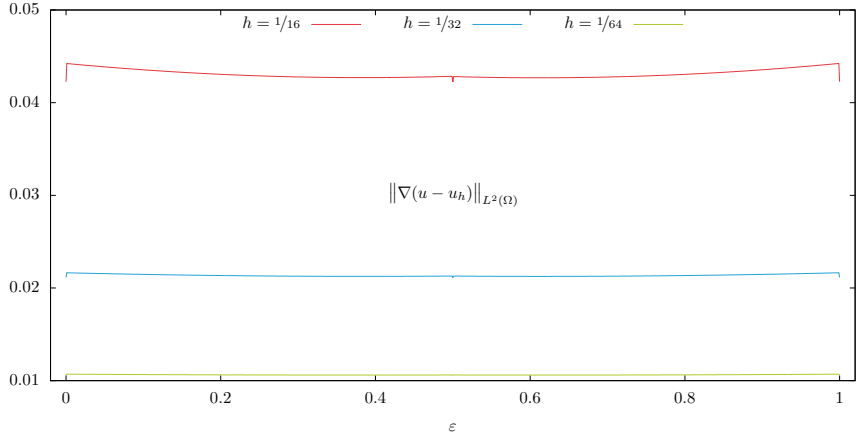}
  }
\caption{Error in the $H^1$-norm depending on $\varepsilon$, strategy 1.}
\end{figure}

\begin{figure}[h]
  \centering
  \resizebox{\textwidth}{!}{
  \input{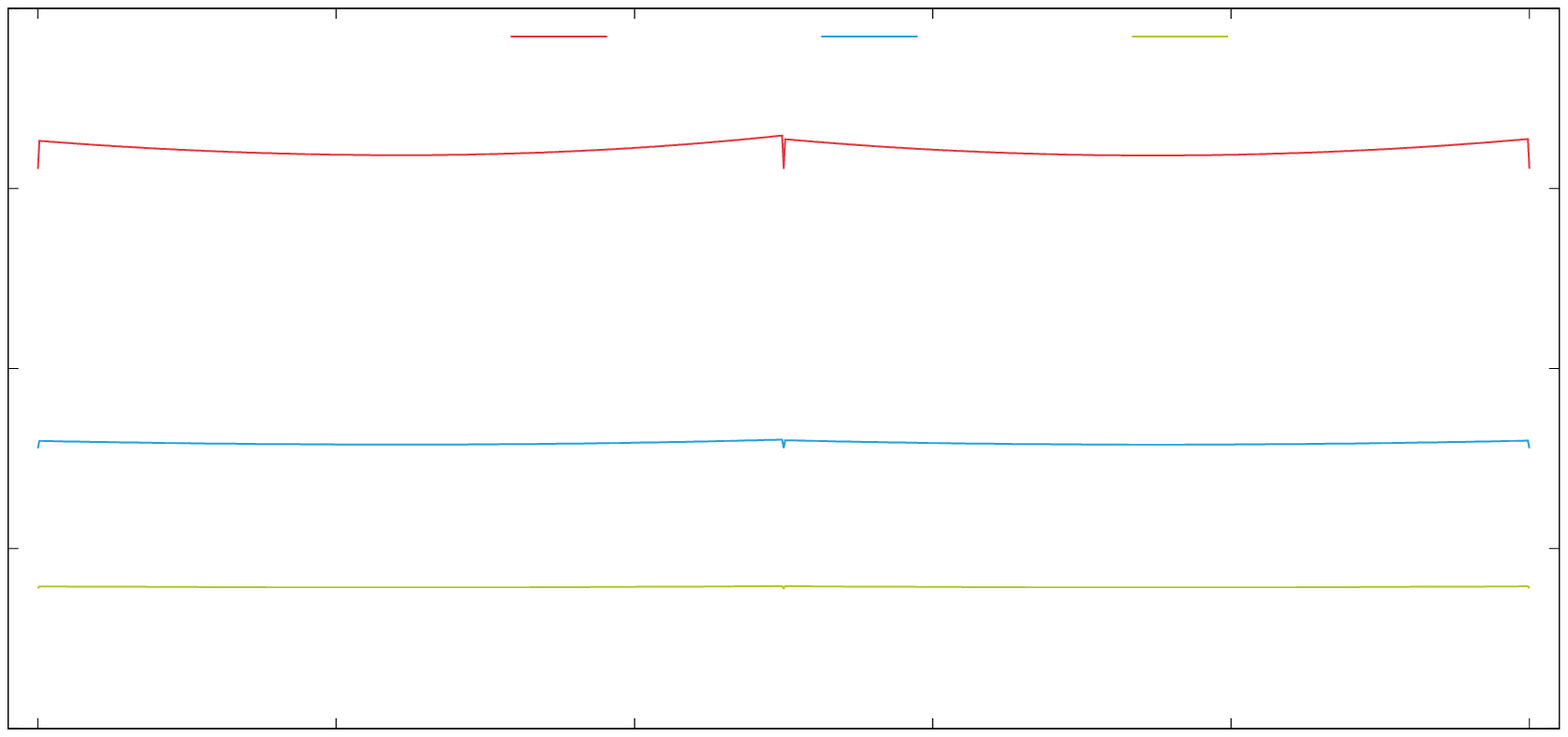}
  }
\caption{Error in the $H^1$-norm error depending on $\varepsilon$, strategy 2.}\label{horizontalH1_cross}
\end{figure}

\begin{figure}[h]
  \centering
  \resizebox{\textwidth}{!}{
  \input{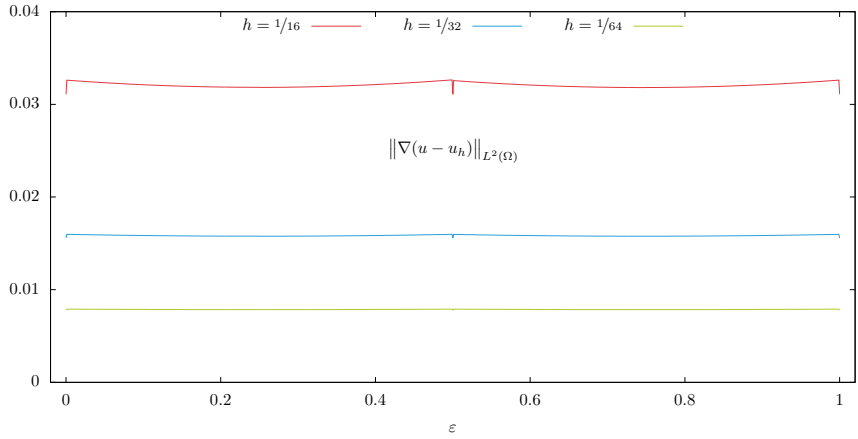}
  }
\caption{Error in the $H^1$-norm depending on $\varepsilon$, strategy 3.}\label{horizontalH1}
\end{figure}

The errors are smaller if we choose the free parameters such that the maximum angle condition will
be satisfied. However, the error behaves quite similar for all three strategies to choose the free
parameters. There is no significant difference between the error in the $L^2$-norm shown in
Figures~\ref{horizontalL2_cross} and~\ref{horizontalL2_cross_prod} and the error in the $H^1$-norm shown in 
Figures~\ref{horizontalH1_cross} and~\ref{horizontalH1}. For this example, we conclude that it
matters to choose the free parameter to satisfy a maximum angle condition. However, it does not seem to matter which of the suggested values to choose.

\subsection{Tilted interface line}
In this numerical example, we consider a straight interface which cuts the domain $\Omega = (-1,1)^2$ in two
subdomains
\[
\Omega_1(\alpha) = \{x\in\Omega\vert \cos(\alpha)x_2<\sin(\alpha)x_1\},
\]
and
\[
\Omega_2(\alpha) = \{x\in\Omega\vert \cos(\alpha)x_2>\sin(\alpha)x_1\}.
\]
Depending on the value of $\alpha$, the straight interface is defined by
$\cos(\alpha)x_2=\sin(\alpha)x_1$. We choose an analytical solution to problem~\eqref{laplace} as
\[
u(x) = \begin{cases} \sin\big(\frac{\kappa_2}{\kappa_1}(\cos(\alpha)x_2-\sin(\alpha)x_1)\big), & x \in
  \Omega_1,\\
\sin(\cos(\alpha)x_2-\sin(\alpha)x_1), & x \in \Omega_2.
\end{cases}
\]
From this exact solution, the right hand side and boundary conditions are computed accordingly.
We plot the error in the $L^2$- and $H^1$-norm for different values of $\alpha\in[0,\pi]$ in
Figures~\ref{tiltedL2} to~\ref{tiltedH1}.
\begin{figure}[h]
  \centering
  \resizebox{\textwidth}{!}{
  \input{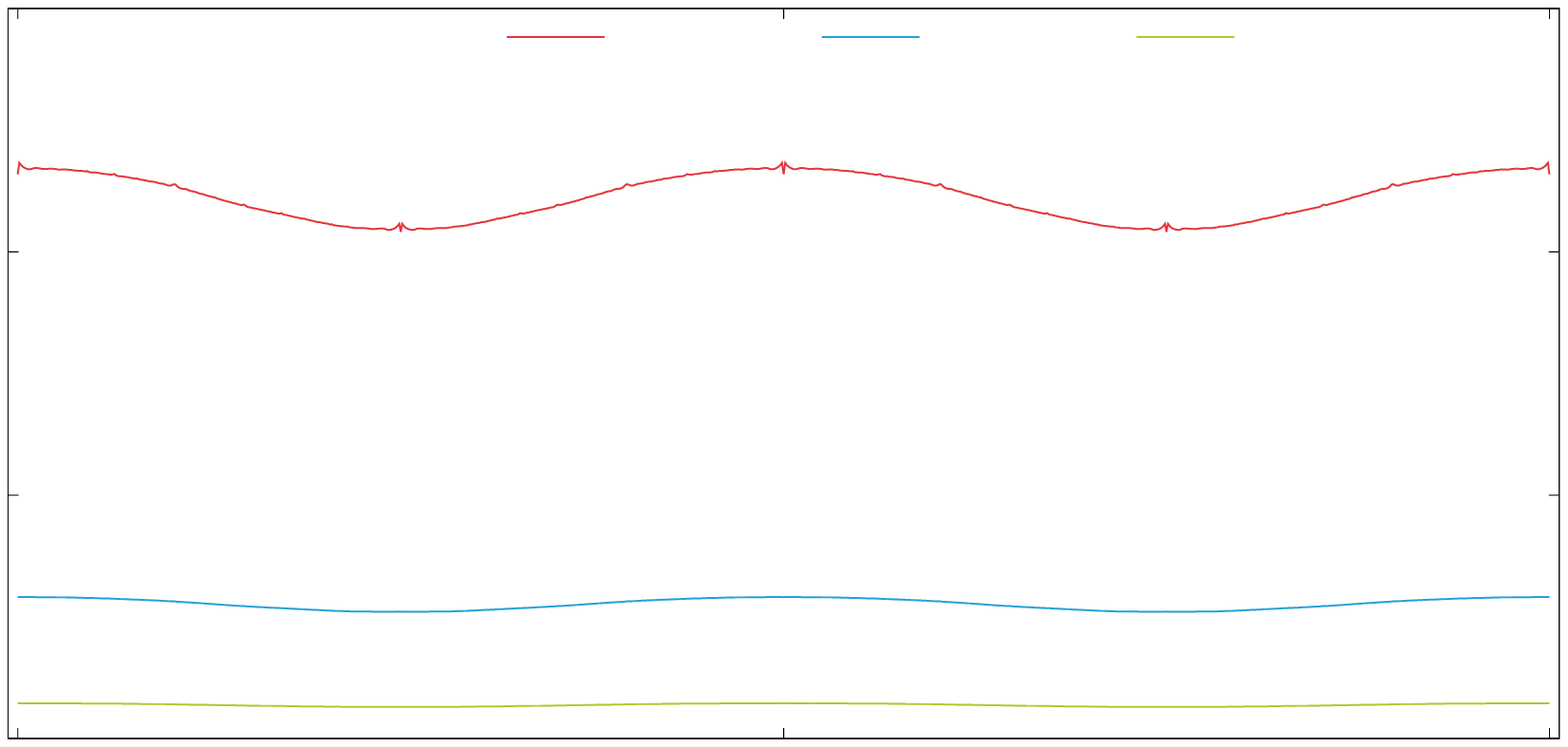}
  }
\caption{Error in the $L^2$-norm depending on $\alpha$, strategy 1.}\label{tiltedL2}
\end{figure}
\begin{figure}[h]
  \centering
  \resizebox{\textwidth}{!}{
  \input{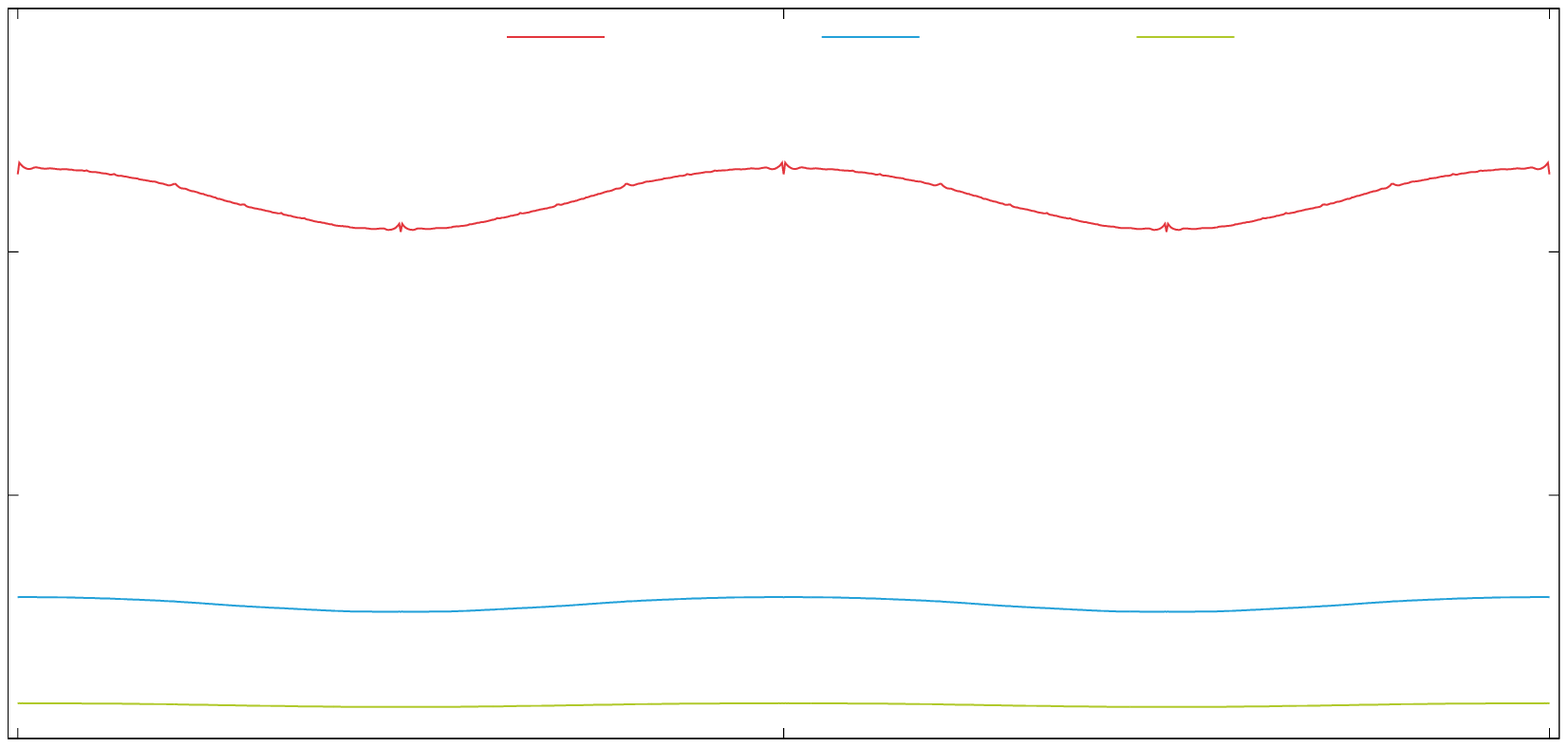}
  }
\caption{Error in the $L^2$-norm depending on $\alpha$, strategy 2.}
\end{figure}
\begin{figure}[h]
  \centering
  \resizebox{\textwidth}{!}{
  \input{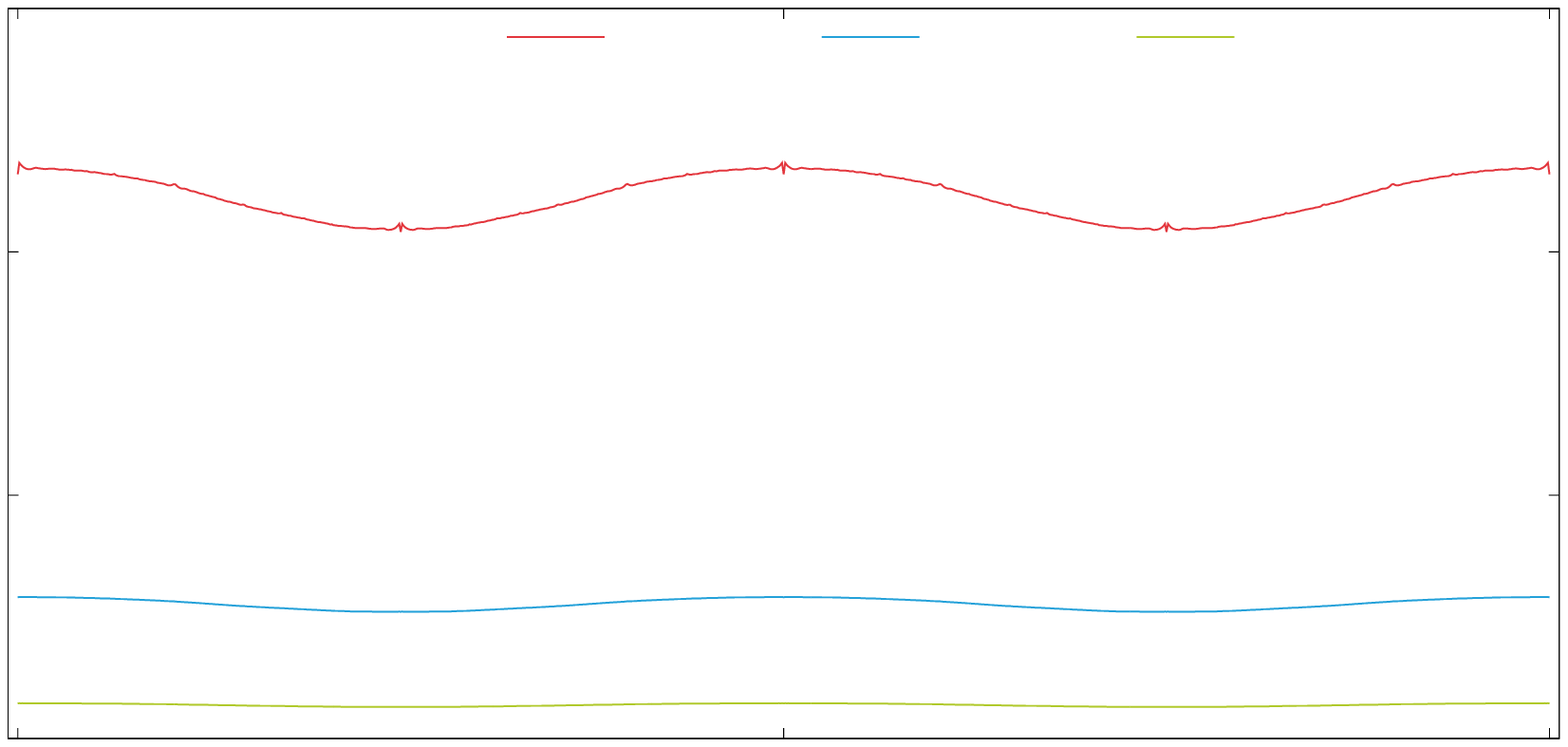}
  }
\caption{Error in the $L^2$-norm depending on $\alpha$, strategy 3.}
\end{figure}

\begin{figure}[h]
  \centering
  \resizebox{\textwidth}{!}{
  \input{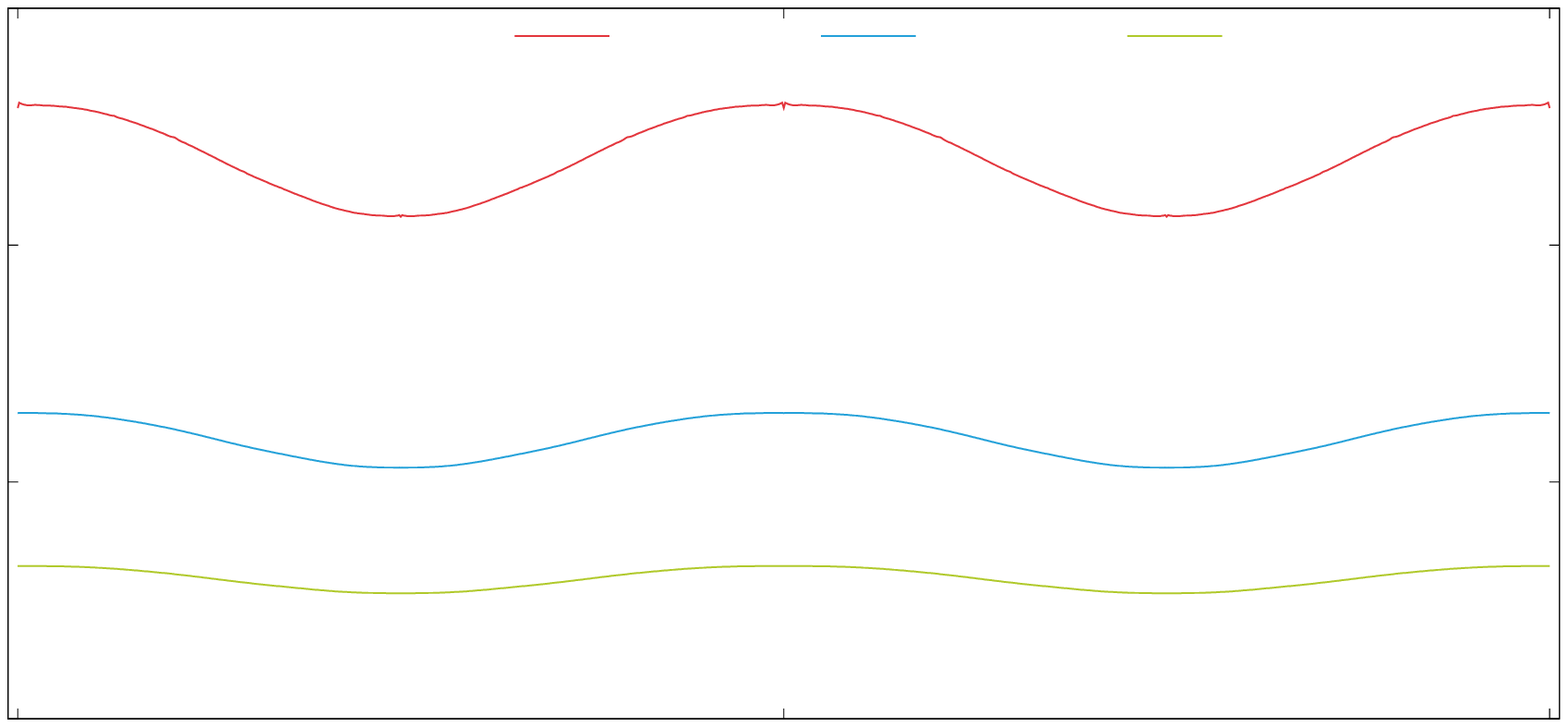}
  }
\caption{Error in the $H^1$-norm depending on $\alpha$, strategy 1.}
\end{figure}
\begin{figure}[h]
  \centering
  \resizebox{\textwidth}{!}{
  \input{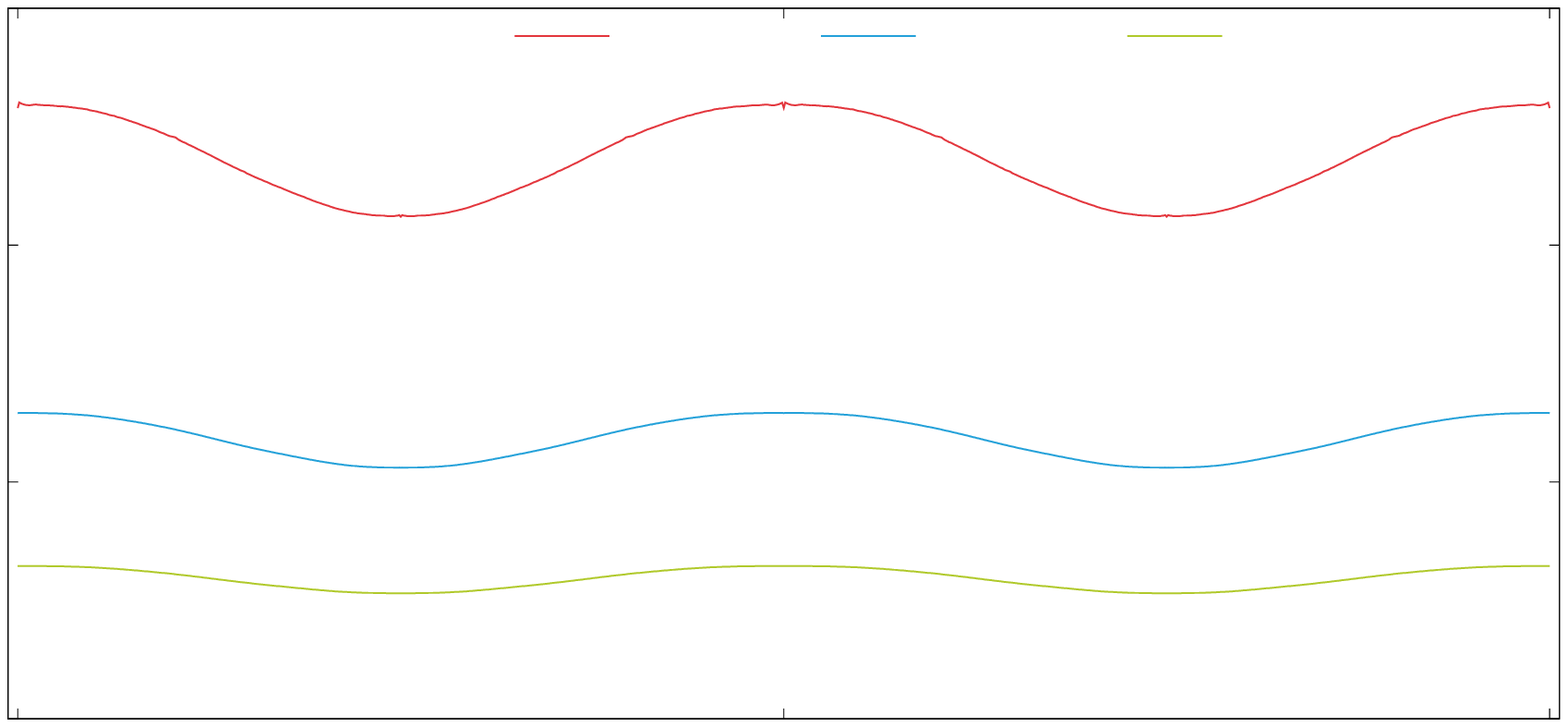}
  }
\caption{Error in the $H^1$-norm depending on $\alpha$, strategy 2.}
\end{figure}

\begin{figure}[h]
  \centering
  \resizebox{\textwidth}{!}{
  \input{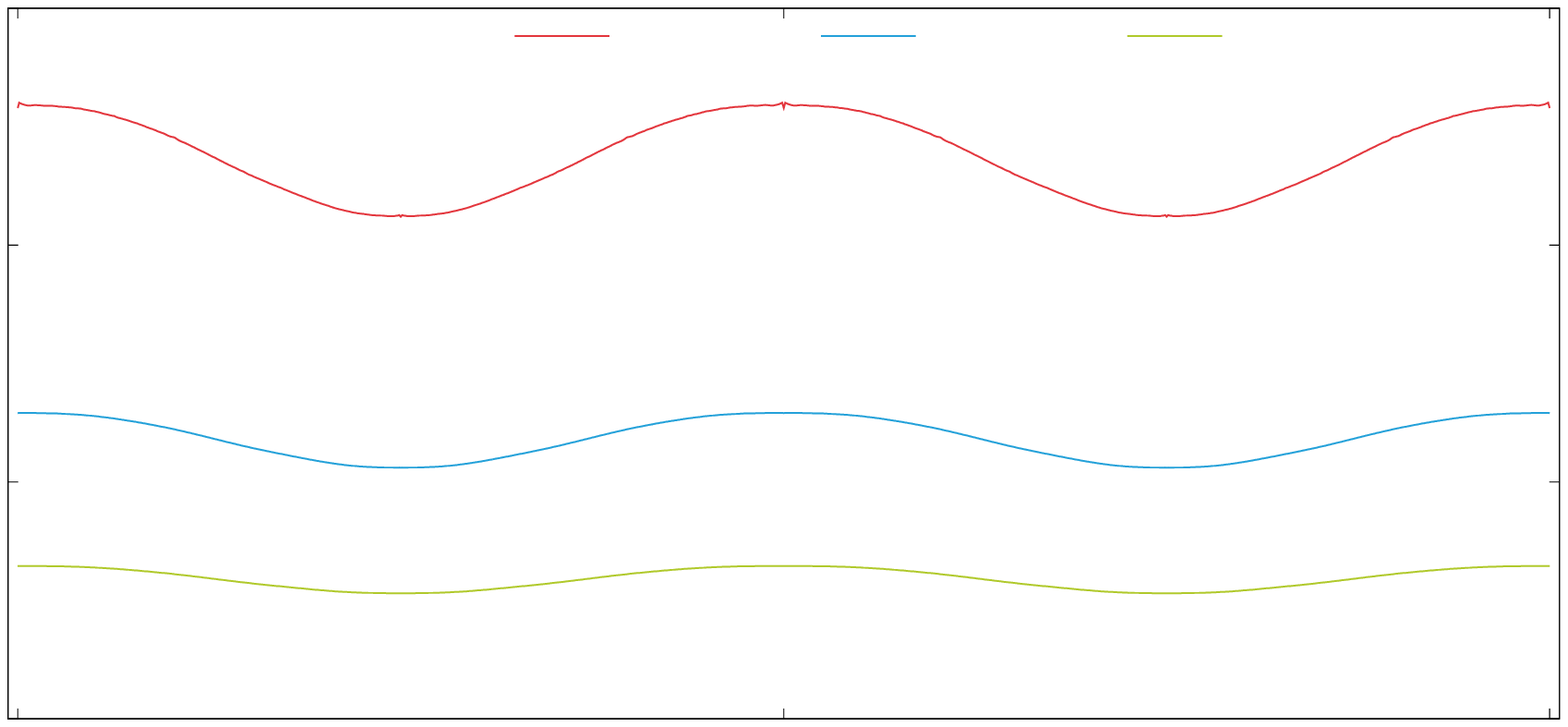}
  }
\caption{Error in the $H^1$-norm depending on $\alpha$, strategy 3.}\label{tiltedH1}
\end{figure}
In this example, the choice of the free parameters does not make any difference. In contrast to the
previous example we even get the same behavior of the error if we choose the free parameters to be
$\nicefrac12$. 
\section{Conclusions}
\label{conclusion}
We presented an extension of the locally adapted patch finite element method to triangles. In order
to resolve a solution at an interface, patch triangles which are cut by the interface are divided
into subtriangles. Linear polynomials on these subtriangles together with a local quadrature rule
are applied. 

Thanks to linear transformations of the subtriangles to the exact location of the cut in the patch
triangle, no degrees of freedom have to be moved, inserted or deleted. Thus, this method is
perfectly suitable for parallel computing on distributed architectures. 

A key ingredient for the a priori error analysis of the locally adapted finite element patch method was the maximum angle condition. In contrast to the
locally adapted patch finite element method on quadrilaterals, the maximum angle condition is not
automatically satisfied on triangular meshes. We suggested various possibilities to overcome this
shortcoming. With these adjustments, the optimal order of convergence for the error in the
$L^2$- and the $H^1$-norm are recovered. This is also shown in the numerical experiments.
Furthermore, we observe in the numerical tests that these
adjustments are not necessarily leading to a better behavior for the error.   

In future research, we plan to parallelize the method and to extend it to three dimensional
tetrahedral meshes. Furthermore, we would like to apply our method to time-dependent and more complex
problems such as fluid-structure interaction problems.

\section*{Acknowledgement}
The authors would like to thank Aur\'elien Larcher for the implementation of higher
order finite elements in the HPC branch of DOLFIN. Without his contribution, this work would not have been possible.

\bibliographystyle{abbrv}
\bibliography{paper}

\end{document}

%% file: results/circular_interface/plot_unopt_err.tex
\begingroup
  \makeatletter
  \providecommand\color[2][]{%
    \GenericError{(gnuplot) \space\space\space\@spaces}{%
      Package color not loaded in conjunction with
      terminal option `colourtext'%
    }{See the gnuplot documentation for explanation.%
    }{Either use 'blacktext' in gnuplot or load the package
      color.sty in LaTeX.}%
    \renewcommand\color[2][]{}%
  }%
  \providecommand\includegraphics[2][]{%
    \GenericError{(gnuplot) \space\space\space\@spaces}{%
      Package graphicx or graphics not loaded%
    }{See the gnuplot documentation for explanation.%
    }{The gnuplot epslatex terminal needs graphicx.sty or graphics.sty.}%
    \renewcommand\includegraphics[2][]{}%
  }%
  \providecommand\rotatebox[2]{#2}%
  \@ifundefined{ifGPcolor}{%
    \newif\ifGPcolor
    \GPcolortrue
  }{}%
  \@ifundefined{ifGPblacktext}{%
    \newif\ifGPblacktext
    \GPblacktexttrue
  }{}%
  \let\gplgaddtomacro\g@addto@macro
  \gdef\gplbacktext{}%
  \gdef\gplfronttext{}%
  \makeatother
  \ifGPblacktext
    \def\colorrgb#1{}%
    \def\colorgray#1{}%
  \else
    \ifGPcolor
      \def\colorrgb#1{\color[rgb]{#1}}%
      \def\colorgray#1{\color[gray]{#1}}%
      \expandafter\def\csname LTw\endcsname{\color{white}}%
      \expandafter\def\csname LTb\endcsname{\color{black}}%
      \expandafter\def\csname LTa\endcsname{\color{black}}%
      \expandafter\def\csname LT0\endcsname{\color[rgb]{1,0,0}}%
      \expandafter\def\csname LT1\endcsname{\color[rgb]{0,1,0}}%
      \expandafter\def\csname LT2\endcsname{\color[rgb]{0,0,1}}%
      \expandafter\def\csname LT3\endcsname{\color[rgb]{1,0,1}}%
      \expandafter\def\csname LT4\endcsname{\color[rgb]{0,1,1}}%
      \expandafter\def\csname LT5\endcsname{\color[rgb]{1,1,0}}%
      \expandafter\def\csname LT6\endcsname{\color[rgb]{0,0,0}}%
      \expandafter\def\csname LT7\endcsname{\color[rgb]{1,0.3,0}}%
      \expandafter\def\csname LT8\endcsname{\color[rgb]{0.5,0.5,0.5}}%
    \else
      \def\colorrgb#1{\color{black}}%
      \def\colorgray#1{\color[gray]{#1}}%
      \expandafter\def\csname LTw\endcsname{\color{white}}%
      \expandafter\def\csname LTb\endcsname{\color{black}}%
      \expandafter\def\csname LTa\endcsname{\color{black}}%
      \expandafter\def\csname LT0\endcsname{\color{black}}%
      \expandafter\def\csname LT1\endcsname{\color{black}}%
      \expandafter\def\csname LT2\endcsname{\color{black}}%
      \expandafter\def\csname LT3\endcsname{\color{black}}%
      \expandafter\def\csname LT4\endcsname{\color{black}}%
      \expandafter\def\csname LT5\endcsname{\color{black}}%
      \expandafter\def\csname LT6\endcsname{\color{black}}%
      \expandafter\def\csname LT7\endcsname{\color{black}}%
      \expandafter\def\csname LT8\endcsname{\color{black}}%
    \fi
  \fi
  \setlength{\unitlength}{0.0500bp}%
  \begin{picture}(10800.00,5400.00)%
    \gplgaddtomacro\gplbacktext{%
      \csname LTb\endcsname%
      \put(1122,440){\makebox(0,0)[r]{\strut{} 0.0001}}%
      \put(1122,1614){\makebox(0,0)[r]{\strut{} 0.001}}%
      \put(1122,2788){\makebox(0,0)[r]{\strut{} 0.01}}%
      \put(1122,3961){\makebox(0,0)[r]{\strut{} 0.1}}%
      \put(1122,5135){\makebox(0,0)[r]{\strut{} 1}}%
      \put(1254,220){\makebox(0,0){\strut{} 0.001}}%
      \put(4304,220){\makebox(0,0){\strut{} 0.01}}%
      \put(7353,220){\makebox(0,0){\strut{} 0.1}}%
      \put(10403,220){\makebox(0,0){\strut{} 1}}%
      \put(1791,2994){\makebox(0,0)[l]{\strut{}$O(h^{\nicefrac12})$}}%
      \put(1791,647){\makebox(0,0)[l]{\strut{}$O(h)$}}%
      \put(1791,4315){\makebox(0,0)[l]{\strut{}$\norm{\nabla(u-u_h)}$}}%
      \put(1791,1614){\makebox(0,0)[l]{\strut{}$\norm{u-u_h}$}}%
      \put(5222,793){\makebox(0,0)[l]{\strut{}mesh size h}}%
    }%
    \gplgaddtomacro\gplfronttext{%
      \csname LTb\endcsname%
      \put(9416,833){\makebox(0,0)[r]{\strut{}P1-elements}}%
      \csname LTb\endcsname%
      \put(9416,613){\makebox(0,0)[r]{\strut{}P2-elements}}%
    }%
    \gplbacktext
    \put(0,0){\includegraphics{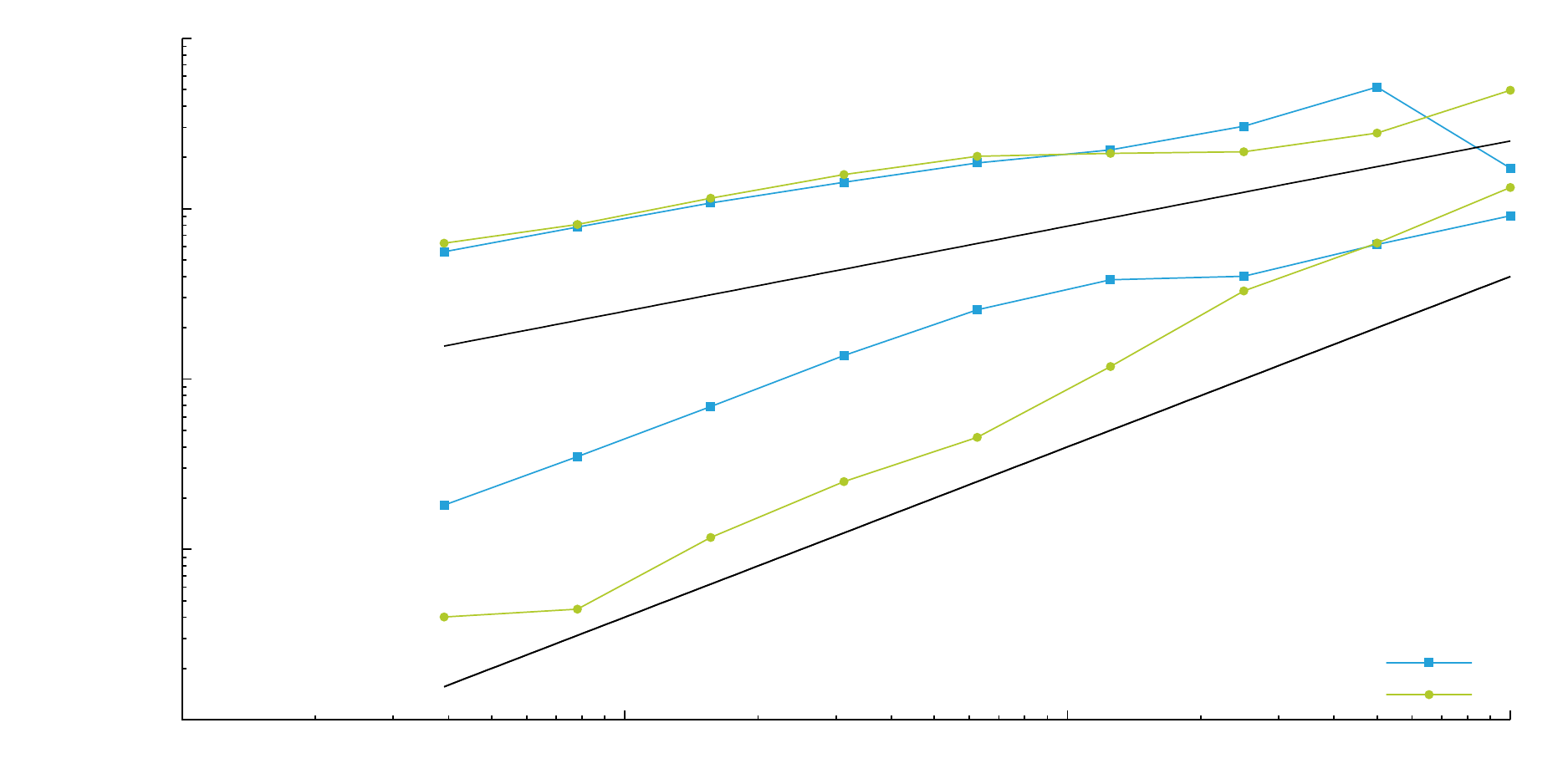}}%
    \gplfronttext
  \end{picture}%
\endgroup

%% file: fig/reference0.tex
\begin{picture}(0,0)%
\includegraphics{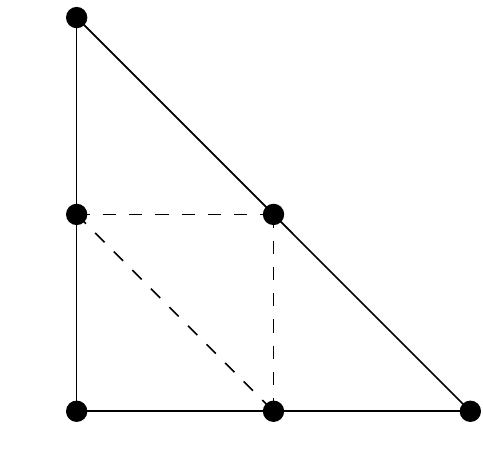}%
\end{picture}%
\setlength{\unitlength}{4144sp}%
\begingroup\makeatletter\ifx\SetFigFont\undefined%
\gdef\SetFigFont#1#2{%
  \fontsize{#1}{#2pt}%
  \selectfont}%
\fi\endgroup%
\begin{picture}(2204,2163)(1450,-4855)
\put(2926,-4786){\makebox(0,0)[lb]{\smash{{\SetFigFont{12}{14.4}{\color[rgb]{1,1,1}$1-s$}%
}}}}
\put(1621,-3391){\rotatebox{90.0}{\makebox(0,0)[lb]{\smash{{\SetFigFont{12}{14.4}{\color[rgb]{1,1,1}$1-q$}%
}}}}}
\put(2251,-3031){\rotatebox{315.0}{\makebox(0,0)[lb]{\smash{{\SetFigFont{12}{14.4}{\color[rgb]{1,1,1}$1-r$}%
}}}}}
\put(3151,-3931){\rotatebox{315.0}{\makebox(0,0)[lb]{\smash{{\SetFigFont{12}{14.4}{\color[rgb]{1,1,1}$r$}%
}}}}}
\put(1621,-4201){\rotatebox{90.0}{\makebox(0,0)[lb]{\smash{{\SetFigFont{12}{14.4}{\color[rgb]{1,1,1}$q$}%
}}}}}
\put(2881,-4291){\makebox(0,0)[lb]{\smash{{\SetFigFont{12}{14.4}{\color[rgb]{0,0,0}$\hat{T}_1$}%
}}}}
\put(2341,-3976){\makebox(0,0)[lb]{\smash{{\SetFigFont{12}{14.4}{\color[rgb]{0,0,0}$\hat{T}_3$}%
}}}}
\put(2026,-3391){\makebox(0,0)[lb]{\smash{{\SetFigFont{12}{14.4}{\color[rgb]{0,0,0}$\hat{T}_2$}%
}}}}
\put(2026,-4291){\makebox(0,0)[lb]{\smash{{\SetFigFont{12}{14.4}{\color[rgb]{0,0,0}$\hat{T}_0$}%
}}}}
\put(2116,-4786){\makebox(0,0)[lb]{\smash{{\SetFigFont{12}{14.4}{\color[rgb]{1,1,1}$s$}%
}}}}
\put(1576,-4786){\makebox(0,0)[lb]{\smash{{\SetFigFont{12}{14.4}{\color[rgb]{0,0,0}$(0,0)$}%
}}}}
\put(3376,-4786){\makebox(0,0)[lb]{\smash{{\SetFigFont{12}{14.4}{\color[rgb]{0,0,0}$(1,0)$}%
}}}}
\put(1981,-2851){\makebox(0,0)[lb]{\smash{{\SetFigFont{12}{14.4}{\color[rgb]{0,0,0}$(0,1)$}%
}}}}
\end{picture}%

%% file: fig/onebigcellp2asp1.tex
\begin{picture}(0,0)%
\includegraphics{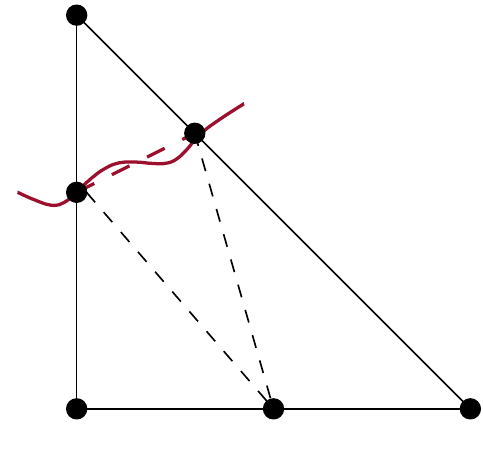}%
\end{picture}%
\setlength{\unitlength}{4144sp}%
\begingroup\makeatletter\ifx\SetFigFont\undefined%
\gdef\SetFigFont#1#2{%
  \fontsize{#1}{#2pt}%
  \selectfont}%
\fi\endgroup%
\begin{picture}(2204,2153)(1450,-4850)
\put(2116,-4786){\makebox(0,0)[lb]{\smash{{\SetFigFont{12}{14.4}{\color[rgb]{0,0,0}$s$}%
}}}}
\put(2926,-4786){\makebox(0,0)[lb]{\smash{{\SetFigFont{12}{14.4}{\color[rgb]{0,0,0}$1-s$}%
}}}}
\put(1936,-4291){\makebox(0,0)[lb]{\smash{{\SetFigFont{12}{14.4}{\color[rgb]{0,0,0}$T_0$}%
}}}}
\put(2791,-4291){\makebox(0,0)[lb]{\smash{{\SetFigFont{12}{14.4}{\color[rgb]{0,0,0}$T_1$}%
}}}}
\put(2116,-3706){\makebox(0,0)[lb]{\smash{{\SetFigFont{12}{14.4}{\color[rgb]{0,0,0}$T_3$}%
}}}}
\put(1936,-3301){\makebox(0,0)[lb]{\smash{{\SetFigFont{12}{14.4}{\color[rgb]{0,0,0}$T_2$}%
}}}}
\put(2026,-2806){\rotatebox{315.0}{\makebox(0,0)[lb]{\smash{{\SetFigFont{12}{14.4}{\color[rgb]{0,0,0}$1-r$}%
}}}}}
\put(2701,-3481){\rotatebox{315.0}{\makebox(0,0)[lb]{\smash{{\SetFigFont{12}{14.4}{\color[rgb]{0,0,0}$r$}%
}}}}}
\put(1621,-4066){\rotatebox{90.0}{\makebox(0,0)[lb]{\smash{{\SetFigFont{12}{14.4}{\color[rgb]{0,0,0}$q$}%
}}}}}
\put(1621,-3391){\rotatebox{90.0}{\makebox(0,0)[lb]{\smash{{\SetFigFont{12}{14.4}{\color[rgb]{0,0,0}$1-q$}%
}}}}}
\end{picture}%

%% file: fig/reference1.tex
\begin{picture}(0,0)%
\includegraphics{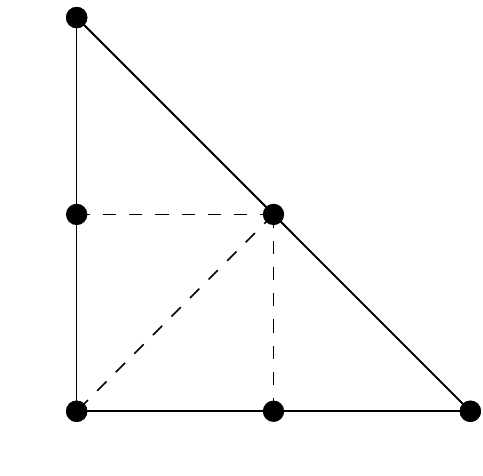}%
\end{picture}%
\setlength{\unitlength}{4144sp}%
\begingroup\makeatletter\ifx\SetFigFont\undefined%
\gdef\SetFigFont#1#2{%
  \fontsize{#1}{#2pt}%
  \selectfont}%
\fi\endgroup%
\begin{picture}(2204,2163)(1450,-4855)
\put(2116,-4786){\makebox(0,0)[lb]{\smash{{\SetFigFont{12}{14.4}{\color[rgb]{1,1,1}$s$}%
}}}}
\put(2926,-4786){\makebox(0,0)[lb]{\smash{{\SetFigFont{12}{14.4}{\color[rgb]{1,1,1}$1-s$}%
}}}}
\put(1621,-3391){\rotatebox{90.0}{\makebox(0,0)[lb]{\smash{{\SetFigFont{12}{14.4}{\color[rgb]{1,1,1}$1-q$}%
}}}}}
\put(2026,-3481){\makebox(0,0)[lb]{\smash{{\SetFigFont{12}{14.4}{\color[rgb]{0,0,0}$\hat{T}_3$}%
}}}}
\put(2026,-3976){\makebox(0,0)[lb]{\smash{{\SetFigFont{12}{14.4}{\color[rgb]{0,0,0}$\hat{T}_0$}%
}}}}
\put(2296,-4336){\makebox(0,0)[lb]{\smash{{\SetFigFont{12}{14.4}{\color[rgb]{0,0,0}$\hat{T}_1$}%
}}}}
\put(2881,-4336){\makebox(0,0)[lb]{\smash{{\SetFigFont{12}{14.4}{\color[rgb]{0,0,0}$\hat{T}_2$}%
}}}}
\put(2251,-3031){\rotatebox{315.0}{\makebox(0,0)[lb]{\smash{{\SetFigFont{12}{14.4}{\color[rgb]{1,1,1}$1-r$}%
}}}}}
\put(1621,-4201){\rotatebox{90.0}{\makebox(0,0)[lb]{\smash{{\SetFigFont{12}{14.4}{\color[rgb]{1,1,1}$q$}%
}}}}}
\put(3196,-3931){\rotatebox{315.0}{\makebox(0,0)[lb]{\smash{{\SetFigFont{12}{14.4}{\color[rgb]{1,1,1}$r$}%
}}}}}
\put(1576,-4786){\makebox(0,0)[lb]{\smash{{\SetFigFont{12}{14.4}{\color[rgb]{0,0,0}$(0,0)$}%
}}}}
\put(3376,-4786){\makebox(0,0)[lb]{\smash{{\SetFigFont{12}{14.4}{\color[rgb]{0,0,0}$(1,0)$}%
}}}}
\put(1981,-2851){\makebox(0,0)[lb]{\smash{{\SetFigFont{12}{14.4}{\color[rgb]{0,0,0}$(0,1)$}%
}}}}
\end{picture}%

%% file: fig/onebigcellp2asp1vertex.tex
\begin{picture}(0,0)%
\includegraphics{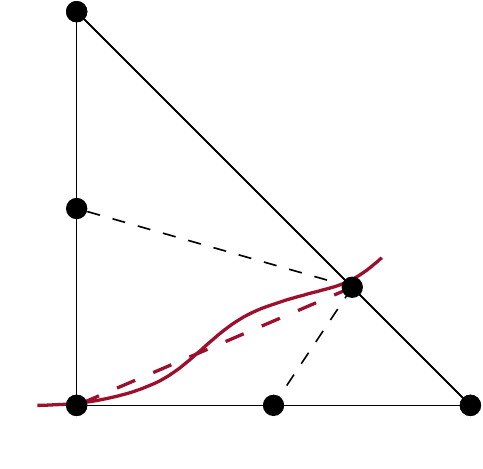}%
\end{picture}%
\setlength{\unitlength}{4144sp}%
\begingroup\makeatletter\ifx\SetFigFont\undefined%
\gdef\SetFigFont#1#2{%
  \fontsize{#1}{#2pt}%
  \selectfont}%
\fi\endgroup%
\begin{picture}(2204,2142)(1450,-4850)
\put(2116,-4786){\makebox(0,0)[lb]{\smash{{\SetFigFont{12}{14.4}{\color[rgb]{0,0,0}$s$}%
}}}}
\put(2926,-4786){\makebox(0,0)[lb]{\smash{{\SetFigFont{12}{14.4}{\color[rgb]{0,0,0}$1-s$}%
}}}}
\put(1621,-4066){\rotatebox{90.0}{\makebox(0,0)[lb]{\smash{{\SetFigFont{12}{14.4}{\color[rgb]{0,0,0}$q$}%
}}}}}
\put(1621,-3391){\rotatebox{90.0}{\makebox(0,0)[lb]{\smash{{\SetFigFont{12}{14.4}{\color[rgb]{0,0,0}$1-q$}%
}}}}}
\put(3331,-4111){\rotatebox{315.0}{\makebox(0,0)[lb]{\smash{{\SetFigFont{12}{14.4}{\color[rgb]{0,0,0}$r$}%
}}}}}
\put(2386,-3166){\rotatebox{315.0}{\makebox(0,0)[lb]{\smash{{\SetFigFont{12}{14.4}{\color[rgb]{0,0,0}$1-r$}%
}}}}}
\end{picture}%

%% file: fig/reference2.tex
\begin{picture}(0,0)%
\includegraphics{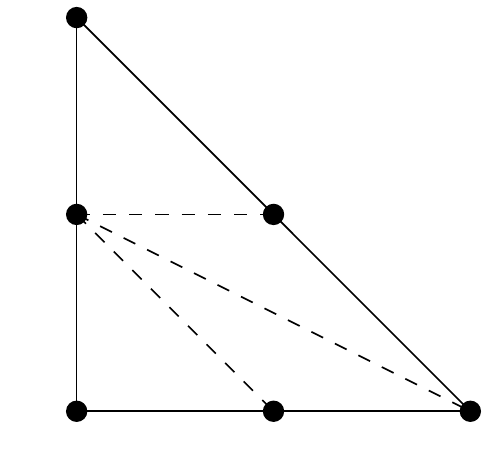}%
\end{picture}%
\setlength{\unitlength}{4144sp}%
\begingroup\makeatletter\ifx\SetFigFont\undefined%
\gdef\SetFigFont#1#2{%
  \fontsize{#1}{#2pt}%
  \selectfont}%
\fi\endgroup%
\begin{picture}(2204,2163)(1450,-4855)
\put(2116,-4786){\makebox(0,0)[lb]{\smash{{\SetFigFont{12}{14.4}{\color[rgb]{1,1,1}$s$}%
}}}}
\put(2926,-4786){\makebox(0,0)[lb]{\smash{{\SetFigFont{12}{14.4}{\color[rgb]{1,1,1}$1-s$}%
}}}}
\put(1621,-3391){\rotatebox{90.0}{\makebox(0,0)[lb]{\smash{{\SetFigFont{12}{14.4}{\color[rgb]{1,1,1}$1-q$}%
}}}}}
\put(2251,-3031){\rotatebox{315.0}{\makebox(0,0)[lb]{\smash{{\SetFigFont{12}{14.4}{\color[rgb]{1,1,1}$1-r$}%
}}}}}
\put(3151,-3931){\rotatebox{315.0}{\makebox(0,0)[lb]{\smash{{\SetFigFont{12}{14.4}{\color[rgb]{1,1,1}$r$}%
}}}}}
\put(1621,-4201){\rotatebox{90.0}{\makebox(0,0)[lb]{\smash{{\SetFigFont{12}{14.4}{\color[rgb]{1,1,1}$q$}%
}}}}}
\put(1981,-4336){\makebox(0,0)[lb]{\smash{{\SetFigFont{12}{14.4}{\color[rgb]{0,0,0}$\hat{T}_0$}%
}}}}
\put(2611,-4336){\makebox(0,0)[lb]{\smash{{\SetFigFont{12}{14.4}{\color[rgb]{0,0,0}$\hat{T}_1$}%
}}}}
\put(2521,-3931){\makebox(0,0)[lb]{\smash{{\SetFigFont{12}{14.4}{\color[rgb]{0,0,0}$\hat{T}_2$}%
}}}}
\put(1981,-3436){\makebox(0,0)[lb]{\smash{{\SetFigFont{12}{14.4}{\color[rgb]{0,0,0}$\hat{T}_3$}%
}}}}
\put(1576,-4786){\makebox(0,0)[lb]{\smash{{\SetFigFont{12}{14.4}{\color[rgb]{0,0,0}$(0,0)$}%
}}}}
\put(3376,-4786){\makebox(0,0)[lb]{\smash{{\SetFigFont{12}{14.4}{\color[rgb]{0,0,0}$(1,0)$}%
}}}}
\put(1981,-2851){\makebox(0,0)[lb]{\smash{{\SetFigFont{12}{14.4}{\color[rgb]{0,0,0}$(0,1)$}%
}}}}
\end{picture}%

%% file: fig/reference3.tex
\begin{picture}(0,0)%
\includegraphics{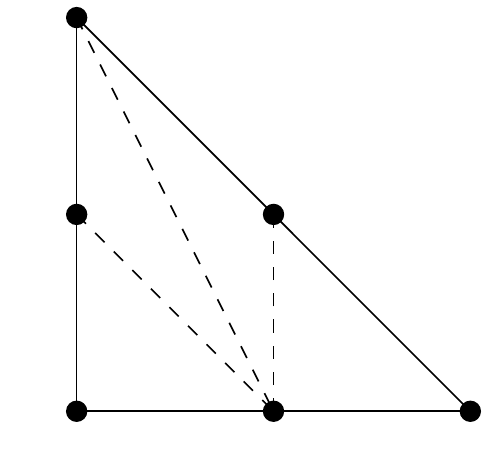}%
\end{picture}%
\setlength{\unitlength}{4144sp}%
\begingroup\makeatletter\ifx\SetFigFont\undefined%
\gdef\SetFigFont#1#2{%
  \fontsize{#1}{#2pt}%
  \selectfont}%
\fi\endgroup%
\begin{picture}(2204,2163)(1450,-4855)
\put(2116,-4786){\makebox(0,0)[lb]{\smash{{\SetFigFont{12}{14.4}{\color[rgb]{1,1,1}$s$}%
}}}}
\put(2926,-4786){\makebox(0,0)[lb]{\smash{{\SetFigFont{12}{14.4}{\color[rgb]{1,1,1}$1-s$}%
}}}}
\put(1621,-3391){\rotatebox{90.0}{\makebox(0,0)[lb]{\smash{{\SetFigFont{12}{14.4}{\color[rgb]{1,1,1}$1-q$}%
}}}}}
\put(2251,-3031){\rotatebox{315.0}{\makebox(0,0)[lb]{\smash{{\SetFigFont{12}{14.4}{\color[rgb]{1,1,1}$1-r$}%
}}}}}
\put(3151,-3931){\rotatebox{315.0}{\makebox(0,0)[lb]{\smash{{\SetFigFont{12}{14.4}{\color[rgb]{1,1,1}$r$}%
}}}}}
\put(1621,-4201){\rotatebox{90.0}{\makebox(0,0)[lb]{\smash{{\SetFigFont{12}{14.4}{\color[rgb]{1,1,1}$q$}%
}}}}}
\put(1981,-4336){\makebox(0,0)[lb]{\smash{{\SetFigFont{12}{14.4}{\color[rgb]{0,0,0}$\hat{T}_0$}%
}}}}
\put(2881,-4336){\makebox(0,0)[lb]{\smash{{\SetFigFont{12}{14.4}{\color[rgb]{0,0,0}$\hat{T}_3$}%
}}}}
\put(2386,-3706){\makebox(0,0)[lb]{\smash{{\SetFigFont{12}{14.4}{\color[rgb]{0,0,0}$\hat{T}_2$}%
}}}}
\put(1936,-3706){\makebox(0,0)[lb]{\smash{{\SetFigFont{12}{14.4}{\color[rgb]{0,0,0}$\hat{T}_1$}%
}}}}
\put(1576,-4786){\makebox(0,0)[lb]{\smash{{\SetFigFont{12}{14.4}{\color[rgb]{0,0,0}$(0,0)$}%
}}}}
\put(3376,-4786){\makebox(0,0)[lb]{\smash{{\SetFigFont{12}{14.4}{\color[rgb]{0,0,0}$(1,0)$}%
}}}}
\put(1981,-2851){\makebox(0,0)[lb]{\smash{{\SetFigFont{12}{14.4}{\color[rgb]{0,0,0}$(0,1)$}%
}}}}
\end{picture}%

%% file: fig/quadrature.tex
\begin{picture}(0,0)%
\includegraphics{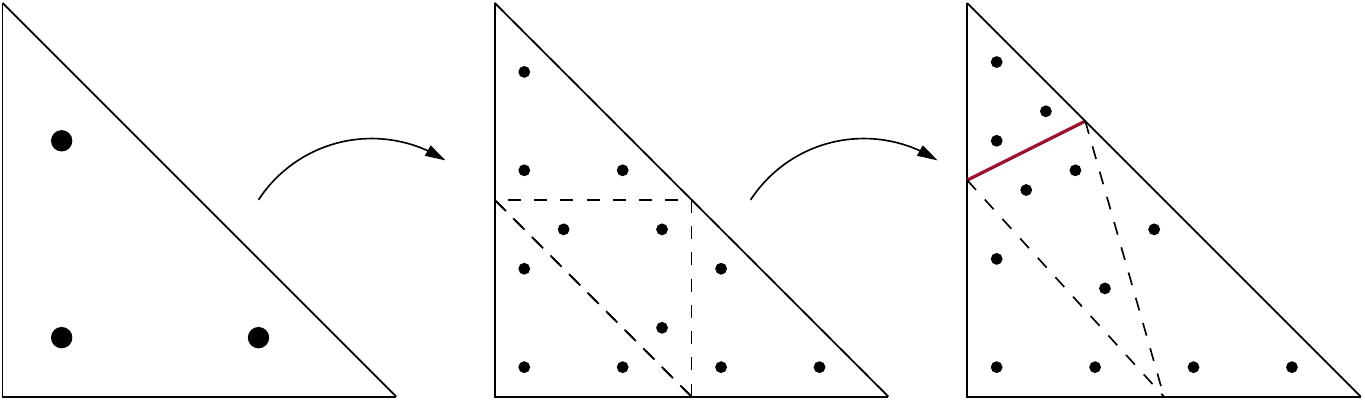}%
\end{picture}%
\setlength{\unitlength}{4144sp}%
\begingroup\makeatletter\ifx\SetFigFont\undefined%
\gdef\SetFigFont#1#2{%
  \fontsize{#1}{#2pt}%
  \selectfont}%
\fi\endgroup%
\begin{picture}(6234,1824)(-461,-4573)
\put(1081,-3256){\makebox(0,0)[lb]{\smash{{\SetFigFont{12}{14.4}{\color[rgb]{0,0,0}$F_{T_i}^1$}%
}}}}
\put(3286,-3256){\makebox(0,0)[lb]{\smash{{\SetFigFont{12}{14.4}{\color[rgb]{0,0,0}$F_{T_i}^2$}%
}}}}
\end{picture}%

%% file: fig/partitiondomain.tex
\begin{picture}(0,0)%
\includegraphics{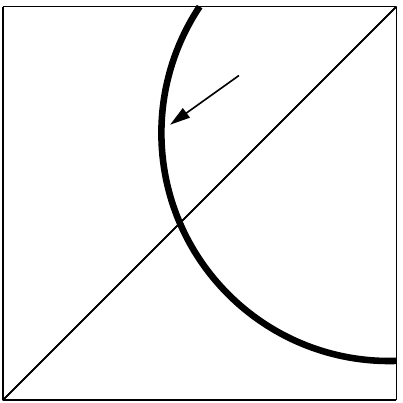}%
\end{picture}%
\setlength{\unitlength}{4144sp}%
\begingroup\makeatletter\ifx\SetFigFont\undefined%
\gdef\SetFigFont#1#2{%
  \fontsize{#1}{#2pt}%
  \selectfont}%
\fi\endgroup%
\begin{picture}(1835,1835)(-11,-973)
\put(1261,-196){\makebox(0,0)[lb]{\smash{{\SetFigFont{12}{14.4}{\color[rgb]{0,0,0}$\Omega_2$}%
}}}}
\put(541,-781){\makebox(0,0)[lb]{\smash{{\SetFigFont{12}{14.4}{\color[rgb]{0,0,0}$\Omega_1$}%
}}}}
\put(1126,479){\makebox(0,0)[lb]{\smash{{\SetFigFont{12}{14.4}{\color[rgb]{0,0,0}$\Gamma$}%
}}}}
\end{picture}%

%% file: fig/partitionmesh.tex
\begin{picture}(0,0)%
\includegraphics{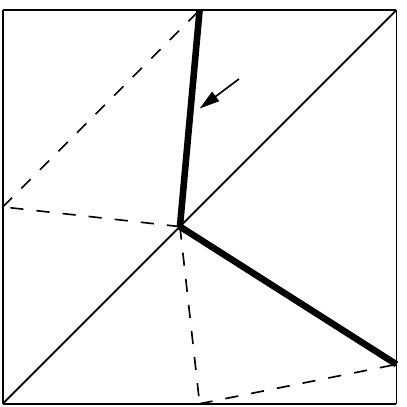}%
\end{picture}%
\setlength{\unitlength}{4144sp}%
\begingroup\makeatletter\ifx\SetFigFont\undefined%
\gdef\SetFigFont#1#2{%
  \fontsize{#1}{#2pt}%
  \selectfont}%
\fi\endgroup%
\begin{picture}(1845,1845)(-11,-973)
\put(1261,-196){\makebox(0,0)[lb]{\smash{{\SetFigFont{12}{14.4}{\color[rgb]{0,0,0}$\mathcal{T}_{2,h}$}%
}}}}
\put(541,-781){\makebox(0,0)[lb]{\smash{{\SetFigFont{12}{14.4}{\color[rgb]{0,0,0}$\mathcal{T}_{1,h}$}%
}}}}
\put(1126,479){\makebox(0,0)[lb]{\smash{{\SetFigFont{12}{14.4}{\color[rgb]{0,0,0}$\Gamma_h$}%
}}}}
\end{picture}%

%% file: fig/limit_line.tex
\begin{picture}(0,0)%
\includegraphics{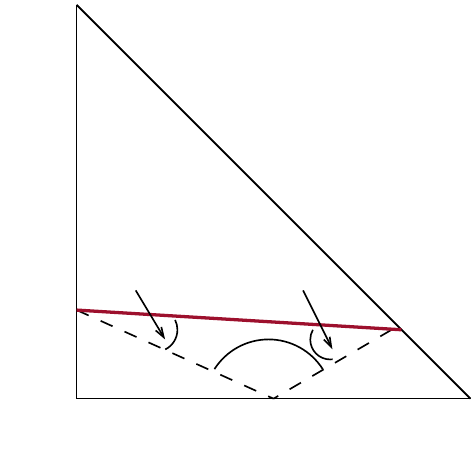}%
\end{picture}%
\setlength{\unitlength}{4144sp}%
\begingroup\makeatletter\ifx\SetFigFont\undefined%
\gdef\SetFigFont#1#2{%
  \fontsize{#1}{#2pt}%
  \selectfont}%
\fi\endgroup%
\begin{picture}(2169,2101)(1450,-4940)
\put(2116,-4876){\makebox(0,0)[lb]{\smash{{\SetFigFont{12}{14.4}{\color[rgb]{0,0,0}$s$}%
}}}}
\put(2926,-4876){\makebox(0,0)[lb]{\smash{{\SetFigFont{12}{14.4}{\color[rgb]{0,0,0}$1-s$}%
}}}}
\put(1621,-4561){\rotatebox{90.0}{\makebox(0,0)[lb]{\smash{{\SetFigFont{12}{14.4}{\color[rgb]{0,0,0}$q$}%
}}}}}
\put(2611,-4561){\makebox(0,0)[lb]{\smash{{\SetFigFont{12}{14.4}{\color[rgb]{0,0,0}$\gamma$}%
}}}}
\put(1981,-4111){\makebox(0,0)[lb]{\smash{{\SetFigFont{12}{14.4}{\color[rgb]{0,0,0}$\beta$}%
}}}}
\put(2701,-4111){\makebox(0,0)[lb]{\smash{{\SetFigFont{12}{14.4}{\color[rgb]{0,0,0}$\alpha$}%
}}}}
\put(2566,-3436){\rotatebox{315.0}{\makebox(0,0)[lb]{\smash{{\SetFigFont{12}{14.4}{\color[rgb]{0,0,0}$1-r$}%
}}}}}
\put(1621,-3796){\rotatebox{90.0}{\makebox(0,0)[lb]{\smash{{\SetFigFont{12}{14.4}{\color[rgb]{0,0,0}$1-q$}%
}}}}}
\put(3511,-4381){\rotatebox{315.0}{\makebox(0,0)[lb]{\smash{{\SetFigFont{12}{14.4}{\color[rgb]{0,0,0}$r$}%
}}}}}
\end{picture}%

%% file: fig/limit_cut_vertex.tex
\begin{picture}(0,0)%
\includegraphics{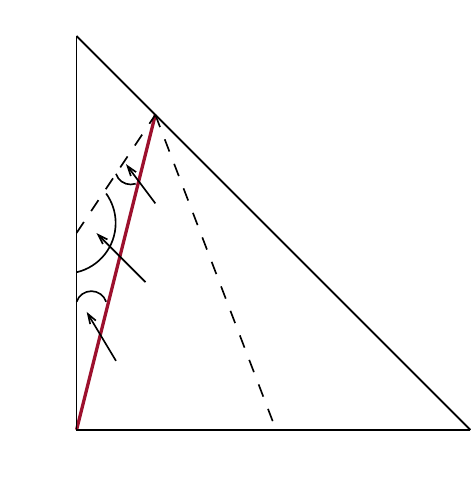}%
\end{picture}%
\setlength{\unitlength}{4144sp}%
\begingroup\makeatletter\ifx\SetFigFont\undefined%
\gdef\SetFigFont#1#2{%
  \fontsize{#1}{#2pt}%
  \selectfont}%
\fi\endgroup%
\begin{picture}(2163,2243)(3250,-3050)
\put(3916,-2986){\makebox(0,0)[lb]{\smash{{\SetFigFont{12}{14.4}{\color[rgb]{0,0,0}$s$}%
}}}}
\put(4726,-2986){\makebox(0,0)[lb]{\smash{{\SetFigFont{12}{14.4}{\color[rgb]{0,0,0}$1-s$}%
}}}}
\put(3421,-1591){\rotatebox{90.0}{\makebox(0,0)[lb]{\smash{{\SetFigFont{12}{14.4}{\color[rgb]{0,0,0}$1-q$}%
}}}}}
\put(3736,-916){\rotatebox{315.0}{\makebox(0,0)[lb]{\smash{{\SetFigFont{12}{14.4}{\color[rgb]{0,0,0}$1-r$}%
}}}}}
\put(3961,-1861){\makebox(0,0)[lb]{\smash{{\SetFigFont{12}{14.4}{\color[rgb]{0,0,0}$\gamma_1$}%
}}}}
\put(3781,-2536){\makebox(0,0)[lb]{\smash{{\SetFigFont{12}{14.4}{\color[rgb]{0,0,0}$\beta_1$}%
}}}}
\put(3421,-2401){\rotatebox{90.0}{\makebox(0,0)[lb]{\smash{{\SetFigFont{12}{14.4}{\color[rgb]{0,0,0}$q$}%
}}}}}
\put(4906,-2086){\rotatebox{315.0}{\makebox(0,0)[lb]{\smash{{\SetFigFont{12}{14.4}{\color[rgb]{0,0,0}$r$}%
}}}}}
\put(3916,-2176){\makebox(0,0)[lb]{\smash{{\SetFigFont{12}{14.4}{\color[rgb]{0,0,0}$\alpha_1$}%
}}}}
\end{picture}%

%% file: fig/cutthroughtwofacets.tex
\begin{picture}(0,0)%
\includegraphics{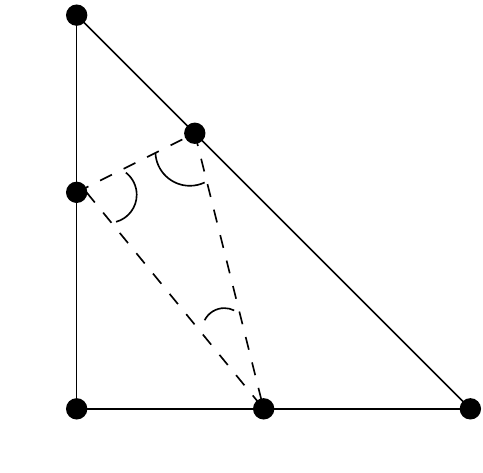}%
\end{picture}%
\setlength{\unitlength}{4144sp}%
\begingroup\makeatletter\ifx\SetFigFont\undefined%
\gdef\SetFigFont#1#2{%
  \fontsize{#1}{#2pt}%
  \selectfont}%
\fi\endgroup%
\begin{picture}(2204,2153)(1450,-4850)
\put(2116,-4786){\makebox(0,0)[lb]{\smash{{\SetFigFont{12}{14.4}{\color[rgb]{0,0,0}$s$}%
}}}}
\put(2926,-4786){\makebox(0,0)[lb]{\smash{{\SetFigFont{12}{14.4}{\color[rgb]{0,0,0}$1-s$}%
}}}}
\put(2026,-2806){\rotatebox{315.0}{\makebox(0,0)[lb]{\smash{{\SetFigFont{12}{14.4}{\color[rgb]{0,0,0}$1-r$}%
}}}}}
\put(2701,-3481){\rotatebox{315.0}{\makebox(0,0)[lb]{\smash{{\SetFigFont{12}{14.4}{\color[rgb]{0,0,0}$r$}%
}}}}}
\put(1621,-4066){\rotatebox{90.0}{\makebox(0,0)[lb]{\smash{{\SetFigFont{12}{14.4}{\color[rgb]{0,0,0}$q$}%
}}}}}
\put(1621,-3391){\rotatebox{90.0}{\makebox(0,0)[lb]{\smash{{\SetFigFont{12}{14.4}{\color[rgb]{0,0,0}$1-q$}%
}}}}}
\put(2431,-4291){\makebox(0,0)[lb]{\smash{{\SetFigFont{12}{14.4}{\color[rgb]{0,0,0}$\gamma$}%
}}}}
\put(1936,-3661){\makebox(0,0)[lb]{\smash{{\SetFigFont{12}{14.4}{\color[rgb]{0,0,0}$\beta$}%
}}}}
\put(2206,-3481){\makebox(0,0)[lb]{\smash{{\SetFigFont{12}{14.4}{\color[rgb]{0,0,0}$\alpha$}%
}}}}
\end{picture}%

%% file: fig/cutthroughvertex0.tex
\begin{picture}(0,0)%
\includegraphics{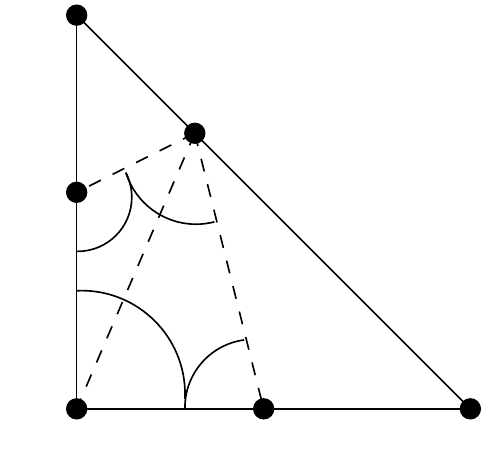}%
\end{picture}%
\setlength{\unitlength}{4144sp}%
\begingroup\makeatletter\ifx\SetFigFont\undefined%
\gdef\SetFigFont#1#2{%
  \fontsize{#1}{#2pt}%
  \selectfont}%
\fi\endgroup%
\begin{picture}(2204,2153)(1450,-4850)
\put(2116,-4786){\makebox(0,0)[lb]{\smash{{\SetFigFont{12}{14.4}{\color[rgb]{0,0,0}$s$}%
}}}}
\put(2926,-4786){\makebox(0,0)[lb]{\smash{{\SetFigFont{12}{14.4}{\color[rgb]{0,0,0}$1-s$}%
}}}}
\put(2026,-2806){\rotatebox{315.0}{\makebox(0,0)[lb]{\smash{{\SetFigFont{12}{14.4}{\color[rgb]{0,0,0}$1-r$}%
}}}}}
\put(1621,-4066){\rotatebox{90.0}{\makebox(0,0)[lb]{\smash{{\SetFigFont{12}{14.4}{\color[rgb]{0,0,0}$q$}%
}}}}}
\put(1621,-3391){\rotatebox{90.0}{\makebox(0,0)[lb]{\smash{{\SetFigFont{12}{14.4}{\color[rgb]{0,0,0}$1-q$}%
}}}}}
\put(2701,-3481){\rotatebox{315.0}{\makebox(0,0)[lb]{\smash{{\SetFigFont{12}{14.4}{\color[rgb]{0,0,0}$r$}%
}}}}}
\put(2071,-3571){\makebox(0,0)[lb]{\smash{{\SetFigFont{12}{14.4}{\color[rgb]{0,0,0}$\gamma_1$}%
}}}}
\put(1981,-4471){\makebox(0,0)[lb]{\smash{{\SetFigFont{12}{14.4}{\color[rgb]{0,0,0}$\alpha_2$}%
}}}}
\put(2251,-3661){\makebox(0,0)[lb]{\smash{{\SetFigFont{12}{14.4}{\color[rgb]{0,0,0}$\gamma_2$}%
}}}}
\put(2386,-4471){\makebox(0,0)[lb]{\smash{{\SetFigFont{12}{14.4}{\color[rgb]{0,0,0}$\beta_2$}%
}}}}
\put(1801,-3706){\makebox(0,0)[lb]{\smash{{\SetFigFont{12}{14.4}{\color[rgb]{0,0,0}$\alpha_1$}%
}}}}
\put(1801,-4156){\makebox(0,0)[lb]{\smash{{\SetFigFont{12}{14.4}{\color[rgb]{0,0,0}$\beta_1$}%
}}}}
\end{picture}%

%% file: fig/cutthroughvertex1.tex
\begin{picture}(0,0)%
\includegraphics{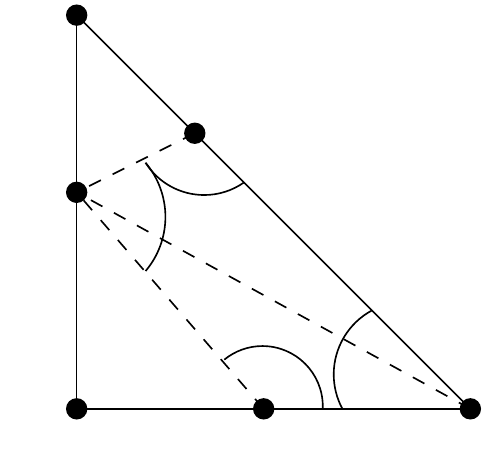}%
\end{picture}%
\setlength{\unitlength}{4144sp}%
\begingroup\makeatletter\ifx\SetFigFont\undefined%
\gdef\SetFigFont#1#2{%
  \fontsize{#1}{#2pt}%
  \selectfont}%
\fi\endgroup%
\begin{picture}(2204,2153)(1450,-4850)
\put(2116,-4786){\makebox(0,0)[lb]{\smash{{\SetFigFont{12}{14.4}{\color[rgb]{0,0,0}$s$}%
}}}}
\put(2926,-4786){\makebox(0,0)[lb]{\smash{{\SetFigFont{12}{14.4}{\color[rgb]{0,0,0}$1-s$}%
}}}}
\put(2026,-2806){\rotatebox{315.0}{\makebox(0,0)[lb]{\smash{{\SetFigFont{12}{14.4}{\color[rgb]{0,0,0}$1-r$}%
}}}}}
\put(2701,-3481){\rotatebox{315.0}{\makebox(0,0)[lb]{\smash{{\SetFigFont{12}{14.4}{\color[rgb]{0,0,0}$r$}%
}}}}}
\put(1621,-4066){\rotatebox{90.0}{\makebox(0,0)[lb]{\smash{{\SetFigFont{12}{14.4}{\color[rgb]{0,0,0}$q$}%
}}}}}
\put(1621,-3391){\rotatebox{90.0}{\makebox(0,0)[lb]{\smash{{\SetFigFont{12}{14.4}{\color[rgb]{0,0,0}$1-q$}%
}}}}}
\put(2251,-3526){\makebox(0,0)[lb]{\smash{{\SetFigFont{12}{14.4}{\color[rgb]{0,0,0}$\gamma_4$}%
}}}}
\put(1981,-3616){\makebox(0,0)[lb]{\smash{{\SetFigFont{12}{14.4}{\color[rgb]{0,0,0}$\beta_4$}%
}}}}
\put(3106,-4291){\makebox(0,0)[lb]{\smash{{\SetFigFont{12}{14.4}{\color[rgb]{0,0,0}$\alpha_4$}%
}}}}
\put(3016,-4471){\makebox(0,0)[lb]{\smash{{\SetFigFont{12}{14.4}{\color[rgb]{0,0,0}$\beta_3$}%
}}}}
\put(2611,-4471){\makebox(0,0)[lb]{\smash{{\SetFigFont{12}{14.4}{\color[rgb]{0,0,0}$\alpha_3$}%
}}}}
\put(1981,-3841){\makebox(0,0)[lb]{\smash{{\SetFigFont{12}{14.4}{\color[rgb]{0,0,0}$\gamma_3$}%
}}}}
\end{picture}%

%% file: fig/linearapproximation.tex
\begin{picture}(0,0)%
\includegraphics{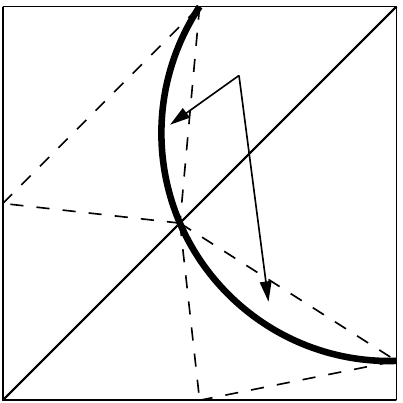}%
\end{picture}%
\setlength{\unitlength}{4144sp}%
\begingroup\makeatletter\ifx\SetFigFont\undefined%
\gdef\SetFigFont#1#2{%
  \fontsize{#1}{#2pt}%
  \selectfont}%
\fi\endgroup%
\begin{picture}(1835,1835)(-11,-973)
\put(1261,-196){\makebox(0,0)[lb]{\smash{{\SetFigFont{12}{14.4}{\color[rgb]{0,0,0}$\Omega_2$}%
}}}}
\put(541,-781){\makebox(0,0)[lb]{\smash{{\SetFigFont{12}{14.4}{\color[rgb]{0,0,0}$\Omega_1$}%
}}}}
\put(136,-376){\makebox(0,0)[lb]{\smash{{\SetFigFont{12}{14.4}{\color[rgb]{0,0,0}$\mathcal{T}_{1,h}$}%
}}}}
\put(1441,254){\makebox(0,0)[lb]{\smash{{\SetFigFont{12}{14.4}{\color[rgb]{0,0,0}$\mathcal{T}_{2,h}$}%
}}}}
\put(901,659){\makebox(0,0)[lb]{\smash{{\SetFigFont{12}{14.4}{\color[rgb]{0,0,0}$\Omega_2\setminus\mathcal{T}_{2,h}$}%
}}}}
\end{picture}%

%% file: results/circular_interface/plot_opt_err.tex
\begingroup
  \makeatletter
  \providecommand\color[2][]{%
    \GenericError{(gnuplot) \space\space\space\@spaces}{%
      Package color not loaded in conjunction with
      terminal option `colourtext'%
    }{See the gnuplot documentation for explanation.%
    }{Either use 'blacktext' in gnuplot or load the package
      color.sty in LaTeX.}%
    \renewcommand\color[2][]{}%
  }%
  \providecommand\includegraphics[2][]{%
    \GenericError{(gnuplot) \space\space\space\@spaces}{%
      Package graphicx or graphics not loaded%
    }{See the gnuplot documentation for explanation.%
    }{The gnuplot epslatex terminal needs graphicx.sty or graphics.sty.}%
    \renewcommand\includegraphics[2][]{}%
  }%
  \providecommand\rotatebox[2]{#2}%
  \@ifundefined{ifGPcolor}{%
    \newif\ifGPcolor
    \GPcolortrue
  }{}%
  \@ifundefined{ifGPblacktext}{%
    \newif\ifGPblacktext
    \GPblacktexttrue
  }{}%
  \let\gplgaddtomacro\g@addto@macro
  \gdef\gplbacktext{}%
  \gdef\gplfronttext{}%
  \makeatother
  \ifGPblacktext
    \def\colorrgb#1{}%
    \def\colorgray#1{}%
  \else
    \ifGPcolor
      \def\colorrgb#1{\color[rgb]{#1}}%
      \def\colorgray#1{\color[gray]{#1}}%
      \expandafter\def\csname LTw\endcsname{\color{white}}%
      \expandafter\def\csname LTb\endcsname{\color{black}}%
      \expandafter\def\csname LTa\endcsname{\color{black}}%
      \expandafter\def\csname LT0\endcsname{\color[rgb]{1,0,0}}%
      \expandafter\def\csname LT1\endcsname{\color[rgb]{0,1,0}}%
      \expandafter\def\csname LT2\endcsname{\color[rgb]{0,0,1}}%
      \expandafter\def\csname LT3\endcsname{\color[rgb]{1,0,1}}%
      \expandafter\def\csname LT4\endcsname{\color[rgb]{0,1,1}}%
      \expandafter\def\csname LT5\endcsname{\color[rgb]{1,1,0}}%
      \expandafter\def\csname LT6\endcsname{\color[rgb]{0,0,0}}%
      \expandafter\def\csname LT7\endcsname{\color[rgb]{1,0.3,0}}%
      \expandafter\def\csname LT8\endcsname{\color[rgb]{0.5,0.5,0.5}}%
    \else
      \def\colorrgb#1{\color{black}}%
      \def\colorgray#1{\color[gray]{#1}}%
      \expandafter\def\csname LTw\endcsname{\color{white}}%
      \expandafter\def\csname LTb\endcsname{\color{black}}%
      \expandafter\def\csname LTa\endcsname{\color{black}}%
      \expandafter\def\csname LT0\endcsname{\color{black}}%
      \expandafter\def\csname LT1\endcsname{\color{black}}%
      \expandafter\def\csname LT2\endcsname{\color{black}}%
      \expandafter\def\csname LT3\endcsname{\color{black}}%
      \expandafter\def\csname LT4\endcsname{\color{black}}%
      \expandafter\def\csname LT5\endcsname{\color{black}}%
      \expandafter\def\csname LT6\endcsname{\color{black}}%
      \expandafter\def\csname LT7\endcsname{\color{black}}%
      \expandafter\def\csname LT8\endcsname{\color{black}}%
    \fi
  \fi
  \setlength{\unitlength}{0.0500bp}%
  \begin{picture}(10800.00,5400.00)%
    \gplgaddtomacro\gplbacktext{%
      \csname LTb\endcsname%
      \put(1122,440){\makebox(0,0)[r]{\strut{} 1e-07}}%
      \put(1122,1111){\makebox(0,0)[r]{\strut{} 1e-06}}%
      \put(1122,1781){\makebox(0,0)[r]{\strut{} 1e-05}}%
      \put(1122,2452){\makebox(0,0)[r]{\strut{} 0.0001}}%
      \put(1122,3123){\makebox(0,0)[r]{\strut{} 0.001}}%
      \put(1122,3794){\makebox(0,0)[r]{\strut{} 0.01}}%
      \put(1122,4464){\makebox(0,0)[r]{\strut{} 0.1}}%
      \put(1122,5135){\makebox(0,0)[r]{\strut{} 1}}%
      \put(1254,220){\makebox(0,0){\strut{} 0.001}}%
      \put(4304,220){\makebox(0,0){\strut{} 0.01}}%
      \put(7353,220){\makebox(0,0){\strut{} 0.1}}%
      \put(10403,220){\makebox(0,0){\strut{} 1}}%
      \put(1495,3019){\makebox(0,0)[l]{\strut{}$O(h)$}}%
      \put(1495,844){\makebox(0,0)[l]{\strut{}$O(h^2)$}}%
      \put(5222,642){\makebox(0,0)[l]{\strut{}mesh size h}}%
    }%
    \gplgaddtomacro\gplfronttext{%
      \csname LTb\endcsname%
      \put(9416,998){\makebox(0,0)[r]{\strut{}$\norm{u-u_h}$}}%
      \csname LTb\endcsname%
      \put(9416,668){\makebox(0,0)[r]{\strut{}$\norm{\nabla(u-u_h)}$}}%
    }%
    \gplbacktext
    \put(0,0){\includegraphics{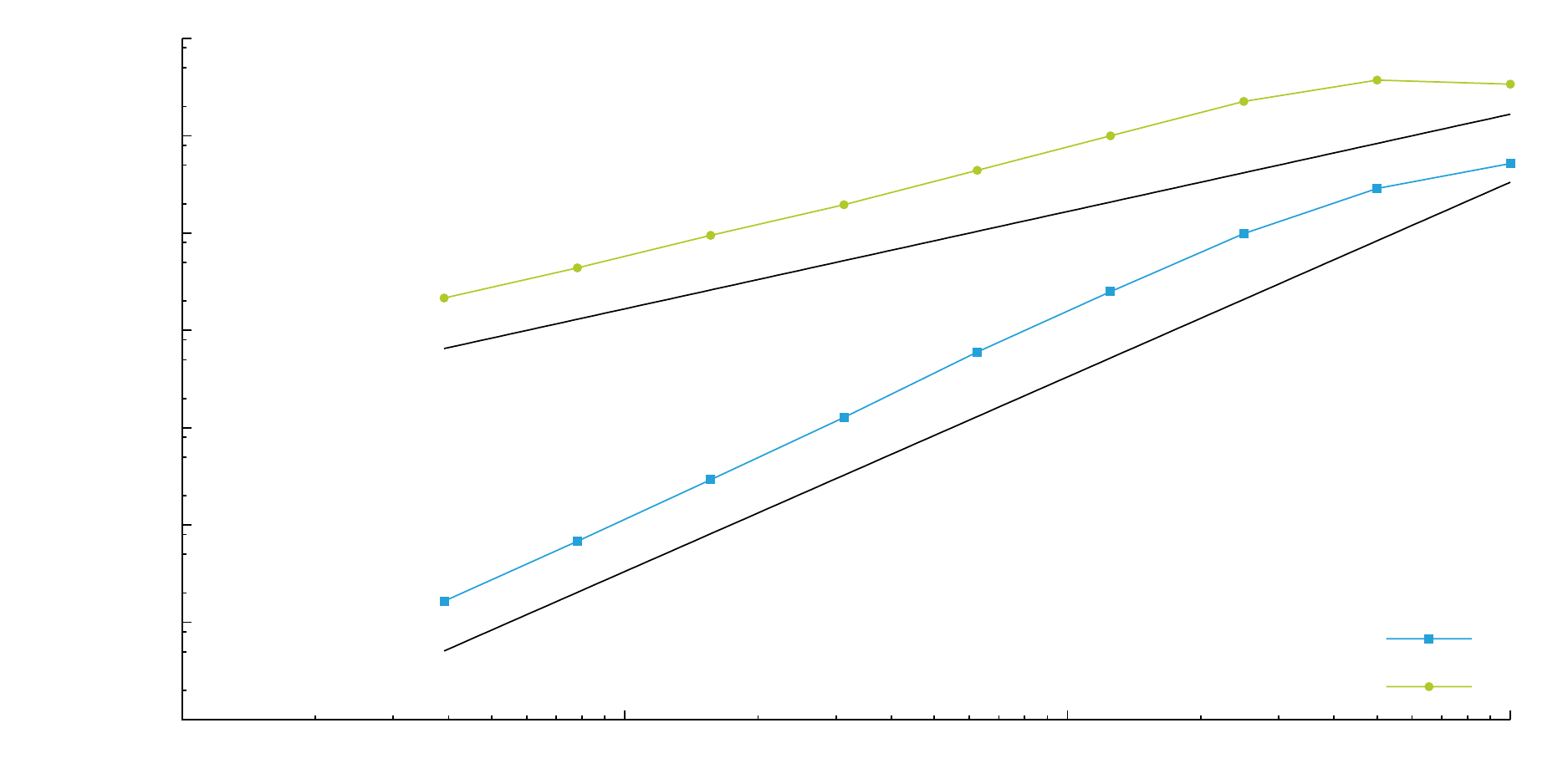}}%
    \gplfronttext
  \end{picture}%
\endgroup

%% file: results/horizontal_cut/plot_L2_error_eps_cross_ohne.tex
\begingroup
  \makeatletter
  \providecommand\color[2][]{%
    \GenericError{(gnuplot) \space\space\space\@spaces}{%
      Package color not loaded in conjunction with
      terminal option `colourtext'%
    }{See the gnuplot documentation for explanation.%
    }{Either use 'blacktext' in gnuplot or load the package
      color.sty in LaTeX.}%
    \renewcommand\color[2][]{}%
  }%
  \providecommand\includegraphics[2][]{%
    \GenericError{(gnuplot) \space\space\space\@spaces}{%
      Package graphicx or graphics not loaded%
    }{See the gnuplot documentation for explanation.%
    }{The gnuplot epslatex terminal needs graphicx.sty or graphics.sty.}%
    \renewcommand\includegraphics[2][]{}%
  }%
  \providecommand\rotatebox[2]{#2}%
  \@ifundefined{ifGPcolor}{%
    \newif\ifGPcolor
    \GPcolortrue
  }{}%
  \@ifundefined{ifGPblacktext}{%
    \newif\ifGPblacktext
    \GPblacktexttrue
  }{}%
  \let\gplgaddtomacro\g@addto@macro
  \gdef\gplbacktext{}%
  \gdef\gplfronttext{}%
  \makeatother
  \ifGPblacktext
    \def\colorrgb#1{}%
    \def\colorgray#1{}%
  \else
    \ifGPcolor
      \def\colorrgb#1{\color[rgb]{#1}}%
      \def\colorgray#1{\color[gray]{#1}}%
      \expandafter\def\csname LTw\endcsname{\color{white}}%
      \expandafter\def\csname LTb\endcsname{\color{black}}%
      \expandafter\def\csname LTa\endcsname{\color{black}}%
      \expandafter\def\csname LT0\endcsname{\color[rgb]{1,0,0}}%
      \expandafter\def\csname LT1\endcsname{\color[rgb]{0,1,0}}%
      \expandafter\def\csname LT2\endcsname{\color[rgb]{0,0,1}}%
      \expandafter\def\csname LT3\endcsname{\color[rgb]{1,0,1}}%
      \expandafter\def\csname LT4\endcsname{\color[rgb]{0,1,1}}%
      \expandafter\def\csname LT5\endcsname{\color[rgb]{1,1,0}}%
      \expandafter\def\csname LT6\endcsname{\color[rgb]{0,0,0}}%
      \expandafter\def\csname LT7\endcsname{\color[rgb]{1,0.3,0}}%
      \expandafter\def\csname LT8\endcsname{\color[rgb]{0.5,0.5,0.5}}%
    \else
      \def\colorrgb#1{\color{black}}%
      \def\colorgray#1{\color[gray]{#1}}%
      \expandafter\def\csname LTw\endcsname{\color{white}}%
      \expandafter\def\csname LTb\endcsname{\color{black}}%
      \expandafter\def\csname LTa\endcsname{\color{black}}%
      \expandafter\def\csname LT0\endcsname{\color{black}}%
      \expandafter\def\csname LT1\endcsname{\color{black}}%
      \expandafter\def\csname LT2\endcsname{\color{black}}%
      \expandafter\def\csname LT3\endcsname{\color{black}}%
      \expandafter\def\csname LT4\endcsname{\color{black}}%
      \expandafter\def\csname LT5\endcsname{\color{black}}%
      \expandafter\def\csname LT6\endcsname{\color{black}}%
      \expandafter\def\csname LT7\endcsname{\color{black}}%
      \expandafter\def\csname LT8\endcsname{\color{black}}%
    \fi
  \fi
    \setlength{\unitlength}{0.0500bp}%
    \ifx\gptboxheight\undefined%
      \newlength{\gptboxheight}%
      \newlength{\gptboxwidth}%
      \newsavebox{\gptboxtext}%
    \fi%
    \setlength{\fboxrule}{0.5pt}%
    \setlength{\fboxsep}{1pt}%
\begin{picture}(10800.00,5400.00)%
    \gplgaddtomacro\gplbacktext{%
      \csname LTb\endcsname%
      \put(990,704){\makebox(0,0)[r]{\strut{}$0$}}%
      \put(990,1590){\makebox(0,0)[r]{\strut{}$0.0001$}}%
      \put(990,2476){\makebox(0,0)[r]{\strut{}$0.0002$}}%
      \put(990,3363){\makebox(0,0)[r]{\strut{}$0.0003$}}%
      \put(990,4249){\makebox(0,0)[r]{\strut{}$0.0004$}}%
      \put(990,5135){\makebox(0,0)[r]{\strut{}$0.0005$}}%
      \put(1300,484){\makebox(0,0){\strut{}$0$}}%
      \put(3085,484){\makebox(0,0){\strut{}$0.2$}}%
      \put(4870,484){\makebox(0,0){\strut{}$0.4$}}%
      \put(6655,484){\makebox(0,0){\strut{}$0.6$}}%
      \put(8440,484){\makebox(0,0){\strut{}$0.8$}}%
      \put(10225,484){\makebox(0,0){\strut{}$1$}}%
      \put(5316,2920){\makebox(0,0)[l]{\strut{}$\norm{u-u_h}_{L^2(\Omega)}$}}%
    }%
    \gplgaddtomacro\gplfronttext{%
      \csname LTb\endcsname%
      \put(5762,154){\makebox(0,0){\strut{}$\varepsilon$}}%
      \csname LTb\endcsname%
      \put(3952,4962){\makebox(0,0)[r]{\strut{}$h = \nicefrac{1}{16}$}}%
      \csname LTb\endcsname%
      \put(5863,4962){\makebox(0,0)[r]{\strut{}$h = \nicefrac{1}{32}$}}%
      \csname LTb\endcsname%
      \put(7774,4962){\makebox(0,0)[r]{\strut{}$h = \nicefrac{1}{64}$}}%
    }%
    \gplbacktext
    \put(0,0){\includegraphics{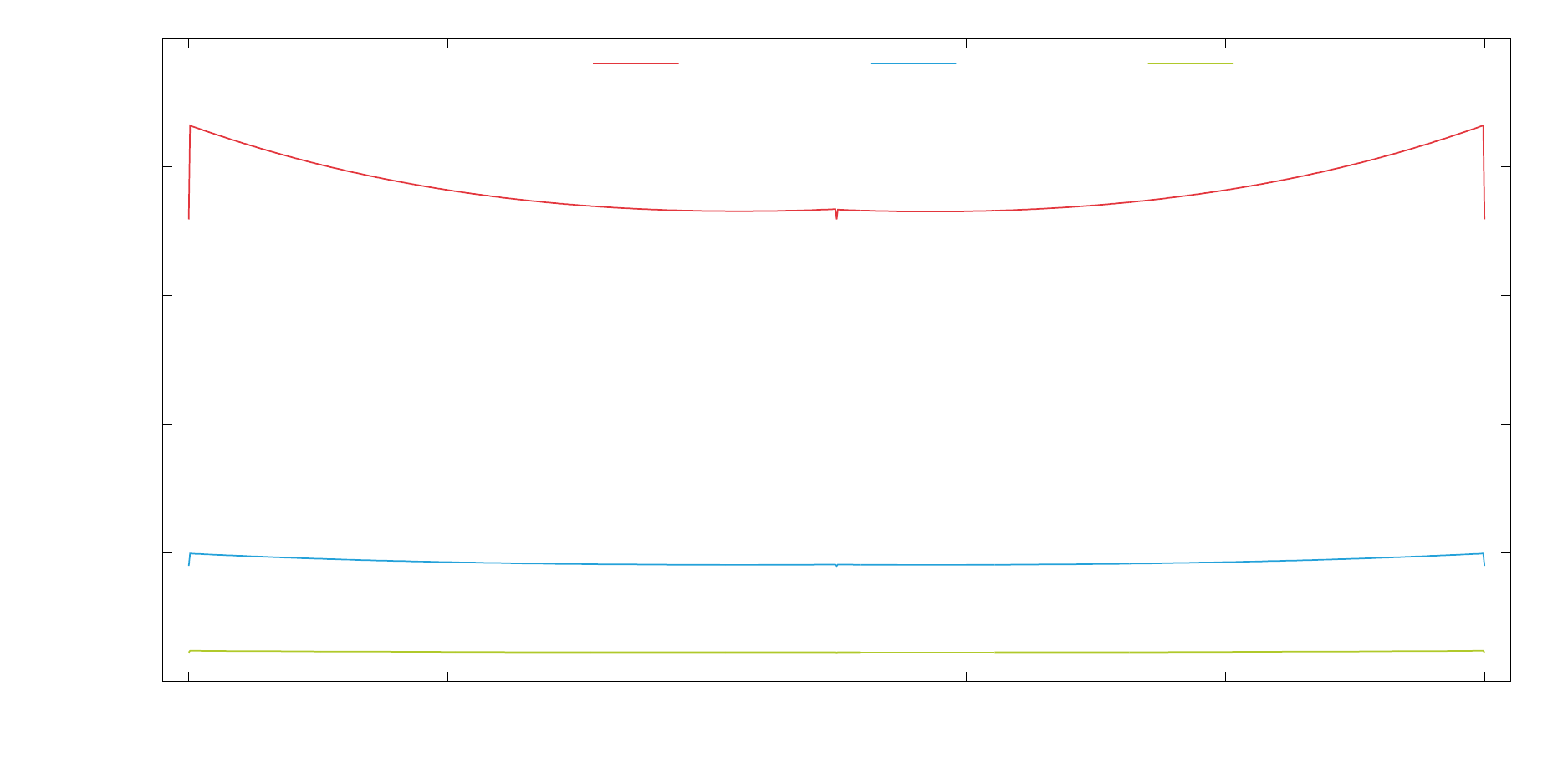}}%
    \gplfronttext
  \end{picture}%
\endgroup

%% file: results/horizontal_cut/plot_L2_error_eps_cross.tex
\begingroup
  \makeatletter
  \providecommand\color[2][]{%
    \GenericError{(gnuplot) \space\space\space\@spaces}{%
      Package color not loaded in conjunction with
      terminal option `colourtext'%
    }{See the gnuplot documentation for explanation.%
    }{Either use 'blacktext' in gnuplot or load the package
      color.sty in LaTeX.}%
    \renewcommand\color[2][]{}%
  }%
  \providecommand\includegraphics[2][]{%
    \GenericError{(gnuplot) \space\space\space\@spaces}{%
      Package graphicx or graphics not loaded%
    }{See the gnuplot documentation for explanation.%
    }{The gnuplot epslatex terminal needs graphicx.sty or graphics.sty.}%
    \renewcommand\includegraphics[2][]{}%
  }%
  \providecommand\rotatebox[2]{#2}%
  \@ifundefined{ifGPcolor}{%
    \newif\ifGPcolor
    \GPcolortrue
  }{}%
  \@ifundefined{ifGPblacktext}{%
    \newif\ifGPblacktext
    \GPblacktexttrue
  }{}%
  \let\gplgaddtomacro\g@addto@macro
  \gdef\gplbacktext{}%
  \gdef\gplfronttext{}%
  \makeatother
  \ifGPblacktext
    \def\colorrgb#1{}%
    \def\colorgray#1{}%
  \else
    \ifGPcolor
      \def\colorrgb#1{\color[rgb]{#1}}%
      \def\colorgray#1{\color[gray]{#1}}%
      \expandafter\def\csname LTw\endcsname{\color{white}}%
      \expandafter\def\csname LTb\endcsname{\color{black}}%
      \expandafter\def\csname LTa\endcsname{\color{black}}%
      \expandafter\def\csname LT0\endcsname{\color[rgb]{1,0,0}}%
      \expandafter\def\csname LT1\endcsname{\color[rgb]{0,1,0}}%
      \expandafter\def\csname LT2\endcsname{\color[rgb]{0,0,1}}%
      \expandafter\def\csname LT3\endcsname{\color[rgb]{1,0,1}}%
      \expandafter\def\csname LT4\endcsname{\color[rgb]{0,1,1}}%
      \expandafter\def\csname LT5\endcsname{\color[rgb]{1,1,0}}%
      \expandafter\def\csname LT6\endcsname{\color[rgb]{0,0,0}}%
      \expandafter\def\csname LT7\endcsname{\color[rgb]{1,0.3,0}}%
      \expandafter\def\csname LT8\endcsname{\color[rgb]{0.5,0.5,0.5}}%
    \else
      \def\colorrgb#1{\color{black}}%
      \def\colorgray#1{\color[gray]{#1}}%
      \expandafter\def\csname LTw\endcsname{\color{white}}%
      \expandafter\def\csname LTb\endcsname{\color{black}}%
      \expandafter\def\csname LTa\endcsname{\color{black}}%
      \expandafter\def\csname LT0\endcsname{\color{black}}%
      \expandafter\def\csname LT1\endcsname{\color{black}}%
      \expandafter\def\csname LT2\endcsname{\color{black}}%
      \expandafter\def\csname LT3\endcsname{\color{black}}%
      \expandafter\def\csname LT4\endcsname{\color{black}}%
      \expandafter\def\csname LT5\endcsname{\color{black}}%
      \expandafter\def\csname LT6\endcsname{\color{black}}%
      \expandafter\def\csname LT7\endcsname{\color{black}}%
      \expandafter\def\csname LT8\endcsname{\color{black}}%
    \fi
  \fi
    \setlength{\unitlength}{0.0500bp}%
    \ifx\gptboxheight\undefined%
      \newlength{\gptboxheight}%
      \newlength{\gptboxwidth}%
      \newsavebox{\gptboxtext}%
    \fi%
    \setlength{\fboxrule}{0.5pt}%
    \setlength{\fboxsep}{1pt}%
\begin{picture}(10800.00,5400.00)%
    \gplgaddtomacro\gplbacktext{%
      \csname LTb\endcsname%
      \put(990,704){\makebox(0,0)[r]{\strut{}$0$}}%
      \put(990,2181){\makebox(0,0)[r]{\strut{}$0.0001$}}%
      \put(990,3658){\makebox(0,0)[r]{\strut{}$0.0002$}}%
      \put(990,5135){\makebox(0,0)[r]{\strut{}$0.0003$}}%
      \put(1300,484){\makebox(0,0){\strut{}$0$}}%
      \put(3085,484){\makebox(0,0){\strut{}$0.2$}}%
      \put(4870,484){\makebox(0,0){\strut{}$0.4$}}%
      \put(6655,484){\makebox(0,0){\strut{}$0.6$}}%
      \put(8440,484){\makebox(0,0){\strut{}$0.8$}}%
      \put(10225,484){\makebox(0,0){\strut{}$1$}}%
      \put(5316,2624){\makebox(0,0)[l]{\strut{}$\norm{u-u_h}_{L^2(\Omega)}$}}%
    }%
    \gplgaddtomacro\gplfronttext{%
      \csname LTb\endcsname%
      \put(5762,154){\makebox(0,0){\strut{}$\varepsilon$}}%
      \csname LTb\endcsname%
      \put(3952,4962){\makebox(0,0)[r]{\strut{}$h = \nicefrac{1}{16}$}}%
      \csname LTb\endcsname%
      \put(5863,4962){\makebox(0,0)[r]{\strut{}$h = \nicefrac{1}{32}$}}%
      \csname LTb\endcsname%
      \put(7774,4962){\makebox(0,0)[r]{\strut{}$h = \nicefrac{1}{64}$}}%
    }%
    \gplbacktext
    \put(0,0){\includegraphics{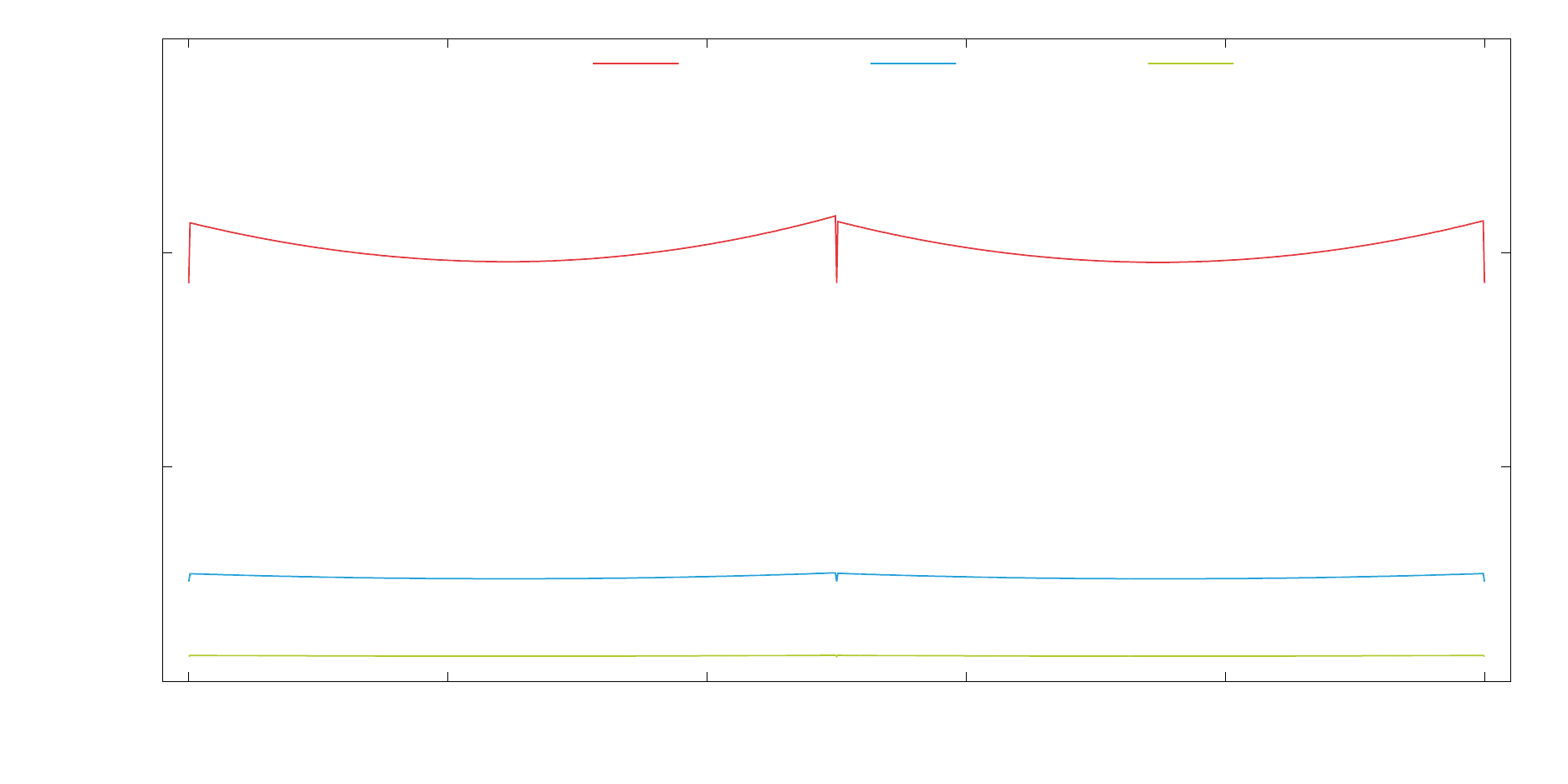}}%
    \gplfronttext
  \end{picture}%
\endgroup

%% file: results/horizontal_cut/plot_L2_error_eps_cross_prod.tex
\begingroup
  \makeatletter
  \providecommand\color[2][]{%
    \GenericError{(gnuplot) \space\space\space\@spaces}{%
      Package color not loaded in conjunction with
      terminal option `colourtext'%
    }{See the gnuplot documentation for explanation.%
    }{Either use 'blacktext' in gnuplot or load the package
      color.sty in LaTeX.}%
    \renewcommand\color[2][]{}%
  }%
  \providecommand\includegraphics[2][]{%
    \GenericError{(gnuplot) \space\space\space\@spaces}{%
      Package graphicx or graphics not loaded%
    }{See the gnuplot documentation for explanation.%
    }{The gnuplot epslatex terminal needs graphicx.sty or graphics.sty.}%
    \renewcommand\includegraphics[2][]{}%
  }%
  \providecommand\rotatebox[2]{#2}%
  \@ifundefined{ifGPcolor}{%
    \newif\ifGPcolor
    \GPcolortrue
  }{}%
  \@ifundefined{ifGPblacktext}{%
    \newif\ifGPblacktext
    \GPblacktexttrue
  }{}%
  \let\gplgaddtomacro\g@addto@macro
  \gdef\gplbacktext{}%
  \gdef\gplfronttext{}%
  \makeatother
  \ifGPblacktext
    \def\colorrgb#1{}%
    \def\colorgray#1{}%
  \else
    \ifGPcolor
      \def\colorrgb#1{\color[rgb]{#1}}%
      \def\colorgray#1{\color[gray]{#1}}%
      \expandafter\def\csname LTw\endcsname{\color{white}}%
      \expandafter\def\csname LTb\endcsname{\color{black}}%
      \expandafter\def\csname LTa\endcsname{\color{black}}%
      \expandafter\def\csname LT0\endcsname{\color[rgb]{1,0,0}}%
      \expandafter\def\csname LT1\endcsname{\color[rgb]{0,1,0}}%
      \expandafter\def\csname LT2\endcsname{\color[rgb]{0,0,1}}%
      \expandafter\def\csname LT3\endcsname{\color[rgb]{1,0,1}}%
      \expandafter\def\csname LT4\endcsname{\color[rgb]{0,1,1}}%
      \expandafter\def\csname LT5\endcsname{\color[rgb]{1,1,0}}%
      \expandafter\def\csname LT6\endcsname{\color[rgb]{0,0,0}}%
      \expandafter\def\csname LT7\endcsname{\color[rgb]{1,0.3,0}}%
      \expandafter\def\csname LT8\endcsname{\color[rgb]{0.5,0.5,0.5}}%
    \else
      \def\colorrgb#1{\color{black}}%
      \def\colorgray#1{\color[gray]{#1}}%
      \expandafter\def\csname LTw\endcsname{\color{white}}%
      \expandafter\def\csname LTb\endcsname{\color{black}}%
      \expandafter\def\csname LTa\endcsname{\color{black}}%
      \expandafter\def\csname LT0\endcsname{\color{black}}%
      \expandafter\def\csname LT1\endcsname{\color{black}}%
      \expandafter\def\csname LT2\endcsname{\color{black}}%
      \expandafter\def\csname LT3\endcsname{\color{black}}%
      \expandafter\def\csname LT4\endcsname{\color{black}}%
      \expandafter\def\csname LT5\endcsname{\color{black}}%
      \expandafter\def\csname LT6\endcsname{\color{black}}%
      \expandafter\def\csname LT7\endcsname{\color{black}}%
      \expandafter\def\csname LT8\endcsname{\color{black}}%
    \fi
  \fi
    \setlength{\unitlength}{0.0500bp}%
    \ifx\gptboxheight\undefined%
      \newlength{\gptboxheight}%
      \newlength{\gptboxwidth}%
      \newsavebox{\gptboxtext}%
    \fi%
    \setlength{\fboxrule}{0.5pt}%
    \setlength{\fboxsep}{1pt}%
\begin{picture}(10800.00,5400.00)%
    \gplgaddtomacro\gplbacktext{%
      \csname LTb\endcsname%
      \put(990,704){\makebox(0,0)[r]{\strut{}$0$}}%
      \put(990,2181){\makebox(0,0)[r]{\strut{}$0.0001$}}%
      \put(990,3658){\makebox(0,0)[r]{\strut{}$0.0002$}}%
      \put(990,5135){\makebox(0,0)[r]{\strut{}$0.0003$}}%
      \put(1300,484){\makebox(0,0){\strut{}$0$}}%
      \put(3085,484){\makebox(0,0){\strut{}$0.2$}}%
      \put(4870,484){\makebox(0,0){\strut{}$0.4$}}%
      \put(6655,484){\makebox(0,0){\strut{}$0.6$}}%
      \put(8440,484){\makebox(0,0){\strut{}$0.8$}}%
      \put(10225,484){\makebox(0,0){\strut{}$1$}}%
      \put(5316,2624){\makebox(0,0)[l]{\strut{}$\norm{u-u_h}_{L^2(\Omega)}$}}%
    }%
    \gplgaddtomacro\gplfronttext{%
      \csname LTb\endcsname%
      \put(5762,154){\makebox(0,0){\strut{}$\varepsilon$}}%
      \csname LTb\endcsname%
      \put(3952,4962){\makebox(0,0)[r]{\strut{}$h = \nicefrac{1}{16}$}}%
      \csname LTb\endcsname%
      \put(5863,4962){\makebox(0,0)[r]{\strut{}$h = \nicefrac{1}{32}$}}%
      \csname LTb\endcsname%
      \put(7774,4962){\makebox(0,0)[r]{\strut{}$h = \nicefrac{1}{64}$}}%
    }%
    \gplbacktext
    \put(0,0){\includegraphics{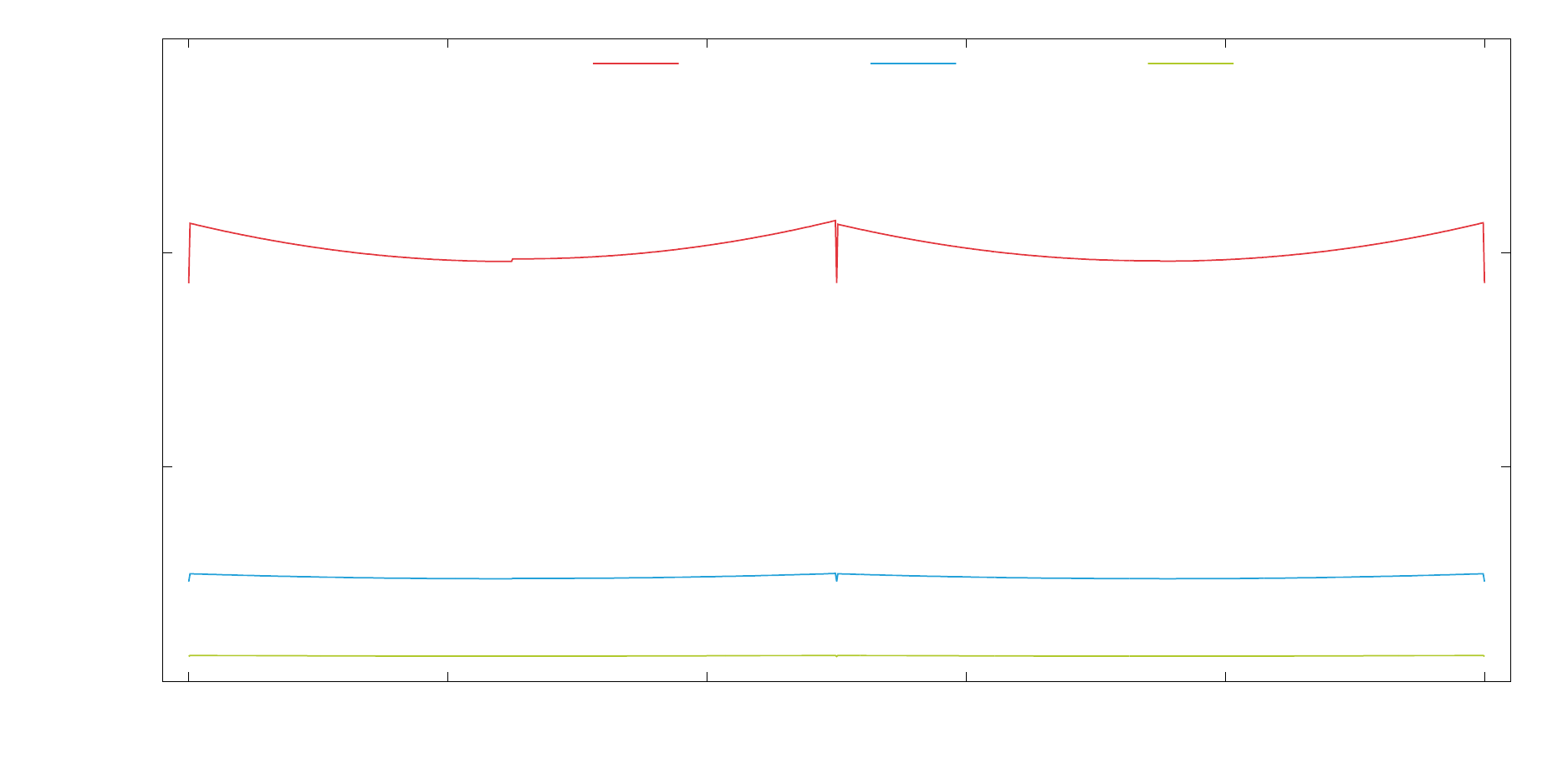}}%
    \gplfronttext
  \end{picture}%
\endgroup

%% file: results/horizontal_cut/plot_H1_error_eps_cross_ohne.tex
\begingroup
  \makeatletter
  \providecommand\color[2][]{%
    \GenericError{(gnuplot) \space\space\space\@spaces}{%
      Package color not loaded in conjunction with
      terminal option `colourtext'%
    }{See the gnuplot documentation for explanation.%
    }{Either use 'blacktext' in gnuplot or load the package
      color.sty in LaTeX.}%
    \renewcommand\color[2][]{}%
  }%
  \providecommand\includegraphics[2][]{%
    \GenericError{(gnuplot) \space\space\space\@spaces}{%
      Package graphicx or graphics not loaded%
    }{See the gnuplot documentation for explanation.%
    }{The gnuplot epslatex terminal needs graphicx.sty or graphics.sty.}%
    \renewcommand\includegraphics[2][]{}%
  }%
  \providecommand\rotatebox[2]{#2}%
  \@ifundefined{ifGPcolor}{%
    \newif\ifGPcolor
    \GPcolortrue
  }{}%
  \@ifundefined{ifGPblacktext}{%
    \newif\ifGPblacktext
    \GPblacktexttrue
  }{}%
  \let\gplgaddtomacro\g@addto@macro
  \gdef\gplbacktext{}%
  \gdef\gplfronttext{}%
  \makeatother
  \ifGPblacktext
    \def\colorrgb#1{}%
    \def\colorgray#1{}%
  \else
    \ifGPcolor
      \def\colorrgb#1{\color[rgb]{#1}}%
      \def\colorgray#1{\color[gray]{#1}}%
      \expandafter\def\csname LTw\endcsname{\color{white}}%
      \expandafter\def\csname LTb\endcsname{\color{black}}%
      \expandafter\def\csname LTa\endcsname{\color{black}}%
      \expandafter\def\csname LT0\endcsname{\color[rgb]{1,0,0}}%
      \expandafter\def\csname LT1\endcsname{\color[rgb]{0,1,0}}%
      \expandafter\def\csname LT2\endcsname{\color[rgb]{0,0,1}}%
      \expandafter\def\csname LT3\endcsname{\color[rgb]{1,0,1}}%
      \expandafter\def\csname LT4\endcsname{\color[rgb]{0,1,1}}%
      \expandafter\def\csname LT5\endcsname{\color[rgb]{1,1,0}}%
      \expandafter\def\csname LT6\endcsname{\color[rgb]{0,0,0}}%
      \expandafter\def\csname LT7\endcsname{\color[rgb]{1,0.3,0}}%
      \expandafter\def\csname LT8\endcsname{\color[rgb]{0.5,0.5,0.5}}%
    \else
      \def\colorrgb#1{\color{black}}%
      \def\colorgray#1{\color[gray]{#1}}%
      \expandafter\def\csname LTw\endcsname{\color{white}}%
      \expandafter\def\csname LTb\endcsname{\color{black}}%
      \expandafter\def\csname LTa\endcsname{\color{black}}%
      \expandafter\def\csname LT0\endcsname{\color{black}}%
      \expandafter\def\csname LT1\endcsname{\color{black}}%
      \expandafter\def\csname LT2\endcsname{\color{black}}%
      \expandafter\def\csname LT3\endcsname{\color{black}}%
      \expandafter\def\csname LT4\endcsname{\color{black}}%
      \expandafter\def\csname LT5\endcsname{\color{black}}%
      \expandafter\def\csname LT6\endcsname{\color{black}}%
      \expandafter\def\csname LT7\endcsname{\color{black}}%
      \expandafter\def\csname LT8\endcsname{\color{black}}%
    \fi
  \fi
    \setlength{\unitlength}{0.0500bp}%
    \ifx\gptboxheight\undefined%
      \newlength{\gptboxheight}%
      \newlength{\gptboxwidth}%
      \newsavebox{\gptboxtext}%
    \fi%
    \setlength{\fboxrule}{0.5pt}%
    \setlength{\fboxsep}{1pt}%
\begin{picture}(10800.00,5400.00)%
    \gplgaddtomacro\gplbacktext{%
      \csname LTb\endcsname%
      \put(726,704){\makebox(0,0)[r]{\strut{}$0.01$}}%
      \put(726,1812){\makebox(0,0)[r]{\strut{}$0.02$}}%
      \put(726,2919){\makebox(0,0)[r]{\strut{}$0.03$}}%
      \put(726,4027){\makebox(0,0)[r]{\strut{}$0.04$}}%
      \put(726,5135){\makebox(0,0)[r]{\strut{}$0.05$}}%
      \put(1042,484){\makebox(0,0){\strut{}$0$}}%
      \put(2877,484){\makebox(0,0){\strut{}$0.2$}}%
      \put(4713,484){\makebox(0,0){\strut{}$0.4$}}%
      \put(6548,484){\makebox(0,0){\strut{}$0.6$}}%
      \put(8384,484){\makebox(0,0){\strut{}$0.8$}}%
      \put(10219,484){\makebox(0,0){\strut{}$1$}}%
      \put(4896,2919){\makebox(0,0)[l]{\strut{}$\norm{\nabla(u-u_h)}_{L^2(\Omega)}$}}%
    }%
    \gplgaddtomacro\gplfronttext{%
      \csname LTb\endcsname%
      \put(5630,154){\makebox(0,0){\strut{}$\varepsilon$}}%
      \csname LTb\endcsname%
      \put(3820,4962){\makebox(0,0)[r]{\strut{}$h = \nicefrac{1}{16}$}}%
      \csname LTb\endcsname%
      \put(5731,4962){\makebox(0,0)[r]{\strut{}$h = \nicefrac{1}{32}$}}%
      \csname LTb\endcsname%
      \put(7642,4962){\makebox(0,0)[r]{\strut{}$h = \nicefrac{1}{64}$}}%
    }%
    \gplbacktext
    \put(0,0){\includegraphics{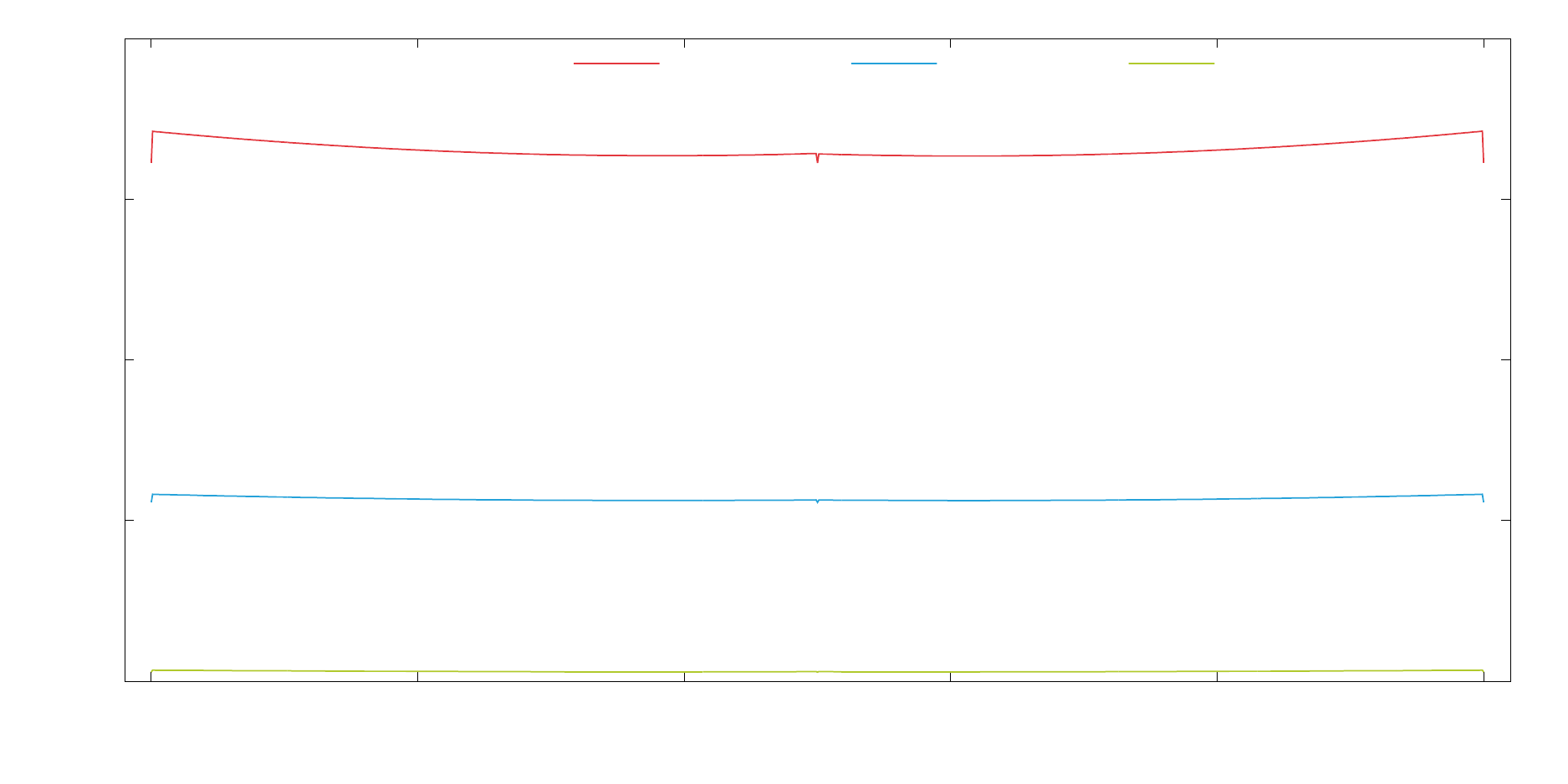}}%
    \gplfronttext
  \end{picture}%
\endgroup

%% file: results/horizontal_cut/plot_H1_error_eps_cross.tex
\begingroup
  \makeatletter
  \providecommand\color[2][]{%
    \GenericError{(gnuplot) \space\space\space\@spaces}{%
      Package color not loaded in conjunction with
      terminal option `colourtext'%
    }{See the gnuplot documentation for explanation.%
    }{Either use 'blacktext' in gnuplot or load the package
      color.sty in LaTeX.}%
    \renewcommand\color[2][]{}%
  }%
  \providecommand\includegraphics[2][]{%
    \GenericError{(gnuplot) \space\space\space\@spaces}{%
      Package graphicx or graphics not loaded%
    }{See the gnuplot documentation for explanation.%
    }{The gnuplot epslatex terminal needs graphicx.sty or graphics.sty.}%
    \renewcommand\includegraphics[2][]{}%
  }%
  \providecommand\rotatebox[2]{#2}%
  \@ifundefined{ifGPcolor}{%
    \newif\ifGPcolor
    \GPcolortrue
  }{}%
  \@ifundefined{ifGPblacktext}{%
    \newif\ifGPblacktext
    \GPblacktexttrue
  }{}%
  \let\gplgaddtomacro\g@addto@macro
  \gdef\gplbacktext{}%
  \gdef\gplfronttext{}%
  \makeatother
  \ifGPblacktext
    \def\colorrgb#1{}%
    \def\colorgray#1{}%
  \else
    \ifGPcolor
      \def\colorrgb#1{\color[rgb]{#1}}%
      \def\colorgray#1{\color[gray]{#1}}%
      \expandafter\def\csname LTw\endcsname{\color{white}}%
      \expandafter\def\csname LTb\endcsname{\color{black}}%
      \expandafter\def\csname LTa\endcsname{\color{black}}%
      \expandafter\def\csname LT0\endcsname{\color[rgb]{1,0,0}}%
      \expandafter\def\csname LT1\endcsname{\color[rgb]{0,1,0}}%
      \expandafter\def\csname LT2\endcsname{\color[rgb]{0,0,1}}%
      \expandafter\def\csname LT3\endcsname{\color[rgb]{1,0,1}}%
      \expandafter\def\csname LT4\endcsname{\color[rgb]{0,1,1}}%
      \expandafter\def\csname LT5\endcsname{\color[rgb]{1,1,0}}%
      \expandafter\def\csname LT6\endcsname{\color[rgb]{0,0,0}}%
      \expandafter\def\csname LT7\endcsname{\color[rgb]{1,0.3,0}}%
      \expandafter\def\csname LT8\endcsname{\color[rgb]{0.5,0.5,0.5}}%
    \else
      \def\colorrgb#1{\color{black}}%
      \def\colorgray#1{\color[gray]{#1}}%
      \expandafter\def\csname LTw\endcsname{\color{white}}%
      \expandafter\def\csname LTb\endcsname{\color{black}}%
      \expandafter\def\csname LTa\endcsname{\color{black}}%
      \expandafter\def\csname LT0\endcsname{\color{black}}%
      \expandafter\def\csname LT1\endcsname{\color{black}}%
      \expandafter\def\csname LT2\endcsname{\color{black}}%
      \expandafter\def\csname LT3\endcsname{\color{black}}%
      \expandafter\def\csname LT4\endcsname{\color{black}}%
      \expandafter\def\csname LT5\endcsname{\color{black}}%
      \expandafter\def\csname LT6\endcsname{\color{black}}%
      \expandafter\def\csname LT7\endcsname{\color{black}}%
      \expandafter\def\csname LT8\endcsname{\color{black}}%
    \fi
  \fi
    \setlength{\unitlength}{0.0500bp}%
    \ifx\gptboxheight\undefined%
      \newlength{\gptboxheight}%
      \newlength{\gptboxwidth}%
      \newsavebox{\gptboxtext}%
    \fi%
    \setlength{\fboxrule}{0.5pt}%
    \setlength{\fboxsep}{1pt}%
\begin{picture}(10800.00,5400.00)%
    \gplgaddtomacro\gplbacktext{%
      \csname LTb\endcsname%
      \put(726,704){\makebox(0,0)[r]{\strut{}$0$}}%
      \put(726,1812){\makebox(0,0)[r]{\strut{}$0.01$}}%
      \put(726,2920){\makebox(0,0)[r]{\strut{}$0.02$}}%
      \put(726,4027){\makebox(0,0)[r]{\strut{}$0.03$}}%
      \put(726,5135){\makebox(0,0)[r]{\strut{}$0.04$}}%
      \put(1042,484){\makebox(0,0){\strut{}$0$}}%
      \put(2877,484){\makebox(0,0){\strut{}$0.2$}}%
      \put(4713,484){\makebox(0,0){\strut{}$0.4$}}%
      \put(6548,484){\makebox(0,0){\strut{}$0.6$}}%
      \put(8384,484){\makebox(0,0){\strut{}$0.8$}}%
      \put(10219,484){\makebox(0,0){\strut{}$1$}}%
      \put(4896,3473){\makebox(0,0)[l]{\strut{}$\norm{\nabla(u-u_h)}_{L^2(\Omega)}$}}%
    }%
    \gplgaddtomacro\gplfronttext{%
      \csname LTb\endcsname%
      \put(5630,154){\makebox(0,0){\strut{}$\varepsilon$}}%
      \csname LTb\endcsname%
      \put(3820,4962){\makebox(0,0)[r]{\strut{}$h = \nicefrac{1}{16}$}}%
      \csname LTb\endcsname%
      \put(5731,4962){\makebox(0,0)[r]{\strut{}$h = \nicefrac{1}{32}$}}%
      \csname LTb\endcsname%
      \put(7642,4962){\makebox(0,0)[r]{\strut{}$h = \nicefrac{1}{64}$}}%
    }%
    \gplbacktext
    \put(0,0){\includegraphics{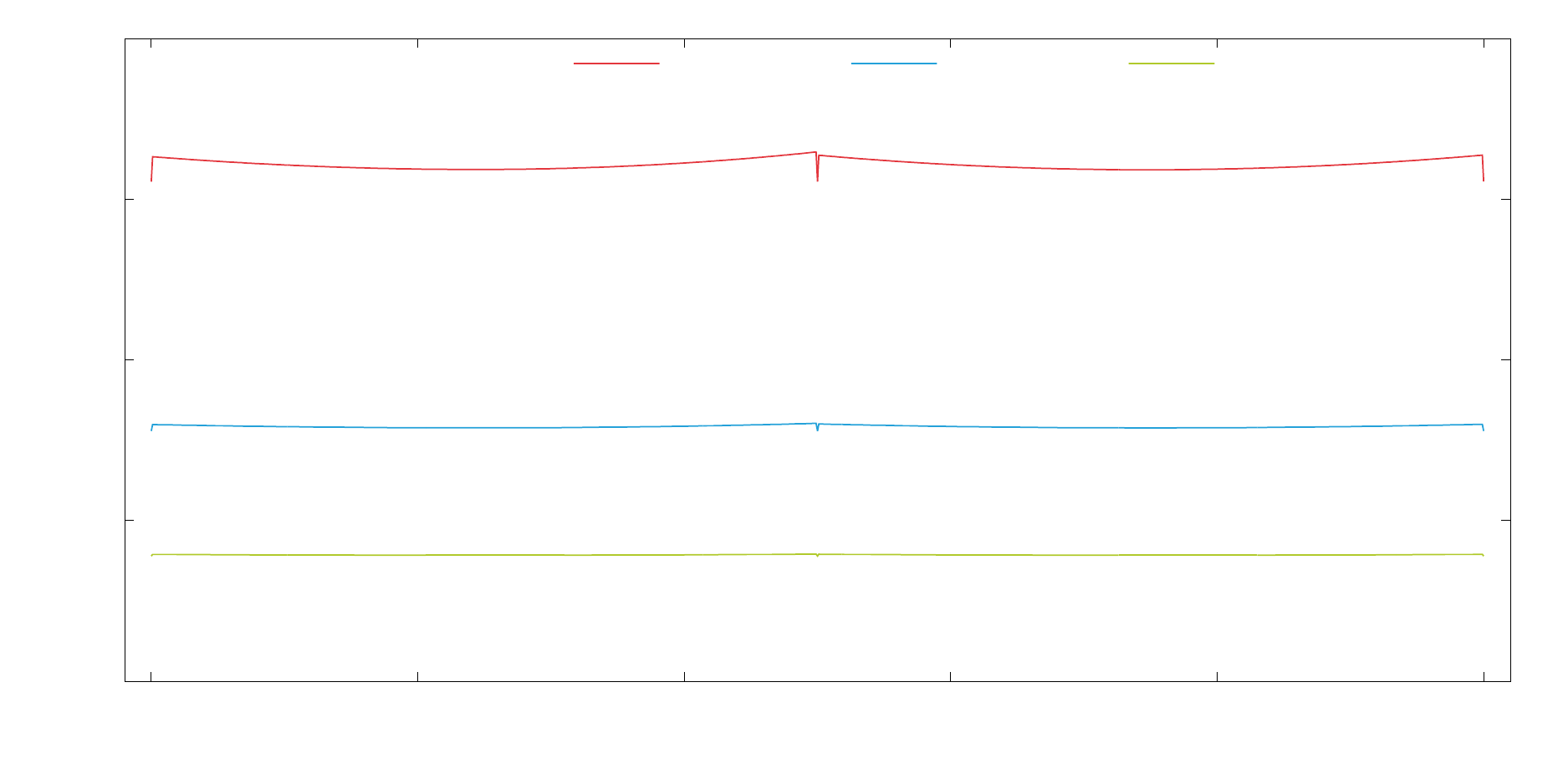}}%
    \gplfronttext
  \end{picture}%
\endgroup

%% file: results/horizontal_cut/plot_H1_error_eps_cross_prod.tex
\begingroup
  \makeatletter
  \providecommand\color[2][]{%
    \GenericError{(gnuplot) \space\space\space\@spaces}{%
      Package color not loaded in conjunction with
      terminal option `colourtext'%
    }{See the gnuplot documentation for explanation.%
    }{Either use 'blacktext' in gnuplot or load the package
      color.sty in LaTeX.}%
    \renewcommand\color[2][]{}%
  }%
  \providecommand\includegraphics[2][]{%
    \GenericError{(gnuplot) \space\space\space\@spaces}{%
      Package graphicx or graphics not loaded%
    }{See the gnuplot documentation for explanation.%
    }{The gnuplot epslatex terminal needs graphicx.sty or graphics.sty.}%
    \renewcommand\includegraphics[2][]{}%
  }%
  \providecommand\rotatebox[2]{#2}%
  \@ifundefined{ifGPcolor}{%
    \newif\ifGPcolor
    \GPcolortrue
  }{}%
  \@ifundefined{ifGPblacktext}{%
    \newif\ifGPblacktext
    \GPblacktexttrue
  }{}%
  \let\gplgaddtomacro\g@addto@macro
  \gdef\gplbacktext{}%
  \gdef\gplfronttext{}%
  \makeatother
  \ifGPblacktext
    \def\colorrgb#1{}%
    \def\colorgray#1{}%
  \else
    \ifGPcolor
      \def\colorrgb#1{\color[rgb]{#1}}%
      \def\colorgray#1{\color[gray]{#1}}%
      \expandafter\def\csname LTw\endcsname{\color{white}}%
      \expandafter\def\csname LTb\endcsname{\color{black}}%
      \expandafter\def\csname LTa\endcsname{\color{black}}%
      \expandafter\def\csname LT0\endcsname{\color[rgb]{1,0,0}}%
      \expandafter\def\csname LT1\endcsname{\color[rgb]{0,1,0}}%
      \expandafter\def\csname LT2\endcsname{\color[rgb]{0,0,1}}%
      \expandafter\def\csname LT3\endcsname{\color[rgb]{1,0,1}}%
      \expandafter\def\csname LT4\endcsname{\color[rgb]{0,1,1}}%
      \expandafter\def\csname LT5\endcsname{\color[rgb]{1,1,0}}%
      \expandafter\def\csname LT6\endcsname{\color[rgb]{0,0,0}}%
      \expandafter\def\csname LT7\endcsname{\color[rgb]{1,0.3,0}}%
      \expandafter\def\csname LT8\endcsname{\color[rgb]{0.5,0.5,0.5}}%
    \else
      \def\colorrgb#1{\color{black}}%
      \def\colorgray#1{\color[gray]{#1}}%
      \expandafter\def\csname LTw\endcsname{\color{white}}%
      \expandafter\def\csname LTb\endcsname{\color{black}}%
      \expandafter\def\csname LTa\endcsname{\color{black}}%
      \expandafter\def\csname LT0\endcsname{\color{black}}%
      \expandafter\def\csname LT1\endcsname{\color{black}}%
      \expandafter\def\csname LT2\endcsname{\color{black}}%
      \expandafter\def\csname LT3\endcsname{\color{black}}%
      \expandafter\def\csname LT4\endcsname{\color{black}}%
      \expandafter\def\csname LT5\endcsname{\color{black}}%
      \expandafter\def\csname LT6\endcsname{\color{black}}%
      \expandafter\def\csname LT7\endcsname{\color{black}}%
      \expandafter\def\csname LT8\endcsname{\color{black}}%
    \fi
  \fi
    \setlength{\unitlength}{0.0500bp}%
    \ifx\gptboxheight\undefined%
      \newlength{\gptboxheight}%
      \newlength{\gptboxwidth}%
      \newsavebox{\gptboxtext}%
    \fi%
    \setlength{\fboxrule}{0.5pt}%
    \setlength{\fboxsep}{1pt}%
\begin{picture}(10800.00,5400.00)%
    \gplgaddtomacro\gplbacktext{%
      \csname LTb\endcsname%
      \put(726,704){\makebox(0,0)[r]{\strut{}$0$}}%
      \put(726,1812){\makebox(0,0)[r]{\strut{}$0.01$}}%
      \put(726,2920){\makebox(0,0)[r]{\strut{}$0.02$}}%
      \put(726,4027){\makebox(0,0)[r]{\strut{}$0.03$}}%
      \put(726,5135){\makebox(0,0)[r]{\strut{}$0.04$}}%
      \put(1042,484){\makebox(0,0){\strut{}$0$}}%
      \put(2877,484){\makebox(0,0){\strut{}$0.2$}}%
      \put(4713,484){\makebox(0,0){\strut{}$0.4$}}%
      \put(6548,484){\makebox(0,0){\strut{}$0.6$}}%
      \put(8384,484){\makebox(0,0){\strut{}$0.8$}}%
      \put(10219,484){\makebox(0,0){\strut{}$1$}}%
      \put(4896,3473){\makebox(0,0)[l]{\strut{}$\norm{\nabla(u-u_h)}_{L^2(\Omega)}$}}%
    }%
    \gplgaddtomacro\gplfronttext{%
      \csname LTb\endcsname%
      \put(5630,154){\makebox(0,0){\strut{}$\varepsilon$}}%
      \csname LTb\endcsname%
      \put(3820,4962){\makebox(0,0)[r]{\strut{}$h = \nicefrac{1}{16}$}}%
      \csname LTb\endcsname%
      \put(5731,4962){\makebox(0,0)[r]{\strut{}$h = \nicefrac{1}{32}$}}%
      \csname LTb\endcsname%
      \put(7642,4962){\makebox(0,0)[r]{\strut{}$h = \nicefrac{1}{64}$}}%
    }%
    \gplbacktext
    \put(0,0){\includegraphics{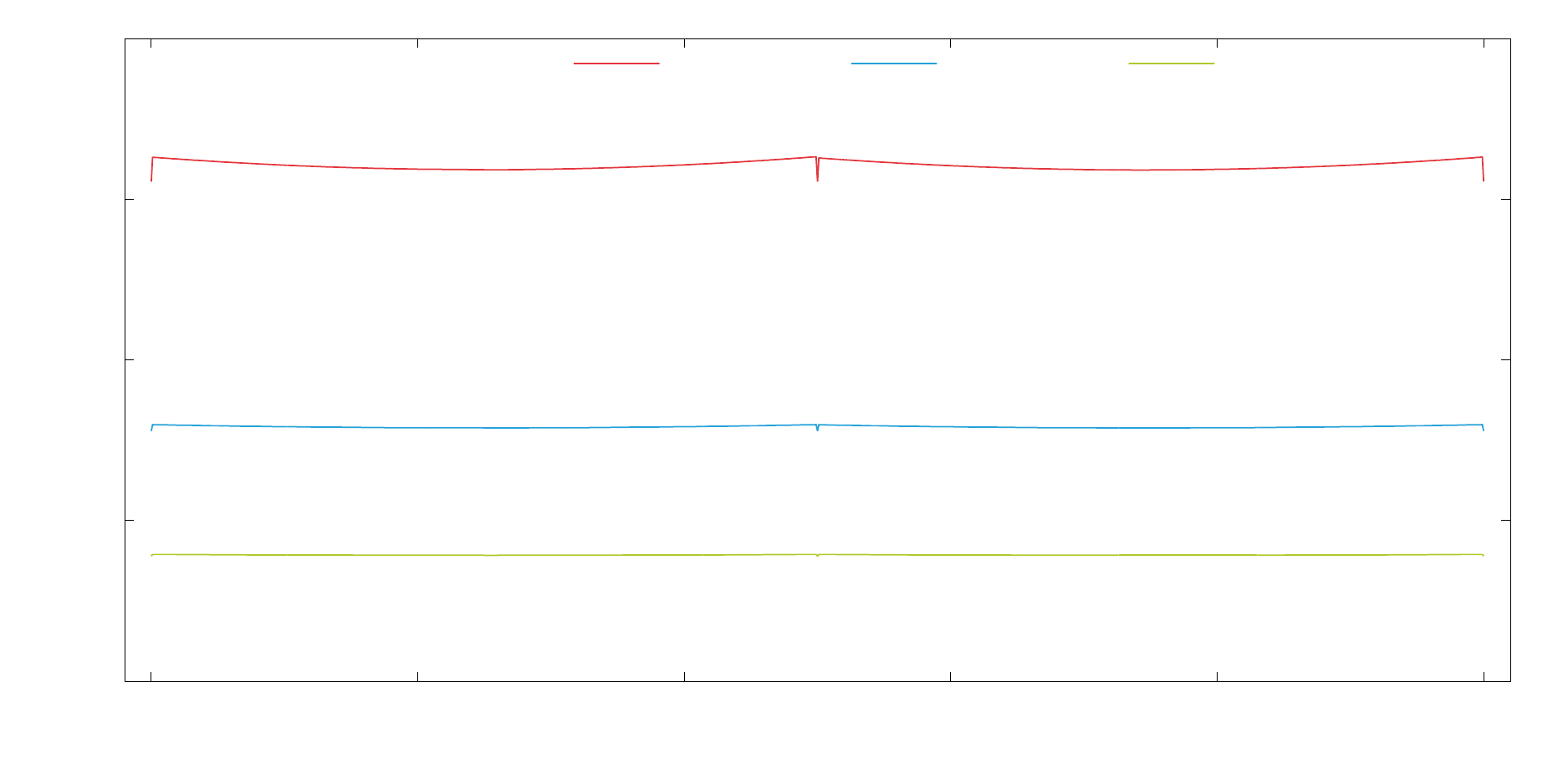}}%
    \gplfronttext
  \end{picture}%
\endgroup

%% file: results/tilted_interface/plot_L2_error_eps_cross_ohne.tex
\begingroup
  \makeatletter
  \providecommand\color[2][]{%
    \GenericError{(gnuplot) \space\space\space\@spaces}{%
      Package color not loaded in conjunction with
      terminal option `colourtext'%
    }{See the gnuplot documentation for explanation.%
    }{Either use 'blacktext' in gnuplot or load the package
      color.sty in LaTeX.}%
    \renewcommand\color[2][]{}%
  }%
  \providecommand\includegraphics[2][]{%
    \GenericError{(gnuplot) \space\space\space\@spaces}{%
      Package graphicx or graphics not loaded%
    }{See the gnuplot documentation for explanation.%
    }{The gnuplot epslatex terminal needs graphicx.sty or graphics.sty.}%
    \renewcommand\includegraphics[2][]{}%
  }%
  \providecommand\rotatebox[2]{#2}%
  \@ifundefined{ifGPcolor}{%
    \newif\ifGPcolor
    \GPcolortrue
  }{}%
  \@ifundefined{ifGPblacktext}{%
    \newif\ifGPblacktext
    \GPblacktexttrue
  }{}%
  \let\gplgaddtomacro\g@addto@macro
  \gdef\gplbacktext{}%
  \gdef\gplfronttext{}%
  \makeatother
  \ifGPblacktext
    \def\colorrgb#1{}%
    \def\colorgray#1{}%
  \else
    \ifGPcolor
      \def\colorrgb#1{\color[rgb]{#1}}%
      \def\colorgray#1{\color[gray]{#1}}%
      \expandafter\def\csname LTw\endcsname{\color{white}}%
      \expandafter\def\csname LTb\endcsname{\color{black}}%
      \expandafter\def\csname LTa\endcsname{\color{black}}%
      \expandafter\def\csname LT0\endcsname{\color[rgb]{1,0,0}}%
      \expandafter\def\csname LT1\endcsname{\color[rgb]{0,1,0}}%
      \expandafter\def\csname LT2\endcsname{\color[rgb]{0,0,1}}%
      \expandafter\def\csname LT3\endcsname{\color[rgb]{1,0,1}}%
      \expandafter\def\csname LT4\endcsname{\color[rgb]{0,1,1}}%
      \expandafter\def\csname LT5\endcsname{\color[rgb]{1,1,0}}%
      \expandafter\def\csname LT6\endcsname{\color[rgb]{0,0,0}}%
      \expandafter\def\csname LT7\endcsname{\color[rgb]{1,0.3,0}}%
      \expandafter\def\csname LT8\endcsname{\color[rgb]{0.5,0.5,0.5}}%
    \else
      \def\colorrgb#1{\color{black}}%
      \def\colorgray#1{\color[gray]{#1}}%
      \expandafter\def\csname LTw\endcsname{\color{white}}%
      \expandafter\def\csname LTb\endcsname{\color{black}}%
      \expandafter\def\csname LTa\endcsname{\color{black}}%
      \expandafter\def\csname LT0\endcsname{\color{black}}%
      \expandafter\def\csname LT1\endcsname{\color{black}}%
      \expandafter\def\csname LT2\endcsname{\color{black}}%
      \expandafter\def\csname LT3\endcsname{\color{black}}%
      \expandafter\def\csname LT4\endcsname{\color{black}}%
      \expandafter\def\csname LT5\endcsname{\color{black}}%
      \expandafter\def\csname LT6\endcsname{\color{black}}%
      \expandafter\def\csname LT7\endcsname{\color{black}}%
      \expandafter\def\csname LT8\endcsname{\color{black}}%
    \fi
  \fi
    \setlength{\unitlength}{0.0500bp}%
    \ifx\gptboxheight\undefined%
      \newlength{\gptboxheight}%
      \newlength{\gptboxwidth}%
      \newsavebox{\gptboxtext}%
    \fi%
    \setlength{\fboxrule}{0.5pt}%
    \setlength{\fboxsep}{1pt}%
\begin{picture}(10800.00,5400.00)%
    \gplgaddtomacro\gplbacktext{%
      \csname LTb\endcsname%
      \put(858,704){\makebox(0,0)[r]{\strut{}$0$}}%
      \put(858,2181){\makebox(0,0)[r]{\strut{}$0.002$}}%
      \put(858,3658){\makebox(0,0)[r]{\strut{}$0.004$}}%
      \put(858,5135){\makebox(0,0)[r]{\strut{}$0.006$}}%
      \put(1049,484){\makebox(0,0){\strut{}0}}%
      \put(5697,484){\makebox(0,0){\strut{}$\nicefrac{\pi}{2}$}}%
      \put(10344,484){\makebox(0,0){\strut{}$\pi$}}%
      \put(5105,2550){\makebox(0,0)[l]{\strut{}$\norm{u-u_h}_{L^2(\Omega)}$}}%
    }%
    \gplgaddtomacro\gplfronttext{%
      \csname LTb\endcsname%
      \put(5696,154){\makebox(0,0){\strut{}$\alpha$}}%
      \csname LTb\endcsname%
      \put(3886,4962){\makebox(0,0)[r]{\strut{}$h = \nicefrac{1}{16}$}}%
      \csname LTb\endcsname%
      \put(5797,4962){\makebox(0,0)[r]{\strut{}$h = \nicefrac{1}{32}$}}%
      \csname LTb\endcsname%
      \put(7708,4962){\makebox(0,0)[r]{\strut{}$h = \nicefrac{1}{64}$}}%
    }%
    \gplbacktext
    \put(0,0){\includegraphics{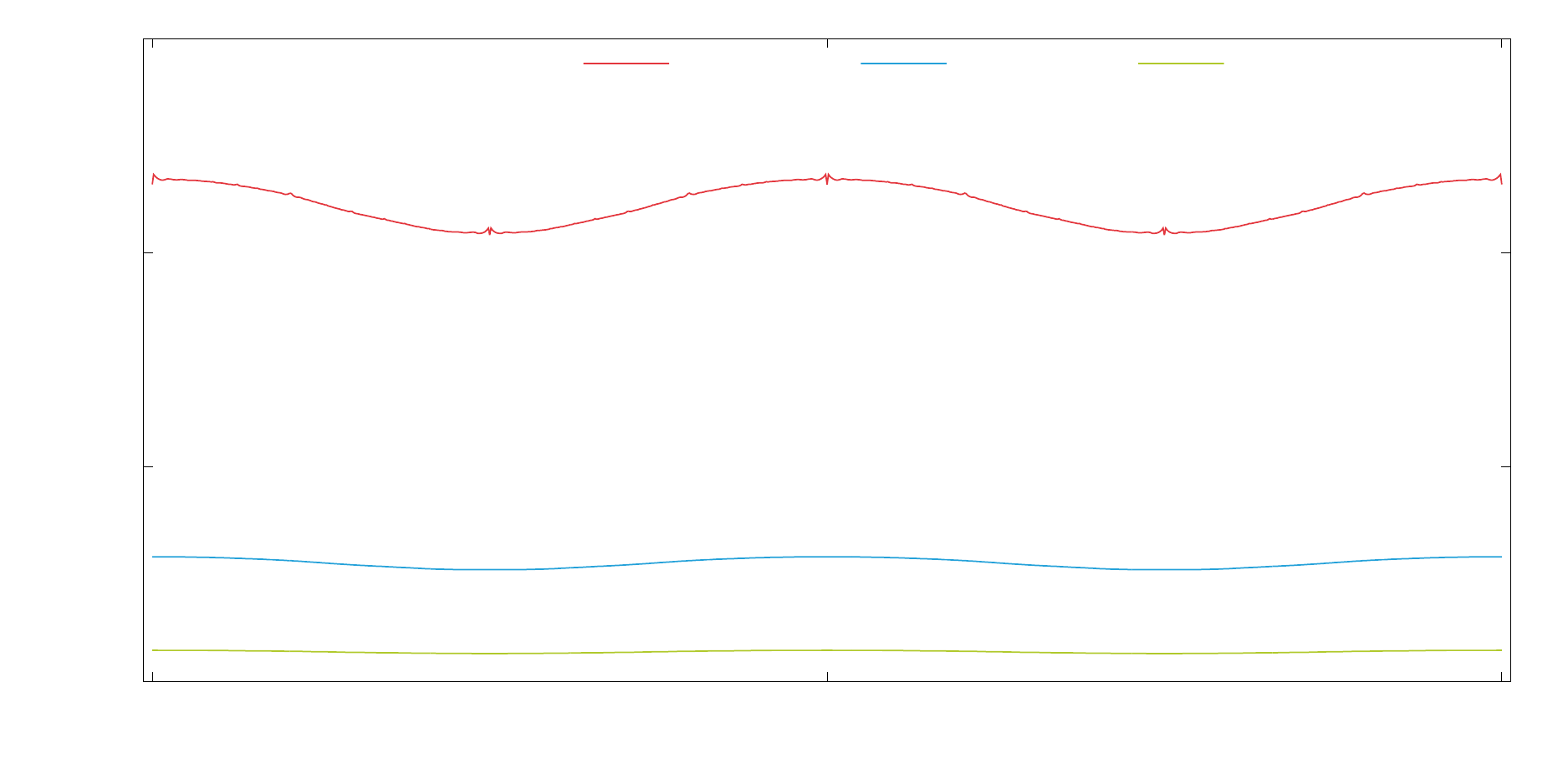}}%
    \gplfronttext
  \end{picture}%
\endgroup

%% file: results/tilted_interface/plot_L2_error_eps_cross.tex
\begingroup
  \makeatletter
  \providecommand\color[2][]{%
    \GenericError{(gnuplot) \space\space\space\@spaces}{%
      Package color not loaded in conjunction with
      terminal option `colourtext'%
    }{See the gnuplot documentation for explanation.%
    }{Either use 'blacktext' in gnuplot or load the package
      color.sty in LaTeX.}%
    \renewcommand\color[2][]{}%
  }%
  \providecommand\includegraphics[2][]{%
    \GenericError{(gnuplot) \space\space\space\@spaces}{%
      Package graphicx or graphics not loaded%
    }{See the gnuplot documentation for explanation.%
    }{The gnuplot epslatex terminal needs graphicx.sty or graphics.sty.}%
    \renewcommand\includegraphics[2][]{}%
  }%
  \providecommand\rotatebox[2]{#2}%
  \@ifundefined{ifGPcolor}{%
    \newif\ifGPcolor
    \GPcolortrue
  }{}%
  \@ifundefined{ifGPblacktext}{%
    \newif\ifGPblacktext
    \GPblacktexttrue
  }{}%
  \let\gplgaddtomacro\g@addto@macro
  \gdef\gplbacktext{}%
  \gdef\gplfronttext{}%
  \makeatother
  \ifGPblacktext
    \def\colorrgb#1{}%
    \def\colorgray#1{}%
  \else
    \ifGPcolor
      \def\colorrgb#1{\color[rgb]{#1}}%
      \def\colorgray#1{\color[gray]{#1}}%
      \expandafter\def\csname LTw\endcsname{\color{white}}%
      \expandafter\def\csname LTb\endcsname{\color{black}}%
      \expandafter\def\csname LTa\endcsname{\color{black}}%
      \expandafter\def\csname LT0\endcsname{\color[rgb]{1,0,0}}%
      \expandafter\def\csname LT1\endcsname{\color[rgb]{0,1,0}}%
      \expandafter\def\csname LT2\endcsname{\color[rgb]{0,0,1}}%
      \expandafter\def\csname LT3\endcsname{\color[rgb]{1,0,1}}%
      \expandafter\def\csname LT4\endcsname{\color[rgb]{0,1,1}}%
      \expandafter\def\csname LT5\endcsname{\color[rgb]{1,1,0}}%
      \expandafter\def\csname LT6\endcsname{\color[rgb]{0,0,0}}%
      \expandafter\def\csname LT7\endcsname{\color[rgb]{1,0.3,0}}%
      \expandafter\def\csname LT8\endcsname{\color[rgb]{0.5,0.5,0.5}}%
    \else
      \def\colorrgb#1{\color{black}}%
      \def\colorgray#1{\color[gray]{#1}}%
      \expandafter\def\csname LTw\endcsname{\color{white}}%
      \expandafter\def\csname LTb\endcsname{\color{black}}%
      \expandafter\def\csname LTa\endcsname{\color{black}}%
      \expandafter\def\csname LT0\endcsname{\color{black}}%
      \expandafter\def\csname LT1\endcsname{\color{black}}%
      \expandafter\def\csname LT2\endcsname{\color{black}}%
      \expandafter\def\csname LT3\endcsname{\color{black}}%
      \expandafter\def\csname LT4\endcsname{\color{black}}%
      \expandafter\def\csname LT5\endcsname{\color{black}}%
      \expandafter\def\csname LT6\endcsname{\color{black}}%
      \expandafter\def\csname LT7\endcsname{\color{black}}%
      \expandafter\def\csname LT8\endcsname{\color{black}}%
    \fi
  \fi
    \setlength{\unitlength}{0.0500bp}%
    \ifx\gptboxheight\undefined%
      \newlength{\gptboxheight}%
      \newlength{\gptboxwidth}%
      \newsavebox{\gptboxtext}%
    \fi%
    \setlength{\fboxrule}{0.5pt}%
    \setlength{\fboxsep}{1pt}%
\begin{picture}(10800.00,5400.00)%
    \gplgaddtomacro\gplbacktext{%
      \csname LTb\endcsname%
      \put(858,704){\makebox(0,0)[r]{\strut{}$0$}}%
      \put(858,2181){\makebox(0,0)[r]{\strut{}$0.002$}}%
      \put(858,3658){\makebox(0,0)[r]{\strut{}$0.004$}}%
      \put(858,5135){\makebox(0,0)[r]{\strut{}$0.006$}}%
      \put(1049,484){\makebox(0,0){\strut{}0}}%
      \put(5697,484){\makebox(0,0){\strut{}$\nicefrac{\pi}{2}$}}%
      \put(10344,484){\makebox(0,0){\strut{}$\pi$}}%
      \put(5105,2550){\makebox(0,0)[l]{\strut{}$\norm{u-u_h}_{L^2(\Omega)}$}}%
    }%
    \gplgaddtomacro\gplfronttext{%
      \csname LTb\endcsname%
      \put(5696,154){\makebox(0,0){\strut{}$\alpha$}}%
      \csname LTb\endcsname%
      \put(3886,4962){\makebox(0,0)[r]{\strut{}$h = \nicefrac{1}{16}$}}%
      \csname LTb\endcsname%
      \put(5797,4962){\makebox(0,0)[r]{\strut{}$h = \nicefrac{1}{32}$}}%
      \csname LTb\endcsname%
      \put(7708,4962){\makebox(0,0)[r]{\strut{}$h = \nicefrac{1}{64}$}}%
    }%
    \gplbacktext
    \put(0,0){\includegraphics{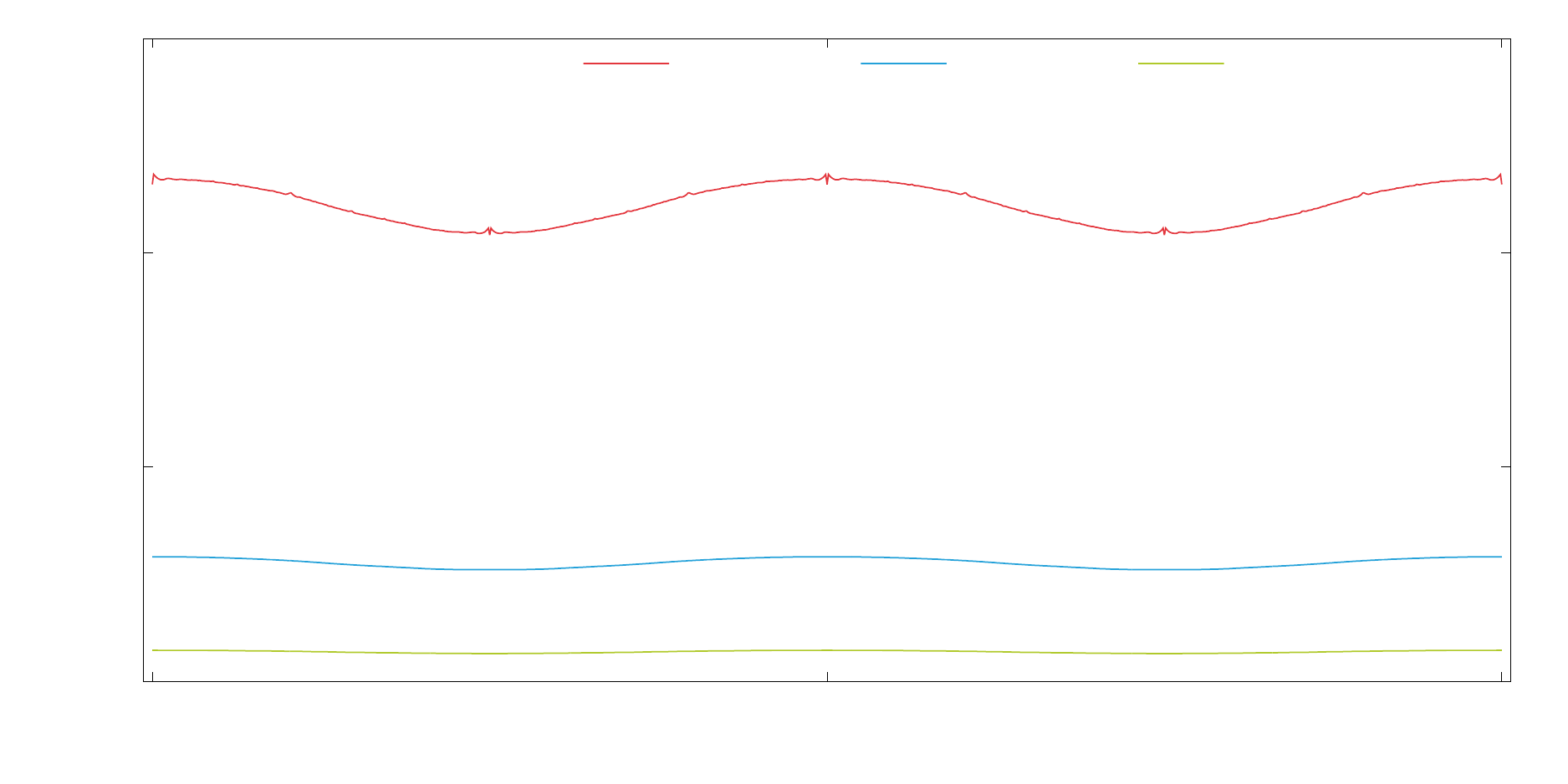}}%
    \gplfronttext
  \end{picture}%
\endgroup

%% file: results/tilted_interface/plot_L2_error_eps_cross_prod.tex
\begingroup
  \makeatletter
  \providecommand\color[2][]{%
    \GenericError{(gnuplot) \space\space\space\@spaces}{%
      Package color not loaded in conjunction with
      terminal option `colourtext'%
    }{See the gnuplot documentation for explanation.%
    }{Either use 'blacktext' in gnuplot or load the package
      color.sty in LaTeX.}%
    \renewcommand\color[2][]{}%
  }%
  \providecommand\includegraphics[2][]{%
    \GenericError{(gnuplot) \space\space\space\@spaces}{%
      Package graphicx or graphics not loaded%
    }{See the gnuplot documentation for explanation.%
    }{The gnuplot epslatex terminal needs graphicx.sty or graphics.sty.}%
    \renewcommand\includegraphics[2][]{}%
  }%
  \providecommand\rotatebox[2]{#2}%
  \@ifundefined{ifGPcolor}{%
    \newif\ifGPcolor
    \GPcolortrue
  }{}%
  \@ifundefined{ifGPblacktext}{%
    \newif\ifGPblacktext
    \GPblacktexttrue
  }{}%
  \let\gplgaddtomacro\g@addto@macro
  \gdef\gplbacktext{}%
  \gdef\gplfronttext{}%
  \makeatother
  \ifGPblacktext
    \def\colorrgb#1{}%
    \def\colorgray#1{}%
  \else
    \ifGPcolor
      \def\colorrgb#1{\color[rgb]{#1}}%
      \def\colorgray#1{\color[gray]{#1}}%
      \expandafter\def\csname LTw\endcsname{\color{white}}%
      \expandafter\def\csname LTb\endcsname{\color{black}}%
      \expandafter\def\csname LTa\endcsname{\color{black}}%
      \expandafter\def\csname LT0\endcsname{\color[rgb]{1,0,0}}%
      \expandafter\def\csname LT1\endcsname{\color[rgb]{0,1,0}}%
      \expandafter\def\csname LT2\endcsname{\color[rgb]{0,0,1}}%
      \expandafter\def\csname LT3\endcsname{\color[rgb]{1,0,1}}%
      \expandafter\def\csname LT4\endcsname{\color[rgb]{0,1,1}}%
      \expandafter\def\csname LT5\endcsname{\color[rgb]{1,1,0}}%
      \expandafter\def\csname LT6\endcsname{\color[rgb]{0,0,0}}%
      \expandafter\def\csname LT7\endcsname{\color[rgb]{1,0.3,0}}%
      \expandafter\def\csname LT8\endcsname{\color[rgb]{0.5,0.5,0.5}}%
    \else
      \def\colorrgb#1{\color{black}}%
      \def\colorgray#1{\color[gray]{#1}}%
      \expandafter\def\csname LTw\endcsname{\color{white}}%
      \expandafter\def\csname LTb\endcsname{\color{black}}%
      \expandafter\def\csname LTa\endcsname{\color{black}}%
      \expandafter\def\csname LT0\endcsname{\color{black}}%
      \expandafter\def\csname LT1\endcsname{\color{black}}%
      \expandafter\def\csname LT2\endcsname{\color{black}}%
      \expandafter\def\csname LT3\endcsname{\color{black}}%
      \expandafter\def\csname LT4\endcsname{\color{black}}%
      \expandafter\def\csname LT5\endcsname{\color{black}}%
      \expandafter\def\csname LT6\endcsname{\color{black}}%
      \expandafter\def\csname LT7\endcsname{\color{black}}%
      \expandafter\def\csname LT8\endcsname{\color{black}}%
    \fi
  \fi
    \setlength{\unitlength}{0.0500bp}%
    \ifx\gptboxheight\undefined%
      \newlength{\gptboxheight}%
      \newlength{\gptboxwidth}%
      \newsavebox{\gptboxtext}%
    \fi%
    \setlength{\fboxrule}{0.5pt}%
    \setlength{\fboxsep}{1pt}%
\begin{picture}(10800.00,5400.00)%
    \gplgaddtomacro\gplbacktext{%
      \csname LTb\endcsname%
      \put(858,704){\makebox(0,0)[r]{\strut{}$0$}}%
      \put(858,2181){\makebox(0,0)[r]{\strut{}$0.002$}}%
      \put(858,3658){\makebox(0,0)[r]{\strut{}$0.004$}}%
      \put(858,5135){\makebox(0,0)[r]{\strut{}$0.006$}}%
      \put(1049,484){\makebox(0,0){\strut{}0}}%
      \put(5697,484){\makebox(0,0){\strut{}$\nicefrac{\pi}{2}$}}%
      \put(10344,484){\makebox(0,0){\strut{}$\pi$}}%
      \put(5105,2550){\makebox(0,0)[l]{\strut{}$\norm{u-u_h}_{L^2(\Omega)}$}}%
    }%
    \gplgaddtomacro\gplfronttext{%
      \csname LTb\endcsname%
      \put(5696,154){\makebox(0,0){\strut{}$\alpha$}}%
      \csname LTb\endcsname%
      \put(3886,4962){\makebox(0,0)[r]{\strut{}$h = \nicefrac{1}{16}$}}%
      \csname LTb\endcsname%
      \put(5797,4962){\makebox(0,0)[r]{\strut{}$h = \nicefrac{1}{32}$}}%
      \csname LTb\endcsname%
      \put(7708,4962){\makebox(0,0)[r]{\strut{}$h = \nicefrac{1}{64}$}}%
    }%
    \gplbacktext
    \put(0,0){\includegraphics{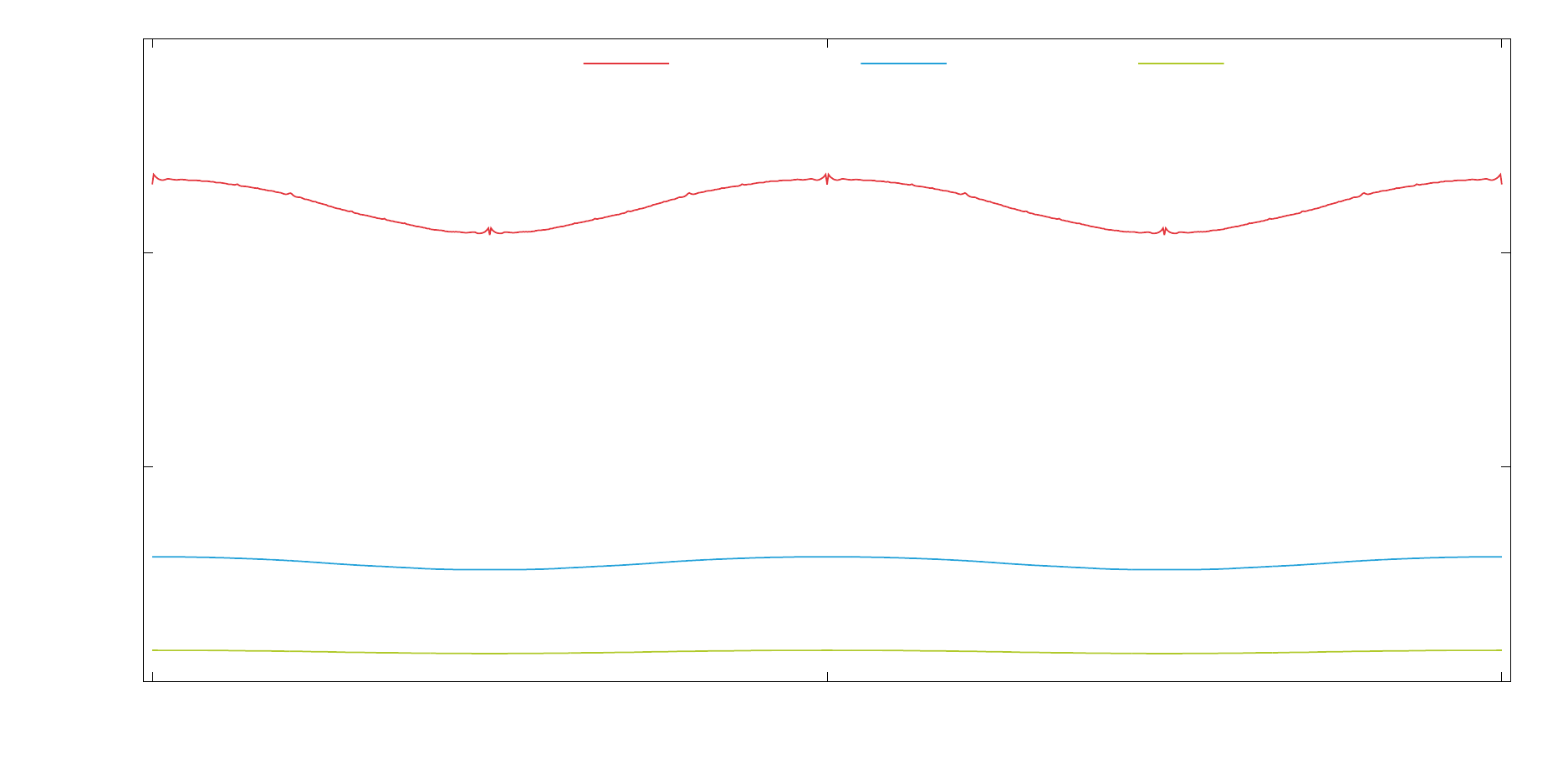}}%
    \gplfronttext
  \end{picture}%
\endgroup

%% file: results/tilted_interface/plot_H1_error_eps_cross_ohne.tex
\begingroup
  \makeatletter
  \providecommand\color[2][]{%
    \GenericError{(gnuplot) \space\space\space\@spaces}{%
      Package color not loaded in conjunction with
      terminal option `colourtext'%
    }{See the gnuplot documentation for explanation.%
    }{Either use 'blacktext' in gnuplot or load the package
      color.sty in LaTeX.}%
    \renewcommand\color[2][]{}%
  }%
  \providecommand\includegraphics[2][]{%
    \GenericError{(gnuplot) \space\space\space\@spaces}{%
      Package graphicx or graphics not loaded%
    }{See the gnuplot documentation for explanation.%
    }{The gnuplot epslatex terminal needs graphicx.sty or graphics.sty.}%
    \renewcommand\includegraphics[2][]{}%
  }%
  \providecommand\rotatebox[2]{#2}%
  \@ifundefined{ifGPcolor}{%
    \newif\ifGPcolor
    \GPcolortrue
  }{}%
  \@ifundefined{ifGPblacktext}{%
    \newif\ifGPblacktext
    \GPblacktexttrue
  }{}%
  \let\gplgaddtomacro\g@addto@macro
  \gdef\gplbacktext{}%
  \gdef\gplfronttext{}%
  \makeatother
  \ifGPblacktext
    \def\colorrgb#1{}%
    \def\colorgray#1{}%
  \else
    \ifGPcolor
      \def\colorrgb#1{\color[rgb]{#1}}%
      \def\colorgray#1{\color[gray]{#1}}%
      \expandafter\def\csname LTw\endcsname{\color{white}}%
      \expandafter\def\csname LTb\endcsname{\color{black}}%
      \expandafter\def\csname LTa\endcsname{\color{black}}%
      \expandafter\def\csname LT0\endcsname{\color[rgb]{1,0,0}}%
      \expandafter\def\csname LT1\endcsname{\color[rgb]{0,1,0}}%
      \expandafter\def\csname LT2\endcsname{\color[rgb]{0,0,1}}%
      \expandafter\def\csname LT3\endcsname{\color[rgb]{1,0,1}}%
      \expandafter\def\csname LT4\endcsname{\color[rgb]{0,1,1}}%
      \expandafter\def\csname LT5\endcsname{\color[rgb]{1,1,0}}%
      \expandafter\def\csname LT6\endcsname{\color[rgb]{0,0,0}}%
      \expandafter\def\csname LT7\endcsname{\color[rgb]{1,0.3,0}}%
      \expandafter\def\csname LT8\endcsname{\color[rgb]{0.5,0.5,0.5}}%
    \else
      \def\colorrgb#1{\color{black}}%
      \def\colorgray#1{\color[gray]{#1}}%
      \expandafter\def\csname LTw\endcsname{\color{white}}%
      \expandafter\def\csname LTb\endcsname{\color{black}}%
      \expandafter\def\csname LTa\endcsname{\color{black}}%
      \expandafter\def\csname LT0\endcsname{\color{black}}%
      \expandafter\def\csname LT1\endcsname{\color{black}}%
      \expandafter\def\csname LT2\endcsname{\color{black}}%
      \expandafter\def\csname LT3\endcsname{\color{black}}%
      \expandafter\def\csname LT4\endcsname{\color{black}}%
      \expandafter\def\csname LT5\endcsname{\color{black}}%
      \expandafter\def\csname LT6\endcsname{\color{black}}%
      \expandafter\def\csname LT7\endcsname{\color{black}}%
      \expandafter\def\csname LT8\endcsname{\color{black}}%
    \fi
  \fi
    \setlength{\unitlength}{0.0500bp}%
    \ifx\gptboxheight\undefined%
      \newlength{\gptboxheight}%
      \newlength{\gptboxwidth}%
      \newsavebox{\gptboxtext}%
    \fi%
    \setlength{\fboxrule}{0.5pt}%
    \setlength{\fboxsep}{1pt}%
\begin{picture}(10800.00,5400.00)%
    \gplgaddtomacro\gplbacktext{%
      \csname LTb\endcsname%
      \put(594,704){\makebox(0,0)[r]{\strut{}$0$}}%
      \put(594,2181){\makebox(0,0)[r]{\strut{}$0.3$}}%
      \put(594,3658){\makebox(0,0)[r]{\strut{}$0.6$}}%
      \put(594,5135){\makebox(0,0)[r]{\strut{}$0.9$}}%
      \put(787,484){\makebox(0,0){\strut{}0}}%
      \put(5565,484){\makebox(0,0){\strut{}$\nicefrac{\pi}{2}$}}%
      \put(10342,484){\makebox(0,0){\strut{}$\pi$}}%
      \put(4956,3166){\makebox(0,0)[l]{\strut{}$\norm{\nabla(u-u_h)}_{L^2(\Omega)}$}}%
    }%
    \gplgaddtomacro\gplfronttext{%
      \csname LTb\endcsname%
      \put(5564,154){\makebox(0,0){\strut{}$\alpha$}}%
      \csname LTb\endcsname%
      \put(3754,4962){\makebox(0,0)[r]{\strut{}$h = \nicefrac{1}{16}$}}%
      \csname LTb\endcsname%
      \put(5665,4962){\makebox(0,0)[r]{\strut{}$h = \nicefrac{1}{32}$}}%
      \csname LTb\endcsname%
      \put(7576,4962){\makebox(0,0)[r]{\strut{}$h = \nicefrac{1}{64}$}}%
    }%
    \gplbacktext
    \put(0,0){\includegraphics{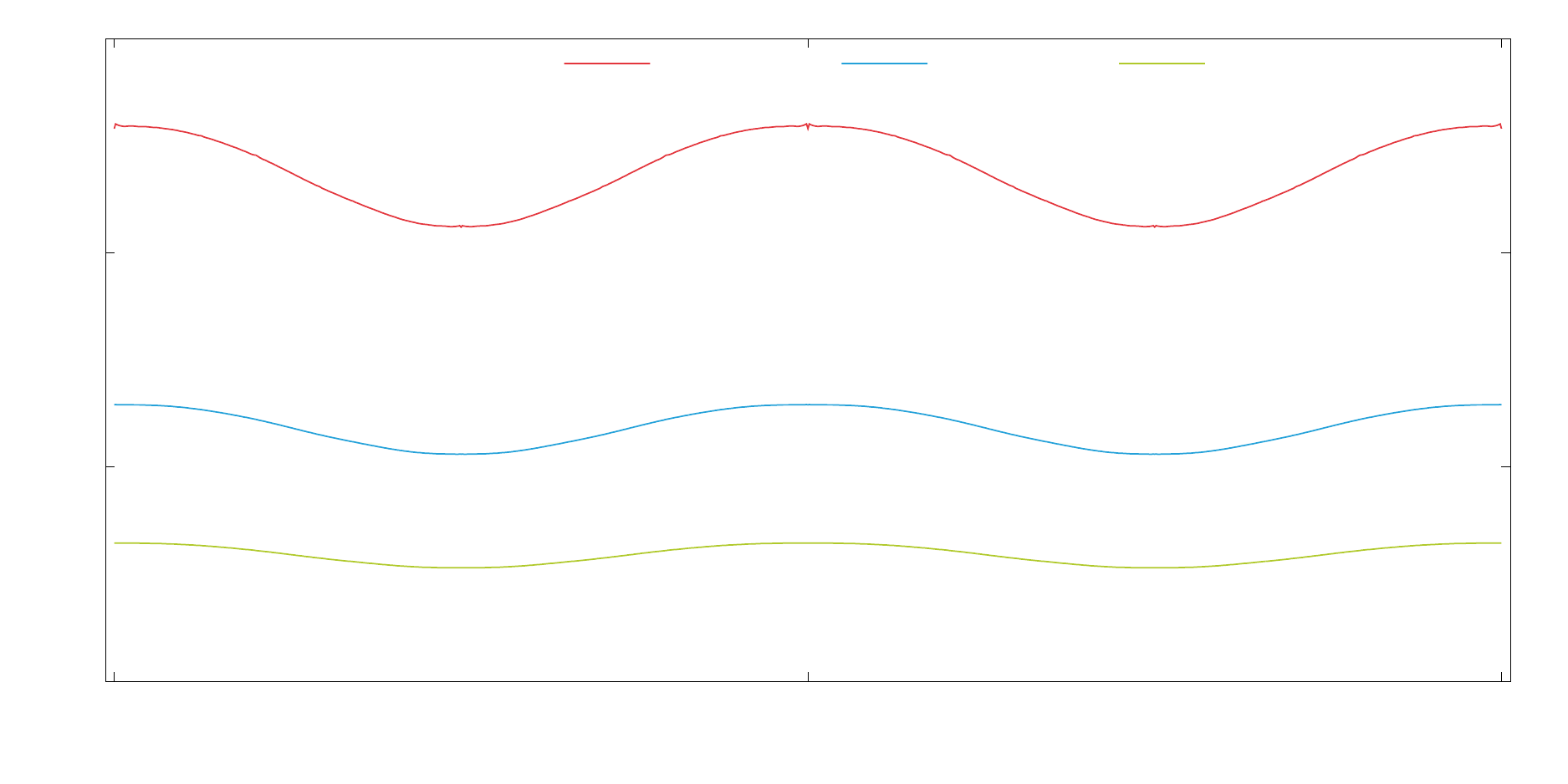}}%
    \gplfronttext
  \end{picture}%
\endgroup

%% file: results/tilted_interface/plot_H1_error_eps_cross.tex
\begingroup
  \makeatletter
  \providecommand\color[2][]{%
    \GenericError{(gnuplot) \space\space\space\@spaces}{%
      Package color not loaded in conjunction with
      terminal option `colourtext'%
    }{See the gnuplot documentation for explanation.%
    }{Either use 'blacktext' in gnuplot or load the package
      color.sty in LaTeX.}%
    \renewcommand\color[2][]{}%
  }%
  \providecommand\includegraphics[2][]{%
    \GenericError{(gnuplot) \space\space\space\@spaces}{%
      Package graphicx or graphics not loaded%
    }{See the gnuplot documentation for explanation.%
    }{The gnuplot epslatex terminal needs graphicx.sty or graphics.sty.}%
    \renewcommand\includegraphics[2][]{}%
  }%
  \providecommand\rotatebox[2]{#2}%
  \@ifundefined{ifGPcolor}{%
    \newif\ifGPcolor
    \GPcolortrue
  }{}%
  \@ifundefined{ifGPblacktext}{%
    \newif\ifGPblacktext
    \GPblacktexttrue
  }{}%
  \let\gplgaddtomacro\g@addto@macro
  \gdef\gplbacktext{}%
  \gdef\gplfronttext{}%
  \makeatother
  \ifGPblacktext
    \def\colorrgb#1{}%
    \def\colorgray#1{}%
  \else
    \ifGPcolor
      \def\colorrgb#1{\color[rgb]{#1}}%
      \def\colorgray#1{\color[gray]{#1}}%
      \expandafter\def\csname LTw\endcsname{\color{white}}%
      \expandafter\def\csname LTb\endcsname{\color{black}}%
      \expandafter\def\csname LTa\endcsname{\color{black}}%
      \expandafter\def\csname LT0\endcsname{\color[rgb]{1,0,0}}%
      \expandafter\def\csname LT1\endcsname{\color[rgb]{0,1,0}}%
      \expandafter\def\csname LT2\endcsname{\color[rgb]{0,0,1}}%
      \expandafter\def\csname LT3\endcsname{\color[rgb]{1,0,1}}%
      \expandafter\def\csname LT4\endcsname{\color[rgb]{0,1,1}}%
      \expandafter\def\csname LT5\endcsname{\color[rgb]{1,1,0}}%
      \expandafter\def\csname LT6\endcsname{\color[rgb]{0,0,0}}%
      \expandafter\def\csname LT7\endcsname{\color[rgb]{1,0.3,0}}%
      \expandafter\def\csname LT8\endcsname{\color[rgb]{0.5,0.5,0.5}}%
    \else
      \def\colorrgb#1{\color{black}}%
      \def\colorgray#1{\color[gray]{#1}}%
      \expandafter\def\csname LTw\endcsname{\color{white}}%
      \expandafter\def\csname LTb\endcsname{\color{black}}%
      \expandafter\def\csname LTa\endcsname{\color{black}}%
      \expandafter\def\csname LT0\endcsname{\color{black}}%
      \expandafter\def\csname LT1\endcsname{\color{black}}%
      \expandafter\def\csname LT2\endcsname{\color{black}}%
      \expandafter\def\csname LT3\endcsname{\color{black}}%
      \expandafter\def\csname LT4\endcsname{\color{black}}%
      \expandafter\def\csname LT5\endcsname{\color{black}}%
      \expandafter\def\csname LT6\endcsname{\color{black}}%
      \expandafter\def\csname LT7\endcsname{\color{black}}%
      \expandafter\def\csname LT8\endcsname{\color{black}}%
    \fi
  \fi
    \setlength{\unitlength}{0.0500bp}%
    \ifx\gptboxheight\undefined%
      \newlength{\gptboxheight}%
      \newlength{\gptboxwidth}%
      \newsavebox{\gptboxtext}%
    \fi%
    \setlength{\fboxrule}{0.5pt}%
    \setlength{\fboxsep}{1pt}%
\begin{picture}(10800.00,5400.00)%
    \gplgaddtomacro\gplbacktext{%
      \csname LTb\endcsname%
      \put(594,704){\makebox(0,0)[r]{\strut{}$0$}}%
      \put(594,2181){\makebox(0,0)[r]{\strut{}$0.3$}}%
      \put(594,3658){\makebox(0,0)[r]{\strut{}$0.6$}}%
      \put(594,5135){\makebox(0,0)[r]{\strut{}$0.9$}}%
      \put(787,484){\makebox(0,0){\strut{}0}}%
      \put(5565,484){\makebox(0,0){\strut{}$\nicefrac{\pi}{2}$}}%
      \put(10342,484){\makebox(0,0){\strut{}$\pi$}}%
      \put(4956,3166){\makebox(0,0)[l]{\strut{}$\norm{\nabla(u-u_h)}_{L^2(\Omega)}$}}%
    }%
    \gplgaddtomacro\gplfronttext{%
      \csname LTb\endcsname%
      \put(5564,154){\makebox(0,0){\strut{}$\alpha$}}%
      \csname LTb\endcsname%
      \put(3754,4962){\makebox(0,0)[r]{\strut{}$h = \nicefrac{1}{16}$}}%
      \csname LTb\endcsname%
      \put(5665,4962){\makebox(0,0)[r]{\strut{}$h = \nicefrac{1}{32}$}}%
      \csname LTb\endcsname%
      \put(7576,4962){\makebox(0,0)[r]{\strut{}$h = \nicefrac{1}{64}$}}%
    }%
    \gplbacktext
    \put(0,0){\includegraphics{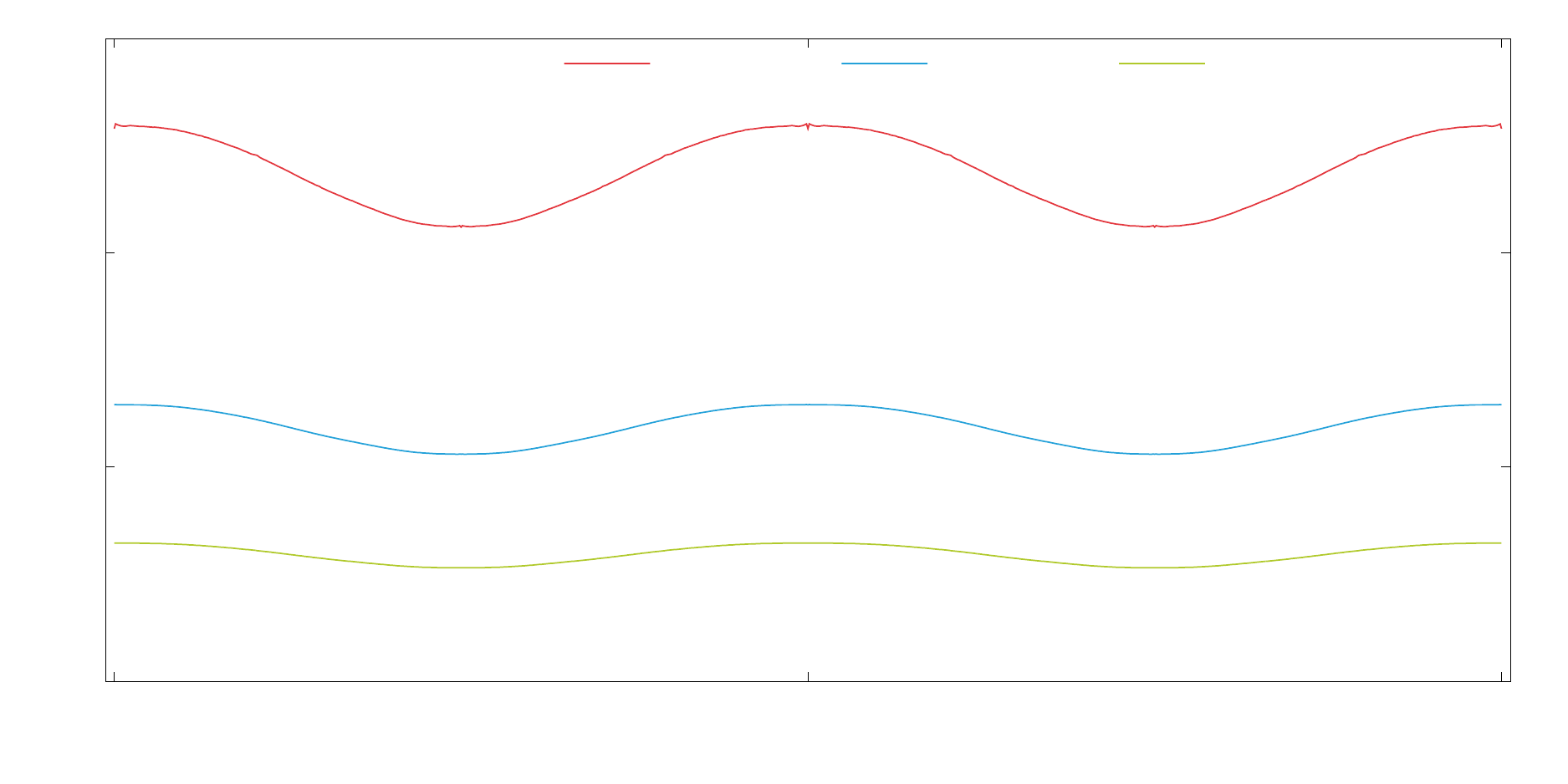}}%
    \gplfronttext
  \end{picture}%
\endgroup

%% file: results/tilted_interface/plot_H1_error_eps_cross_prod.tex
\begingroup
  \makeatletter
  \providecommand\color[2][]{%
    \GenericError{(gnuplot) \space\space\space\@spaces}{%
      Package color not loaded in conjunction with
      terminal option `colourtext'%
    }{See the gnuplot documentation for explanation.%
    }{Either use 'blacktext' in gnuplot or load the package
      color.sty in LaTeX.}%
    \renewcommand\color[2][]{}%
  }%
  \providecommand\includegraphics[2][]{%
    \GenericError{(gnuplot) \space\space\space\@spaces}{%
      Package graphicx or graphics not loaded%
    }{See the gnuplot documentation for explanation.%
    }{The gnuplot epslatex terminal needs graphicx.sty or graphics.sty.}%
    \renewcommand\includegraphics[2][]{}%
  }%
  \providecommand\rotatebox[2]{#2}%
  \@ifundefined{ifGPcolor}{%
    \newif\ifGPcolor
    \GPcolortrue
  }{}%
  \@ifundefined{ifGPblacktext}{%
    \newif\ifGPblacktext
    \GPblacktexttrue
  }{}%
  \let\gplgaddtomacro\g@addto@macro
  \gdef\gplbacktext{}%
  \gdef\gplfronttext{}%
  \makeatother
  \ifGPblacktext
    \def\colorrgb#1{}%
    \def\colorgray#1{}%
  \else
    \ifGPcolor
      \def\colorrgb#1{\color[rgb]{#1}}%
      \def\colorgray#1{\color[gray]{#1}}%
      \expandafter\def\csname LTw\endcsname{\color{white}}%
      \expandafter\def\csname LTb\endcsname{\color{black}}%
      \expandafter\def\csname LTa\endcsname{\color{black}}%
      \expandafter\def\csname LT0\endcsname{\color[rgb]{1,0,0}}%
      \expandafter\def\csname LT1\endcsname{\color[rgb]{0,1,0}}%
      \expandafter\def\csname LT2\endcsname{\color[rgb]{0,0,1}}%
      \expandafter\def\csname LT3\endcsname{\color[rgb]{1,0,1}}%
      \expandafter\def\csname LT4\endcsname{\color[rgb]{0,1,1}}%
      \expandafter\def\csname LT5\endcsname{\color[rgb]{1,1,0}}%
      \expandafter\def\csname LT6\endcsname{\color[rgb]{0,0,0}}%
      \expandafter\def\csname LT7\endcsname{\color[rgb]{1,0.3,0}}%
      \expandafter\def\csname LT8\endcsname{\color[rgb]{0.5,0.5,0.5}}%
    \else
      \def\colorrgb#1{\color{black}}%
      \def\colorgray#1{\color[gray]{#1}}%
      \expandafter\def\csname LTw\endcsname{\color{white}}%
      \expandafter\def\csname LTb\endcsname{\color{black}}%
      \expandafter\def\csname LTa\endcsname{\color{black}}%
      \expandafter\def\csname LT0\endcsname{\color{black}}%
      \expandafter\def\csname LT1\endcsname{\color{black}}%
      \expandafter\def\csname LT2\endcsname{\color{black}}%
      \expandafter\def\csname LT3\endcsname{\color{black}}%
      \expandafter\def\csname LT4\endcsname{\color{black}}%
      \expandafter\def\csname LT5\endcsname{\color{black}}%
      \expandafter\def\csname LT6\endcsname{\color{black}}%
      \expandafter\def\csname LT7\endcsname{\color{black}}%
      \expandafter\def\csname LT8\endcsname{\color{black}}%
    \fi
  \fi
    \setlength{\unitlength}{0.0500bp}%
    \ifx\gptboxheight\undefined%
      \newlength{\gptboxheight}%
      \newlength{\gptboxwidth}%
      \newsavebox{\gptboxtext}%
    \fi%
    \setlength{\fboxrule}{0.5pt}%
    \setlength{\fboxsep}{1pt}%
\begin{picture}(10800.00,5400.00)%
    \gplgaddtomacro\gplbacktext{%
      \csname LTb\endcsname%
      \put(594,704){\makebox(0,0)[r]{\strut{}$0$}}%
      \put(594,2181){\makebox(0,0)[r]{\strut{}$0.3$}}%
      \put(594,3658){\makebox(0,0)[r]{\strut{}$0.6$}}%
      \put(594,5135){\makebox(0,0)[r]{\strut{}$0.9$}}%
      \put(787,484){\makebox(0,0){\strut{}0}}%
      \put(5565,484){\makebox(0,0){\strut{}$\nicefrac{\pi}{2}$}}%
      \put(10342,484){\makebox(0,0){\strut{}$\pi$}}%
      \put(4956,3166){\makebox(0,0)[l]{\strut{}$\norm{\nabla(u-u_h)}_{L^2(\Omega)}$}}%
    }%
    \gplgaddtomacro\gplfronttext{%
      \csname LTb\endcsname%
      \put(5564,154){\makebox(0,0){\strut{}$\alpha$}}%
      \csname LTb\endcsname%
      \put(3754,4962){\makebox(0,0)[r]{\strut{}$h = \nicefrac{1}{16}$}}%
      \csname LTb\endcsname%
      \put(5665,4962){\makebox(0,0)[r]{\strut{}$h = \nicefrac{1}{32}$}}%
      \csname LTb\endcsname%
      \put(7576,4962){\makebox(0,0)[r]{\strut{}$h = \nicefrac{1}{64}$}}%
    }%
    \gplbacktext
    \put(0,0){\includegraphics{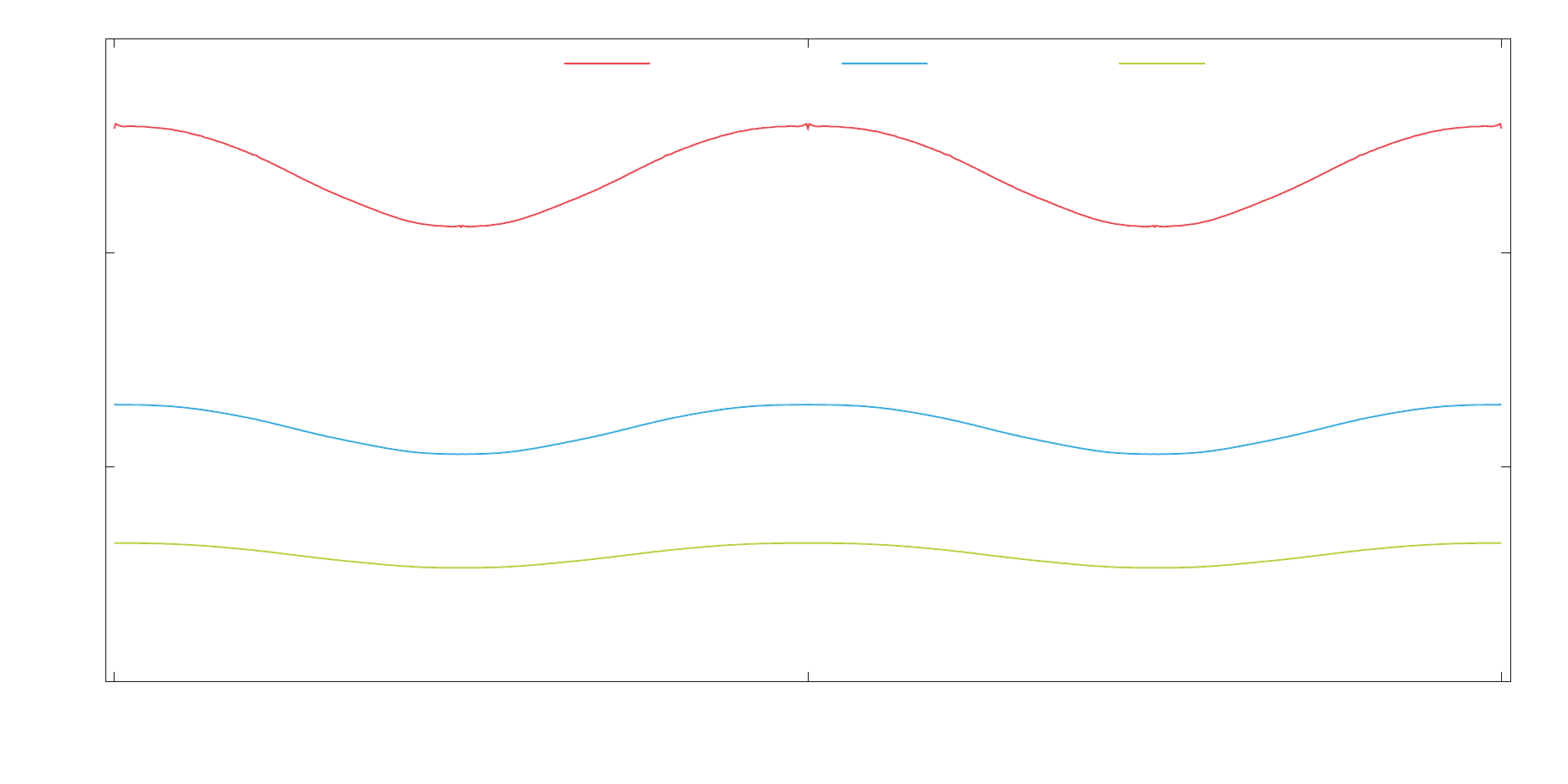}}%
    \gplfronttext
  \end{picture}%
\endgroup